\documentclass{elsarticle}
\usepackage[latin9]{inputenc}
\usepackage{amsthm}
\usepackage{amsmath}
\usepackage{amssymb}
\usepackage{graphicx}
\usepackage{esint}
\PassOptionsToPackage{normalem}{ulem}
\usepackage{ulem}

\makeatletter
  \theoremstyle{definition}
  \newtheorem{defn}{\protect\definitionname}
  \theoremstyle{remark}
  \newtheorem{rem}{\protect\remarkname}
  \theoremstyle{plain}
  \newtheorem{thm}{\protect\theoremname}
 \ifx\proof\undefined\
   \newenvironment{proof}[1][\proofname]{\par
     \normalfont\topsep6\p@\@plus6\p@\relax
     \trivlist
     \itemindent\parindent
     \item[\hskip\labelsep
           \scshape
       #1]\ignorespaces
   }{%
     \endtrivlist\@endpefalse
   }
   \providecommand{\proofname}{Proof}
 \fi
  \theoremstyle{plain}
  \newtheorem{lem}{\protect\lemmaname}
  \theoremstyle{plain}
  \newtheorem{prop}{\protect\propositionname}

\usepackage{calc}\usepackage{amsthm}\usepackage{amstext}\usepackage{subfigure}\usepackage{enumerate}\usepackage{epstopdf}

\theoremstyle{definition}

\providecommand{\theoremname}{Theorem}

\makeatother

\providecommand{\definitionname}{Definition}
\providecommand{\lemmaname}{Lemma}
\providecommand{\propositionname}{Proposition}
\providecommand{\remarkname}{Remark}
\providecommand{\theoremname}{Theorem}

\begin{document}

\title{Hyperbolic and Elliptic Transport Barriers in \\
Three-Dimensional Unsteady Flows}

\author{Daniel Blazevski, George Haller}

\address{Institute for Mechanical Systems, Department of Mechanical and Process
Engineering, }

\address{ETH Zurich, Tannenstrasse 3, Zurich, Switzerland}
\begin{abstract}
We develop a general theory of transport barriers for three-dimensional
unsteady flows with arbitrary time-dependence. The barriers are obtained
as two-dimensional Lagrangian Coherent Structures (LCSs) that create
locally maximal deformation. Along hyperbolic LCSs, this deformation
is induced by locally maximal normal repulsion or attraction. Along
shear LCSs, the deformation is created by locally maximal tangential
shear. Hyperbolic LCSs, therefore, play the role of generalized stable
and unstable manifolds, while closed shear LCSs (elliptic LCSs) act
as generalized KAM tori or KAM-type cylinders. All these barriers
can be computed from our theory as explicitly parametrized surfaces.
We illustrate our results by visualizing two-dimensional hyperbolic
and elliptic barriers in steady and unsteady versions of the ABC flow. 
\end{abstract}
\maketitle

\section{Introduction}

Detecting transport barriers is important in a number of areas, including
geophysical flows \citep{Weiss}, plasma fusion \citep{Plasma}, reactive
flows \citep{Reactive}, and molecular dynamics \citep{toda05}. For
steady and temporally periodic flow models in these areas, classical
dynamical systems theory identifies key invariant manifolds acting
as phase space barriers \citep{Meiss}. Even in this well-understood
setting, however, only specific examples of \emph{de facto} barriers
have been identified. Indeed, even for steady flows, no general approach
to defining and locating multi-dimensional transport barriers has
been available. A commonly used informal definition of barriers as
surfaces with zero transverse flux is easily seen to be inadequate.
Indeed, \emph{any} randomly chosen surface of trajectories (material
surface) admits zero normal flux \citep{geo_theory}.

A number of heuristic flow diagnostics have nevertheless been employed
to infer transport barriers indirectly, mostly targeting two-dimensional
flows (see \citep{boffetta01}and \citep{peacock10} for reviews).
These diagnostics can be highly effective for simple flows, especially
if one tunes their threshold parameters to match expectations. However,
when it comes to discovering more complex flows, the lack of an exact
mathematical foundation renders these approaches problematic, manifested
by frame-dependence, false positives and false negatives \citep{basdevant94,jeong95,haller05}. 

For exploration, decision making and forecasting, one would ideally
require a general approach with solid mathematical foundations. Such
a general approach, the geodesic theory of transport barriers, has
recently emerged for two-dimensional unsteady flows \citep{geo_theory,Agulhas,geo_theory-1}.
Specifically, \citet{geo_theory} constructs transport barriers as
curves most closely shadowed by least-stretching geodesics of the
Cauchy--Green strain tensor derived from the flow map. As a further
improvement, \citet{Agulhas,geo_theory-1} constructs transport barriers
as distinguished geodesics of the Green--Lagrange strain tensors derived
from the flow map. The objective of the present work is to extend
these ideas to unsteady flows in three dimensions.

We are unaware of other approaches that would directly target transport
barrier surfaces in multi-dimensional unsteady flows. Scalar fields
associated with the flow map such as the finite-time Lyapunov exponents
(FTLE) and finite-size Lyapunov exponents have been used as indicators
of hyperbolic coherent structures \citep{Haller_dist,Shadden_nD,FSLE_3D, tallapragada, George_turb_3d, Flatland}.
Under certain conditions, select ridges of these fields can be rigorously
related to repelling and attracting transport barriers 
\cite{var_theory, mo_var_err, Mo_var_theory, Karrasch_comment}. 

Related methods exist for multi-dimensional coherence detection, including
an ergodicity-based approach to visualizing elliptic structures in
steady and temporally periodic flows \citep{Mezic3D}, as well as
a probabilistic approach to locating almost invariant sets in phase
space \citep{FR}.\textbf{ }Both methods seek coherent domains (ergodic
components or almost invariant sets) via a modal truncation of an
infinite-dimensional operator (the Koopman or the Peron--Frobenius
operator, respectively). This process then yields scalar fields whose
topology is expected to reveal coherent sets. Specifically, in \textbf{\citep{Mezic3D}},
low-index eigenfunctions are proposed as indicators of dynamically
distinct regions of the phase space. In \textbf{\citep{FR}}, the
left and right eigenfunctions of the second largest singular value
of the Peron--Frobenius operator are thresholded to maximize the coherence
ratio of pair of sets.

Our focus here is a direct variational construction of transport barriers
as parametrized surfaces. Instead of starting with a particular mathematical
quantity and arguing for its relevance in barrier detection, we start
with a physically motivated question: What objective property makes
transport barriers observable in physical and numerical experiments? 

We put forward the same answer that has been well-tested in two-dimensional
flows. Specifically, for a time-evolving surface $\mathcal{M}(t)$
to be an observed transport barrier, the following two properties
should hold:
\begin{description}
\item [{T1}] $\mathcal{M}(t)$ must be a \emph{material surface}, i.e.,
a two-dimensional invariant manifold in the extended phase space of
positions and time. This implies that barriers locally divide the
phase space and have zero flux between their two sides.
\item [{T2}] $\mathcal{M}(t)$ must impose \emph{locally extreme deformation}
on nearby sets of initial conditions. This is achieved either by locally
maximal normal repulsion or attraction (hyperbolic barrier), or locally
maximal tangential shear (shear barrier). 
\end{description}
Properties T1-T2 provide an extension of the concept of a multi-dimensional
Lagrangian Coherent Structure (LCSs) from the purely hyperbolic case
treated in \citep{var_theory} to the general case. Solving the extremum
problem described in T2 leads to vector fields of admissible unit
normals for hyperbolic and shear barriers. It turns out that surfaces
orthogonal to these admissible normal fields can only exist at locations
where the helicity of these normal fields vanishes. Strain- and shear-helicity
generically vanish on computable two-dimensional surfaces; these zero
sets necessarily contain the transport barriers we seek.

The intersection of transport barriers with select two-dimensional
reference surfaces turns out to satisfy ordinary differential equations
(strain and shear ODEs). These ODEs can be solved numerically, yielding
parametrized \emph{reduced strainlines} and \emph{reduced shearlines}
on the reference surfaces. Open reduced shearlines of zero helicity
signal generalized jets (\emph{parabolic barriers}), while closed
reduced shearlines mark invariant tubes or invariant tori (\emph{elliptic
barriers}). Extracting such parametrized curves over a parametrized
family of reference surfaces leads to explicitly parametrized two-dimensional
transport barriers. 

This construction applies to any three-dimensional flow with general
time-dependence, and uncovers key barriers that shape tracer patterns
over a finite time of observation. This time can be arbitrarily short
or long: our approach, by construction, will locate barriers that
best explain tracer patterns developing over the observational period
chosen. Over longer time intervals, the same approach yields increasingly
accurate approximations for classic hyperbolic and elliptic invariant
manifolds, should those exist in the given flow. 

We first illustrate these results on the steady and time-periodic
ABC flows, which have well-defined steady and time-periodic transport
barriers given by invariant manifolds. Even in these flows, we obtain
new, explicit barrier surfaces that were previously only inferred
from numerical images. Next, we consider a chaotically forced version
of the ABC flow over a finite time interval. For this flow, transport
barriers can only be constructed as temporally aperiodic material
surfaces in the extended phase space. Remarkably, we obtain that select
hyperbolic barriers and torus-type shear barriers continue to exist
even in this fully aperiodic setting. The latter tori tori deform
aperiodically in time, yet continue to provide sharp boundaries for
coherent Lagrangian vortices. Indeed, they exhibit minimal deformation
while nearby material elements in their exteriors stretch exponentially.

\section{Set-up and notation}

Consider the dynamical system 
\begin{equation}
\dot{x}\mathbf{=}v\mathbf{(}x,t),\qquad x\in D\subset\mathbb{R}^{3},\qquad t\in\left[t_{0},t_{0}+T\right],\label{gen_vector_field}
\end{equation}
with a smooth vector field $v(x,t)$ defined over a time interval
of length $T$, for locations $x$ in a compact set $D$. We assume
a finite time time interval in \eqref{gen_vector_field} since data
sets obtained from physical measurements or numerical methods are
only known over such intervals. Moreover, coherent structures of physical
interest are typically transient in nature (eddies, hurricanes, etc). 

A trajectory of \eqref{gen_vector_field} starting from $x_{0}$ at
time $t_{0}$ is denoted by $x(t;t_{0},x_{0})$. The flow map of \eqref{gen_vector_field}
is then defined as 
\begin{equation}
F_{t_{0}}^{t}(x_{0})\colon\,\, x_{0}\mapsto x(t;t_{0},x_{0}),\label{advection_map}
\end{equation}
which is as smooth in $x_{0}$ as the function $v(x,t)$ in $x$.

Assuming that $v$ is of $C^{r}$ with $r\geq1$, the flow gradient
$\nabla F_{t_{0}}^{t}(x_{0})$ can be computed. This linear mapping
advects small initial perturbations $\xi_{0}$ to $x_{0}$ along the
trajectory $x(t,t_{0},x_{0})$ to the evolved perturbation $\xi_{t}=\nabla F_{t_{0}}^{t}(x_{0})\xi_{0}$.
Note that 
\begin{equation}
\left|\xi_{t}\right|^{2}=\left<\xi_{0},C_{t_{0}}^{t}(x_{0})\xi_{0}\right>,\label{eq:xi^2}
\end{equation}
where $C_{t_{0}}^{t}(x_{0}):=\left[\nabla F_{t_{0}}^{t}(x_{0})\right]^{*}\nabla F_{t_{0}}^{t}(x_{0})$
denotes the \emph{Cauchy-Green strain tensor,} and $<\,\cdot\,,\,\cdot\,>$
is the classic Euclidean inner product. 

We will be interested in stationary values of total perturbation growth
\eqref{eq:xi^2} over the time interval $[t_{0},t_{0}+T]$. These
values are precisely the eigenvalues of the symmetric, positive definite
matrix $C_{t_{0}}^{t_{0}+T}(x_{0})$. The eigenvalues $\lambda_{i}(x_{0})$
and their corresponding orthonormal eigenvectors $\xi_{i}(x_{0})$
satisfy 
\begin{equation}
C_{t_{0}}^{t_{0}+T}(x_{0})\xi_{i}(x_{0})=\lambda_{i}(x_{0})\xi_{i}(x_{0}),\qquad\left|\xi_{i}(x_{0})\right|=1,\quad i=1,2,3.
\end{equation}
From now on, we restrict our discussion to an open set $U$ of initial
conditions where the eigenvalues of $C_{t_{0}}^{t_{0}+T}(x_{0})$
are disjoint:

\[
U=\left\{ x_{0}\in D\,:\,\,0<\lambda_{1}(x_{0})<\lambda_{2}(x_{0})<\lambda_{3}(x_{0})\right\} .
\]

\section{Three-dimensional transport barriers }

Here we give a formal definition of transport barriers building on
the properties T1-T2 described in the Introduction. According to T1,
a time-dependent transport barrier $\mathcal{M}(t)$ must be a material
surface, i.e., an invariant manifold in the extended phase space of
the variables $(x,t)$. This necessarily implies 
\begin{equation}
\mathcal{M}(t)=F_{t_{0}}^{t}\left(\mathcal{M}(t_{0})\right),
\end{equation}
for any time $t\in[t_{0},t_{0}+T].$ As long as $\mathcal{M}(t_{0})$
is a smooth surface, so is the surface $\mathcal{M}(t)$ for any fixed
time $t$. The family $\mathcal{M}(t)$ is equally smooth in $t$
by our smoothness assumption on \eqref{gen_vector_field}.

At an initial point $x_{0}\in\mathcal{M}(t_{0})$, let $n_{0}$ denote
a unit normal to $\mathcal{M}(t_{0})$. Then, as discussed in \citep{var_theory},
a smoothly varying unit normal to ${\cal M}(t)$ along the trajectory
$x(t,t_{0},x_{0})$ is given by 
\[
n_{t}(x_{0})=\frac{\left[\nabla F^{t_{0}}_{t}\left(x_{0}\right)\right]^{*}n_{0}}{\left|\left[\nabla F^{t^{0}}_{t}\left(x_{0}\right)\right]^{*}n_{0}\right|}.
\]

For any initial point $x_{0}\in\mathcal{M}(t_{0})$ and initial unit
normal $n_{0}$ to $\mathcal{M}(t_{0})$ at $x_{0}$, we define the
\textit{normal repulsion} $\rho(x_{0,}n_{0})$ of $\mathcal{M}(t)$
along the trajectory $x(t;t_{0},x_{0})$ as the normal component of
the growth of $n_{0}$ under the linearized flow between times $t_{0}$
and $t_{0}+T$ \citep{var_theory}. Specifically, we have 
\[
\rho(x_{0,}n_{0})=\langle n_{t_{0}+T}(x_{0}),\nabla F_{t_{0}}^{t_{0}+T}(x_{0})n_{0}\rangle,
\]
with the geometry illustrated in Fig. \ref{fig:flow_map}. Similarly,
we define the \emph{tangential shear} $\sigma(x_{0},n_{0})$ as the
tangential component of the growth of $n_{0}$ under the linearized
flow along the trajectory $x(t;t_{0},x_{0})$ between times $t_{0}$
and $t_{0}+T$ \citep{geo_theory}. Specifically, we have 
\[
\sigma(x_{0},n_{0})=\left|\nabla F_{t_{0}}^{t_{0}+T}(x_{0})n_{0}-\langle n_{t_{0}+T}(x_{0}),\nabla F_{t_{0}}^{t_{0}+T}(x_{0})n_{0}\rangle n_{t_{0}+T}(x_{0})\right|,
\]
with the geometry shown in Fig. \ref{fig:flow_map}. 

We seek transport barriers as material surfaces that maximize normal
repulsion or tangential shear with respect to perturbations to their
tangent spaces. We do not insist on this maximizing property under
all perturbations to the material surface: we only consider perturbations
to their tangent spaces. This is because we seek a well-defined local
directionality for the transport barrier, while in principle allowing
for it to have a finite thickness. In other words, the barrier may
\emph{a priori} be a thin set of nearby, parallel surfaces.

\begin{figure}
\includegraphics[width=0.95\textwidth]{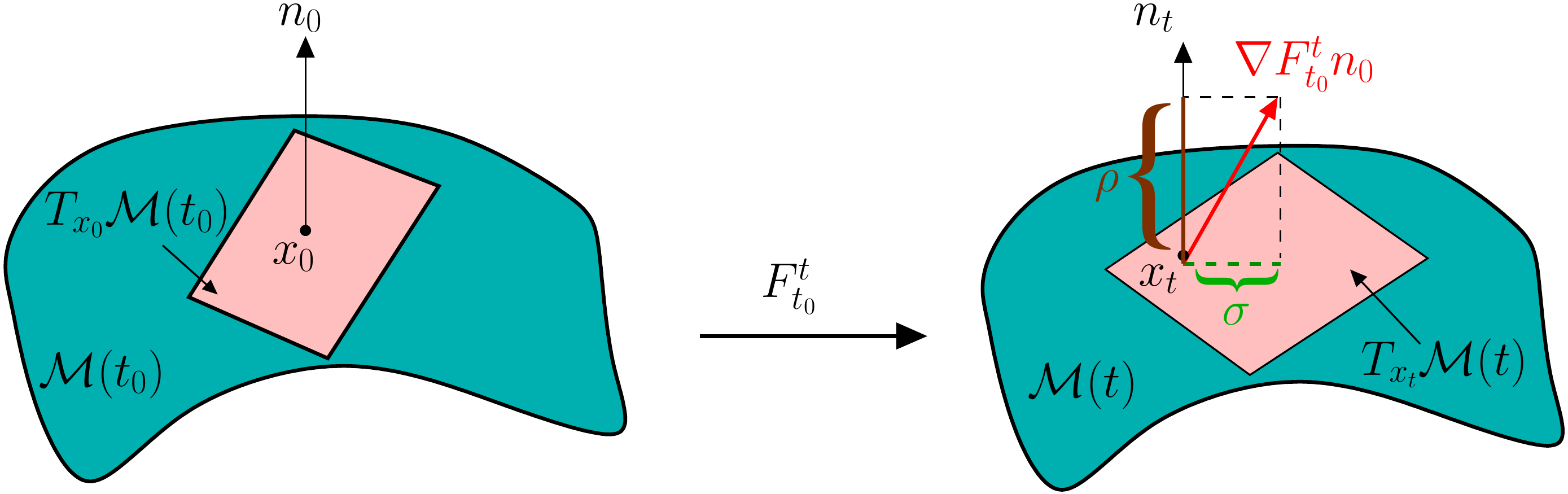}

\caption{The definition of normal repulsion and tangential shear along a material
surface.}

\label{fig:flow_map}
\end{figure}

\begin{defn}
\label{def:barriers}\end{defn}
\begin{description}
\item [{(i)}] A material surface $\mathcal{M}(t)\subset\mathbb{R}^{3}$
is called a \emph{repelling} \textit{hyperbolic LCS} over the interval
$[t_{0},t_{0}+T]$, if for any point $x_{0}\in\mathcal{M}(t_{0})$
and for any other material surface $\hat{\mathcal{M}}(t_{0})$ with
$x_{0}\in\hat{\mathcal{M}}(t_{0})$ and with unit normal $\hat{n}_{0}\not\parallel n_{0}$
at $x_{0}$, we have 
\begin{equation}
\rho(x_{0},\hat{n}_{0})<\rho(x_{0},n_{0}),\qquad\rho(x_{0},n_{0})>1.
\end{equation}

\item [{(ii)}] A material surface $\mathcal{M}(t)\subset\mathbb{R}^{3}$
is called an \emph{attracting} \textit{hyperbolic LCS} over the interval
$[t_{0},t_{0}+T]$, if for any point $x_{0}\in\mathcal{M}(t_{0})$
and for any other material surface $\hat{\mathcal{M}}(t_{0})$ with
$x_{0}\in\hat{\mathcal{M}}(t_{0})$ and with unit normal $\hat{n}_{0}\not\parallel n_{0}$
at $x_{0}$, we have 
\begin{equation}
\rho(x_{0},\hat{n}_{0})>\rho(x_{0},n_{0}),\qquad\qquad\rho(x_{0},n_{0})<1.
\end{equation}

\item [{(iii)}] A material surface $\mathcal{M}(t)\subset\mathbb{R}^{3}$
is called a \textit{shear }\emph{LCS }over the interval $[t_{0},t_{0}+T]$,
if for any point $x_{0}\in\mathcal{M}(t_{0})$ and for for any other
material surface $\hat{\mathcal{M}}(t_{0})$ with $x_{0}\in\hat{\mathcal{M}}(t_{0})$
and with unit normal $\hat{n}_{0}\not\parallel n_{0}$ at $x_{0}$,
we have 
\begin{equation}
\sigma(x_{0},\hat{n}_{0})\leq\sigma(x_{0},n_{0}),
\end{equation}
with $\hat{n}_{0}$ denoting a unit normal to $\hat{\mathcal{M}}(t_{0})$
at the point $x_{0}$ .
\item [{(iv)}] A material surface $\mathcal{M}(t)\subset\mathbb{R}^{3}$
is called a \textit{transport barrier} over the interval $[t_{0},t_{0}+T]$,
if it is either a hyperbolic or a shear LCS over $[t_{0},t].$\end{description}
\begin{rem}
By (i) of Definition 1, hyperbolic LCS exhibit strictly maximal repulsion
relative to all perturbations to their tangent spaces. This is to
guarantee that any potential stretching in directions tangent to the
LCS is smaller than in its normal direction. This condition guarantees
both the observability and the robustness of a repelling hyperbolic
barrier \citep{var_theory}. The same observation applies to attracting
hyperbolic barriers in backward time.
\end{rem}
~~
\begin{rem}
By (ii) of Definition 1, the shear $\sigma(x_{0},n_{0})$ along a
shear barrier $\mathcal{M}(t_{0})$ is not required to be strictly
maximal among all possible perturbations to the tangent space of $\mathcal{M}(t_{0})$.
As it turns out below, there are always two choices of the normal
$n_{0}$ yielding the exact same, locally largest value of $\sigma(x_{0},n_{0})$
at any point $x_{0}$. Shear barriers, therefore, exhibit strictly
maximal shear only with respect to small enough perturbations of their
normals. There will always exist a unique, finite perturbation to
their normal yielding exactly the same locally maximal shear at $x_{0}$.
The two shear extrema at $x_{0}$ represent maximal shear with two
different signs, which is disguised by the absolute value appearing
in the definition of $\sigma(x_{0},n_{0})$. 
\end{rem}
The following theorem shows that for a material surface $\mathcal{M}(t)$
to be a transport barrier over $[t_{0},t_{0}+T]$, its initial position
must be orthogonal to a direction characterizing maximal repulsion
or maximal shear. We use the notation $T_{x_{0}}\mathcal{M}(t_{0})$
for the tangent space of $\mathcal{M}(t_{0})$ at a point $x_{0}$. 
\begin{thm}
{[}Existence of transport barriers{]} Let $\mathcal{M}(t)\subset\mathbb{R}^{3}$
be a material surface over the time interval $[t_{0},t_{0}+T]$. Then\end{thm}
\begin{description}
\item [{(i)}] $\mathcal{M}(t)$ is a repelling hyperbolic LCS if and only
if $\mathcal{M}(t_{0})\subset U$ and $\xi_{3}(x_{0})\perp T_{x_{0}}\mathcal{M}(t_{0})$
holds for all $x_{0}\in\mathcal{M}(t_{0})$.
\item [{(ii)}] $\mathcal{M}(t)$ is an attracting hyperbolic LCS if and
only if $\mathcal{M}(t_{0})\subset U,$ and $\xi_{1}(x_{0})\perp T_{x_{0}}\mathcal{M}(t_{0})$
holds for all $x_{0}\in\mathcal{M}(t_{0})$.
\item [{(iii)}] $\mathcal{M}(t)$ is a shear LCS if and only if $\mathcal{M}(t_{0})\subset U,$
and $n_{\pm}(x_{0})\perp T_{x_{0}}\mathcal{M}(t_{0})$ holds for all
$x_{0}\in\mathcal{M}(t_{0})$ for one choice of the sign $\pm$ in
the vector field 
\[
n_{\pm}(x_{0})=\sqrt{\frac{\sqrt{\lambda_{1}(x_{0})}}{\sqrt{\lambda_{1}(x_{0})}+\sqrt{\lambda_{n}(x_{0})}}}\xi_{1}(x_{0})\pm\sqrt{\frac{\sqrt{\lambda_{3}(x_{0})}}{\sqrt{\lambda_{1}(x_{0})}+\sqrt{\lambda_{3}(x_{0})}}}\xi_{3}(x_{0}).
\]
\end{description}
\begin{proof}
See \ref{App:proof_main_thm}.\end{proof}
\begin{rem}
The above necessary conditions for hyperbolic LCS have previously
been obtained from slightly different considerations, along with examples
illustrating their meaning in \citep{var_theory,Mo_var_theory,FH13}.
For an explicit example of how shear LCS can be found in three-dimensional,
unsteady parallel shear flows, we refer to \ref{App:par_shear_flow}.
\end{rem}
~~~~
\begin{rem}
Unlike in the two-dimensional case \citep{geo_theory}, the shear
LCSs obtained in (iii) of Theorem 1 generally do not preserve their
surface area under an incompressible flow map $F_{t_{0}}^{t_{0}+T}$
, even though they still preserve their enclosed volume (cf. \ref{App:surface_area}).
This enables their use in detecting material footprints of commonly
observed toroidal vortices, such as growing smoke rings.
\end{rem}
~~~~
\begin{rem}
A related recent paper \citep{THK} shows how quasi-invariant hyperbolic
LCS can be used to compute a specific family of hyperbolic barriers
(normally hyperbolic invariant manifolds) in steady flows of arbitrary
dimension. 

~~~~~
\end{rem}

Theorem 1 requires the initial position $\mathcal{M}(t_{0})$ of a
transport barrier to be orthogonal to $\xi_{3}$ (hyperbolic barrier)
or to $n_{\pm}$ (shear barrier). In general, if a two-dimensional
surface is orthogonal to a three-dimensional vector field $\rho(x)$,
then any local parametrization $p(s_{1},s_{2})\colon U\subset\mathbb{R}^{2}\rightarrow\mathbb{R}^{3}$
of the surface must satisfy the first-order quasi-linear system of
PDEs 
\begin{equation}
\begin{split} & \rho_{1}(p)\partial_{s_{1}}p_{1}+\rho_{2}(p)\partial_{s_{1}}p_{2}+\rho_{3}(p)\partial_{s_{1}}p_{3}=0,\\
 & \rho_{1}(p)\partial_{s_{2}}p_{1}+\rho_{2}(p)\partial_{s_{2}}p_{2}+\rho_{3}(p)\partial_{s_{2}}p_{3}=0.
\end{split}
\label{PDE_orth}
\end{equation}

This system of PDEs will only have a smooth solution through a given
point $x_{0}$ if this point is contained in a transport barrier.
To locate such barrier points, we now give computable necessary conditions
for transport barrier locations. In stating these conditions, we will
use the \emph{helicity} $H_{\rho}(x)$ of a three-dimensional vector
field $\rho(x)$, defined as 
\begin{equation}
H_{\rho}(x)=\left<\nabla\times\rho(x),\rho(x)\right>,
\end{equation}
with $\times$ denoting the cross product. 
\begin{thm}
{[}Necessary condition for transport barriers{]} Let $\mathcal{M}(t)\subset\mathbb{R}^{3}$
be a material surface over the time interval $[t_{0},t_{0}+T]$.\end{thm}
\begin{description}
\item [{(i)}] Suppose that $\mathcal{M}(t)$ is a repelling hyperbolic
LCS. Then at all points $x_{0}\in\mathcal{M}(t_{0}),$ we must have
\begin{equation}
H_{\xi_{3}}(x_{0})=0.\label{frob_cond_str-1}
\end{equation}

\item [{(ii)}] Suppose that $\mathcal{M}(t)$ is an attracting hyperbolic
LCS. Then at all points $x_{0}\in\mathcal{M}(t_{0}),$ we must have
\begin{equation}
H_{\xi_{1}}(x_{0})=0.\label{frob_cond_str-1-1}
\end{equation}

\item [{(iii)}] Suppose that $\mathcal{M}(t)$ is a shear LCS. Consider
the two vector fields 
\[
n_{\pm}(x_{0})=\sqrt{\frac{\sqrt{\lambda_{1}(x_{0})}}{\sqrt{\lambda_{1}(x_{0})}+\sqrt{\lambda_{3}(x_{0})}}}\xi_{1}(x_{0})\pm\sqrt{\frac{\sqrt{\lambda_{3}(x_{0})}}{\sqrt{\lambda_{1}(x_{0})}+\sqrt{\lambda_{3}(x_{0})}}}\xi_{3}(x_{0}).
\]
 Then at all points $x_{0}\in\mathcal{M}(t_{0}),$ we must have
\begin{equation}
H_{n_{\pm}}(x_{0})=0\label{frob_cond_shr-1}
\end{equation}
for one choice of the sign in $\pm$.\end{description}
\begin{proof}
See \ref{App:proof_thm2}.\end{proof}
\begin{rem}
The problem of finding surfaces orthogonal to vector fields is locally
equivalent to finding surfaces tangent to two smooth vector fields.
For the existence of such tangent surfaces, the Frobenius Integrability
Theorem provides a necessary condition. This can be shown equivalent
to the zero helicity conditions described above (cf. \ref{App:proof_thm2}). 
\end{rem}
Theorem 3 provides specific scalar equations of the form \eqref{frob_cond_str-1}
and \eqref{frob_cond_shr-1} that a transport barrier $\mathcal{M}(t_{0})$
must satisfy. Rather than solving these equations numerically, we
locate the intersection curves of all potential transport barriers
with a family of two-dimensional reference surfaces. Out of all these
intersection curves, we then select the ones on which the appropriate
helicity condition in Theorem 2 vanishes. This leads to the following
result: 
\begin{thm}
{[}Necessary condition for intersections of transport barriers with
reference surfaces{]} Let $\mathcal{M}(t)\subset\mathbb{R}^{3}$ be
a material surface over the time interval $[t_{0},t_{0}+T]$. Also,
let $\Pi(s_{1})$ denote a smooth, one-parameter family of two-dimensional
orientable surfaces with smooth normal vector fields $n_{\Pi(s_{1})}(x)$. \end{thm}
\begin{description}
\item [{(i)}] Suppose that $\mathcal{M}(t_{0})$ is a repelling hyperbolic
LCS. Then the intersection curve $\gamma_{s_{1}}=\Pi(s_{1})\cap\mathcal{M}(t_{0})$
is a trajectory of a two-dimensional differential equation on $\Pi(s_{1})$,
given by 
\begin{equation}
\gamma_{s_{1}}^{\prime}(s_{2})=n_{\Pi(s_{1})}(\gamma_{s_{1}}(s_{2}))\times\xi_{3}(\gamma_{s_{1}}(s_{2})).\label{eq:redstrain-ODE}
\end{equation}
This trajectory must also satisfy the condition 
\begin{equation}
H_{\xi_{3}}(\gamma_{s_{1}}(s_{2}))=0.\label{eq:strain-helicity}
\end{equation}

\item [{(ii)}] Suppose that $\mathcal{M}(t_{0})$ is an attracting hyperbolic
LCS. Then the intersection curve $\gamma_{s_{1}}=\Pi(s_{1})\cap\mathcal{M}(t_{0})$
is a trajectory of a two-dimensional differential equation on $\Pi(s_{1})$,
given by 
\begin{equation}
\gamma_{s_{1}}^{\prime}(s_{2})=n_{\Pi(s_{1})}(\gamma_{s_{1}}(s_{2}))\times\xi_{1}(\gamma_{s_{1}}(s_{2})).\label{eq:redstrain-ODE-1}
\end{equation}
This trajectory must also satisfy the condition 
\begin{equation}
H_{\xi_{1}}(\gamma_{s_{1}}(s_{2}))=0.\label{eq:strain-helicity-1}
\end{equation}

\item [{(iii)}] Suppose that $\mathcal{M}(t_{0})$ is a shear LCS. Then
the intersection curve $\gamma_{s_{1}}=\Pi(s_{1})\cap\mathcal{M}(t_{0})$
is a trajectory of a two-dimensional differential equation on $\Pi(s_{1})$,
given by 
\begin{equation}
\gamma_{s_{1}}^{\prime}(s_{2})=n_{\Pi(s_{1})}(\gamma_{s_{1}}(s_{2}))\times n_{\pm}(\gamma_{s_{1}}(s_{2})),\label{eq:redshear-ODE}
\end{equation}
for some choice of the sign in $\pm$. This trajectory must also satisfy
the condition 
\begin{equation}
H_{n_{\pm}}(\gamma_{s_{1}}(s_{2}))=0\label{eq:shear-helicity}
\end{equation}
with the same choice of the sign. If the trajectory $\gamma_{s_{1}}(s_{2})$
is a closed orbit, then $\mathcal{M}(t)$ is an elliptic barrier.\end{description}
\begin{proof}
See \ref{App:proof_thm3}.
\end{proof}
~~
\begin{rem}
Theorem 3 yields a local parametrization $p(s_{1},s_{2})=\gamma_{s_{1}}(s_{2})$
for transport barriers in the form of parametrized families of smooth
curves $\gamma_{s_{1}}(s_{2})$.
\end{rem}

\section{Computation of transport barriers }

\label{section:comp} Theorem 3 provides a practical algorithm for
the computation of transport barriers as different types of LCSs in
three-dimensional flows. The barriers can be reconstructed from their
intersections with a family of orientable hypersurfaces. 

In the simplest case, these hypersurfaces are just two-dimensional
planes. For this case, we summarize below the extraction of hyperbolic
LCS (generalized stable and unstable manifolds) and elliptic LCS (invariant
cylinders and generalized KAM tori). Further hints on the numerical
implementation of these algorithmic steps can be found in Appendix
D.

\subsection{Algorithm for hyperbolic LCSs\label{sub:Algorithm-for-hyperbolic}}
\begin{description}
\item [{H1}] Compute the Cauchy--Green strain tensor $C_{t_{0}}^{t_{0}+T}$
and its dominant eigenvector $\xi_{3}$ on a two-dimensional grid
$\mathcal{G}_{0}$ defined on the reference plane $\Pi(s_{1})$ 
\item [{H2}] Pick a sparser grid $\mathcal{G}_{1}$ of initial conditions.
Obtain \emph{reduced strainlines} $\gamma_{s_{1}}(s_{2})$ as trajectories
of the ODE \eqref{eq:redstrain-ODE}, starting from points of $\mathcal{G}_{1}$
satisfying $\left|H_{\xi_{3}}\right|=\left|\left<\nabla\times\xi_{3},\xi_{3}\right>\right|<\epsilon_{0}$
for some threshold parameter $\epsilon_{0}$. 
\item [{H3}] Integrate such reduced strainlines as long as the running
average of $\left|H_{\xi_{3}}\right|$ stays below $\epsilon_{0}$. 
\item [{H4}] Filter the reduced strainline segments so obtained to find
the ones that approximate the zero sets of $H_{\xi_{3}}$ most closely.
Specifically, if the Hausdorff distance 
\[
d(\gamma_{s_{1}},\tilde{\gamma}_{s_{1}})=\max_{x\in\gamma_{1}}\left(\min_{y\in\tilde{\gamma}_{s_{1}}}\|x-y\|\right)+\max_{x\in\tilde{\gamma}_{s_{1}}}\left(\min_{y\in\gamma_{1}}\|x-y\|\right)
\]
two strainline segments $\gamma_{1}$ and $\tilde{\gamma}_{s_{1}}$
are smaller than a small threshold value $d_{0}$, then discard either
$\gamma_{1}$ or $\tilde{\gamma}_{s_{1}}$.
\item [{H5}] Vary the parameter $s_{1}$ in the definition of the plane
family $\Pi(s_{1})$ to obtain uniform coverage of the domain of interest.
Repeat H1-H4 for each $s_{1}$. Obtain repelling hyperbolic LCSs by
fitting a surface to the parametrized curve family $\gamma_{s_{1}}(s_{2}),$
as described in Remark 4.
\item [{H6}] Replace the eigenvector $\xi_{3}$ with $\xi_{1}$ in steps
H1-H4 to obtain \emph{reduced stretchlines} as trajectories of \eqref{eq:redstrain-ODE-1}.
Construct attracting hyperbolic LCSs following step H5.
\end{description}

\subsection{Algorithm for elliptic LCSs\label{sub:Algorithm-for-closed}}
\begin{description}
\item [{SH1}] Compute the Cauchy--Green strain tensor $C_{t_{0}}^{t_{0}+T}$
and its eigenvectors $\xi_{1}$ and $\xi_{3}$ on a two-dimensional
grid $\mathcal{G}_{0}$ defined on the reference plane $\Pi(s_{1})$ 
\item [{SH3}] Pick a sparser grid $\mathcal{G}_{1}$ of initial conditions.
Launch \emph{reduced shearlines} $\gamma_{s_{1}}(s_{2})$ as trajectories
of the ODE \ref{eq:redshear-ODE}, from points of $\mathcal{G}_{1}$
satisfying $\left|H_{n_{\pm}}\right|=\left|\left<\nabla\times\xi_{3},\xi_{3}\right>\right|<\epsilon_{0}$
for some threshold parameter $\epsilon_{0}$. 
\item [{SH3}] Integrate such reduced shearlines as long as the running
average of $\left|H_{n_{\pm}}\right|$ stays below $\epsilon_{0}$. 
\item [{SH4}] Keep only reduced shearlines that form limit cycles.
\item [{SH5}] Vary the parameter $s_{1}$ in the definition of the plane
family $\Pi(s_{1})$ to obtain uniform coverage of the three-dimensional
domain of interest. Repeat SH1-SH4 for each $s_{1}$. Starting from
a closed shearline $\gamma_{s_{1min}}(s_{2})$ on the plane $\Pi(s_{1min})$
corresponding to the lowest value of the parameter $s_{1}$, obtain
a discretized approximation $\gamma_{s_{1}}(s_{2})$ to a closed shear
barrier by always selecting the closest closed reduced shearline in
the planes $\Pi(s_{1})$ under increasing $s_{1}$. Obtain elliptic
LCS by fitting a surface to the curve-family $\gamma_{s_{1}}(s_{2})$,
as described in Remark 4. 
\end{description}

\section{Examples}

\label{Section:results}

\subsection{Steady ABC flow}

\label{Section:ABC_steady}

As a first example, we consider the steady ABC flow 
\begin{equation}
\begin{split} & \dot{x}=A\sin z+C\cos y,\\
 & \dot{y}=B\sin x+A\cos z,\\
 & \dot{z}=C\sin y+B\cos x,
\end{split}
\label{ABC_steady}
\end{equation}
an exact solution of Euler's equation. We select the parameter values
$A=\sqrt{3}$, $B=\sqrt{2}$ and $C=1.0$. This well-studied set of
parameter values yields the Poincare map shown in Fig. \ref{fig:shear_steady}.
We first use the theory developed here to construct the full two-dimensional
transport barriers suggested by this Poincare map. Because the flow
is steady, the transport barriers we seek are also invariant manifolds
in the phase space, not just in the extended phase space. 

We therefore only need to carry out the computational steps H1-H4
and SH1-SH4 of Section \ref{section:comp} to obtain intersection
curves between barriers and a single reference plane $\Pi$. We then
advect these intersection curves under the flow map to obtain the
full two-dimensional barriers. 

In Sections \ref{Section:ABC_periodic} and \ref{Section:ABC_aperiodic},
we consider temporally periodic and aperiodic versions of \eqref{ABC_steady}
where this simple approach will no longer suffice. The present steady
case is only considered here to provide a consistency check on a well-studied
steady flow.

\subsubsection{Elliptic LCSs in the steady ABC flow}

We first perform the computational steps SH1-SH4 of the previous section
in one of the vortical regions seen in the Poincare map plot of Fig.
\ref{fig:shear_steady}. The lower panels of the same figure show
orbits the Poincare map in blue (color only in the online version),
as well as closed reduced shearlines (green) obtained from the computational
steps SH1-SH4 on the plane
\[
\Pi=\left\{ (x,y,z)\,:\, z=0\right\} 
\]
or two different integration times. In both cases, a uniform grid
$\mathcal{G}_{0}$ of $1000\times1000$ initial points was used to
compute the Cauchy--Green strain tensor. The tolerance parameter in
the computational step SH3 is chosen to be $\epsilon_{0}=10^{-2}$.

While the plots in Fig. \ref{fig:shear_steady} show curves in the
$(x,y)$ plane, the analysis is inherently three-dimensional. Indeed,
computing the vector $n_{\pm}$ in equations (\ref{eq:redshear-ODE})-(\ref{eq:shear-helicity})
requires fully three-dimensional trajectory integration.

\begin{figure}[h!]
\begin{centering}
\includegraphics[scale=0.3]{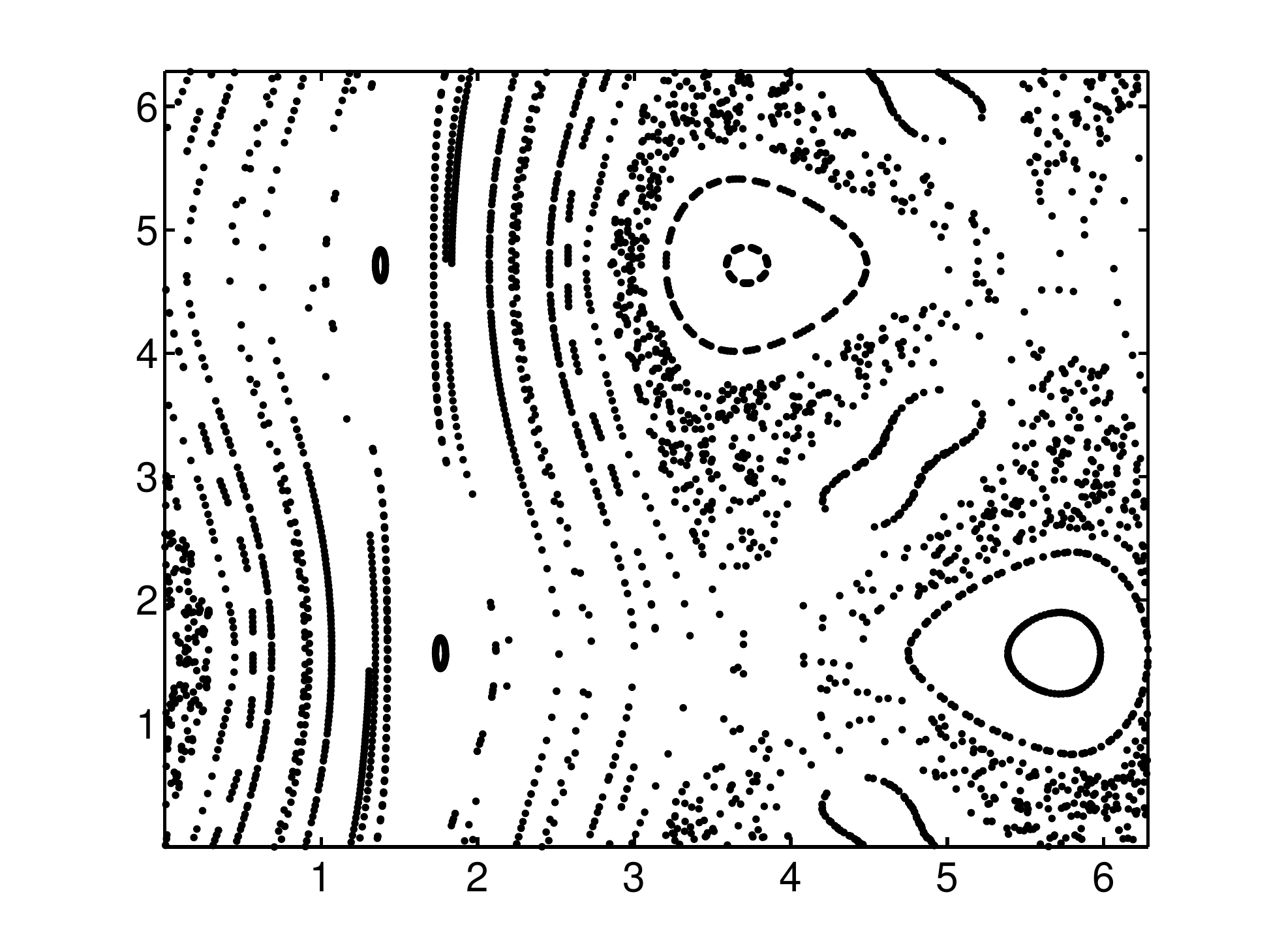}
\par\end{centering}

\centering{}\includegraphics[scale=0.35]{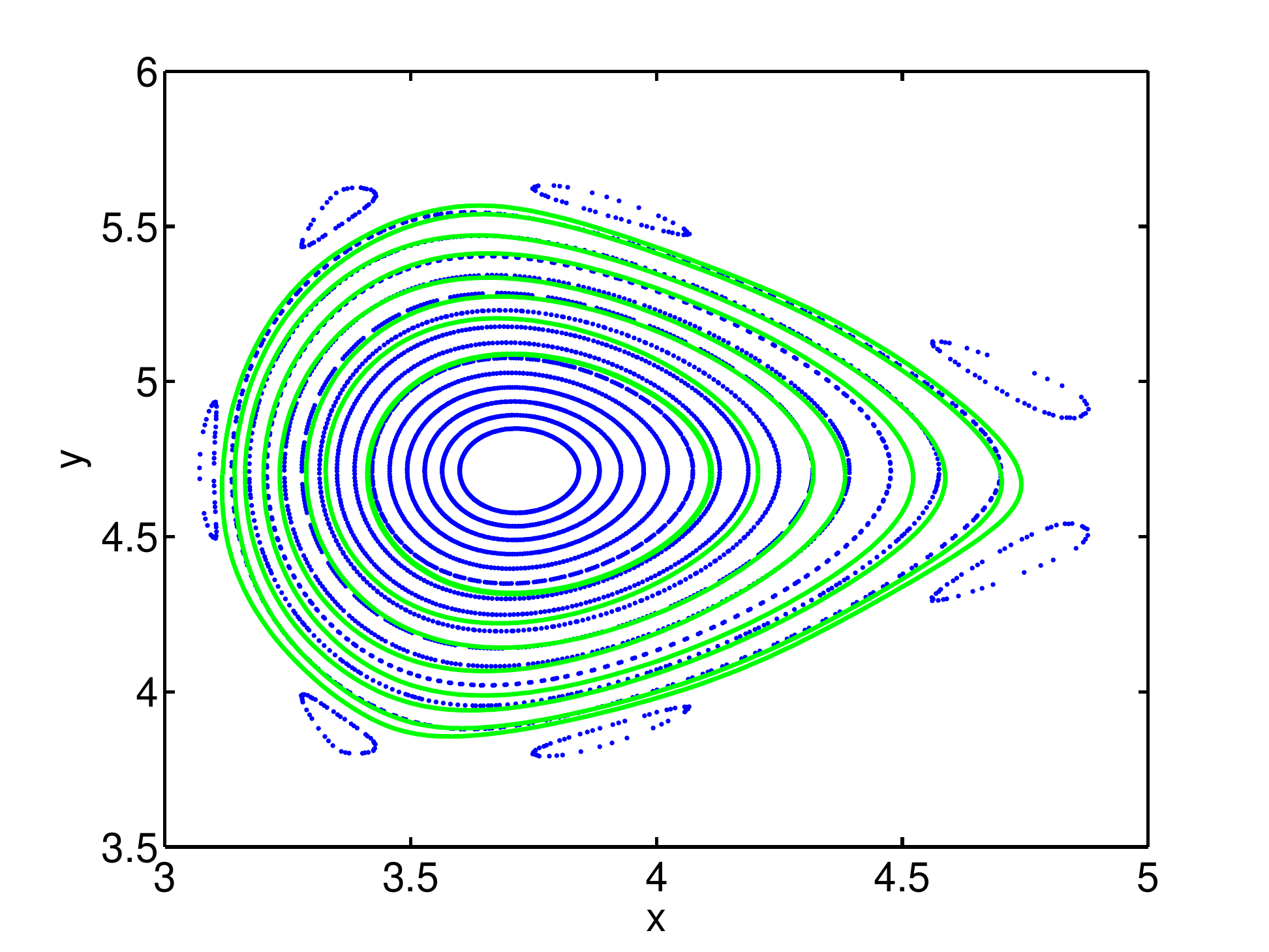}
~~~~~~~~\includegraphics[scale=0.35]{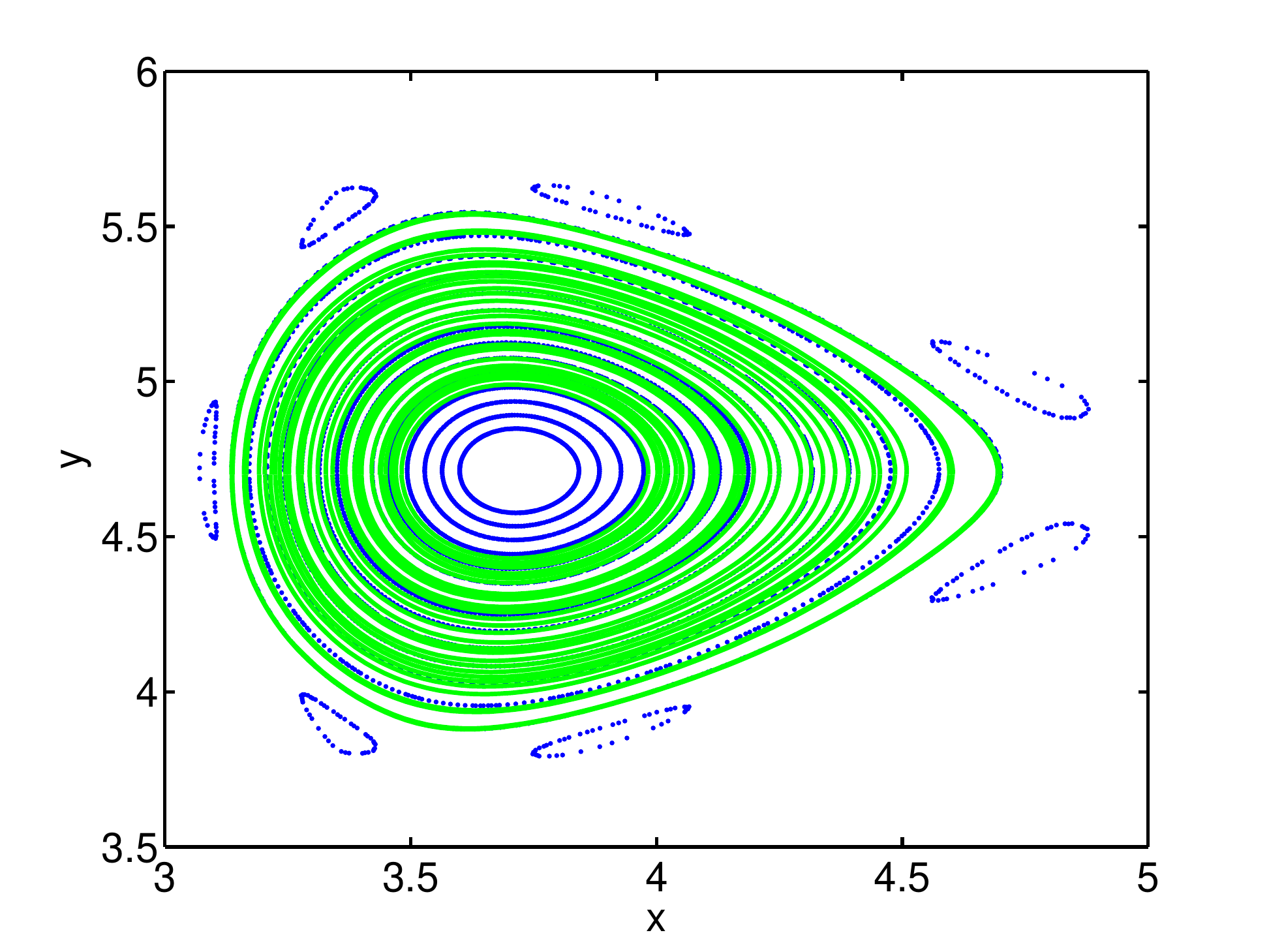}
\caption{\emph{Top}: Poincare map for the steady ABC flow on the $z=0$ plane.
\emph{Bottom:} Closed reduced shearlines on the plane $z=0$ that
approximate invariant tori for the steady ABC flow. The time interval
used in their construction was $[0,40]$ for the left panel and $[0,150]$
for the right panel. }
\label{fig:shear_steady} 
\end{figure}

By the periodic nature of the phase space, the shear LCS obtained
from the advection of closed, reduced shearlines are two-dimensional
tori. To bring out the toroidal nature of these barriers, we introduce
new coordinates with the help of the approximate spatial core $(x_{0}(z),y_{0}(z),z)$
obtained by advecting the vortical center point of Fig. \ref{fig:shear_steady}.
Using this center curve, we introduce the toroidal coordinate system
\begin{equation}
\begin{split} & \bar{x}=\left[x-x_{0}(z)+R_{1}\right]\cos(z),\\
 & \bar{y}=\left[x-x_{0}(z)+R_{1}\right]\sin(z),\\
 & \bar{z}=R_{2}\left[y-y_{0}(z)\right],
\end{split}
\label{eq:trans2}
\end{equation}
where $R_{i}$ are positive constants. A nested family of invariant
tori obtained from this transformation is shown in Fig. \ref{fig:shear_steady_tori}.

\begin{figure}[h!]
\centering{}\includegraphics[width=0.9\textwidth]{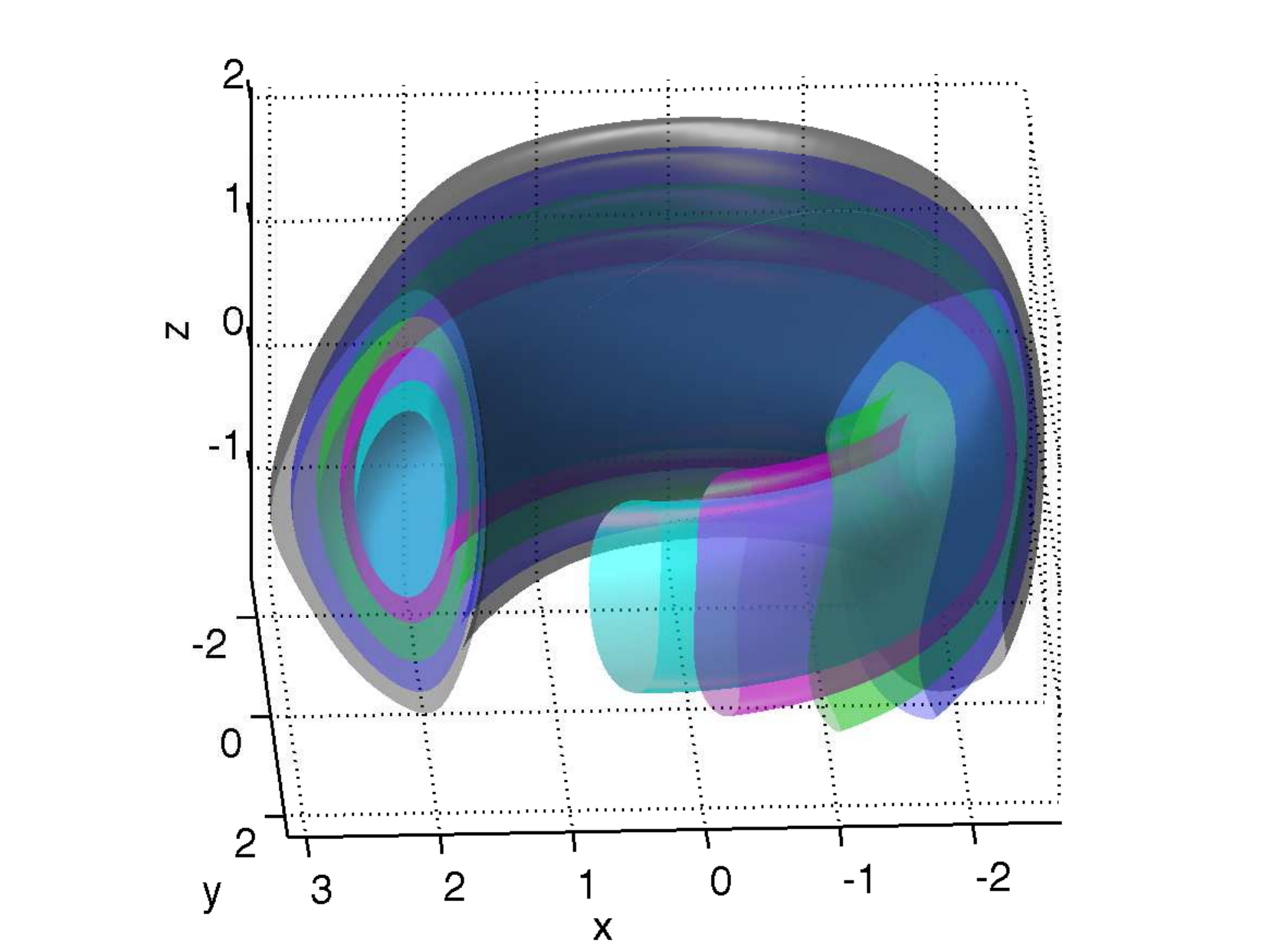}
\caption{An embedding of a nested elliptic LCS family in the steady ABC flow.
These LCSs approximate invariant tori from a finite-time observation
over the time interval $[0,40]$.}
\label{fig:shear_steady_tori} 
\end{figure}

Fig. \ref{fig:hel_shear_steady} shows the shear helicities $H_{n_{\pm}}$
along the line segment $y=4.7$, with dots marking the locations of
closed shearlines obtained from our finite-time analysis. Despite
the significant numerical noise in the computation of the shear helicity,
the zeros of $H_{n_{\pm}}$ move closer and closer to the computed
shearlines, validating these shearlines as curves on a shear LCS in
the sense of Definition 1.

\begin{figure}[h!]
\begin{centering}
\includegraphics[scale=0.35]{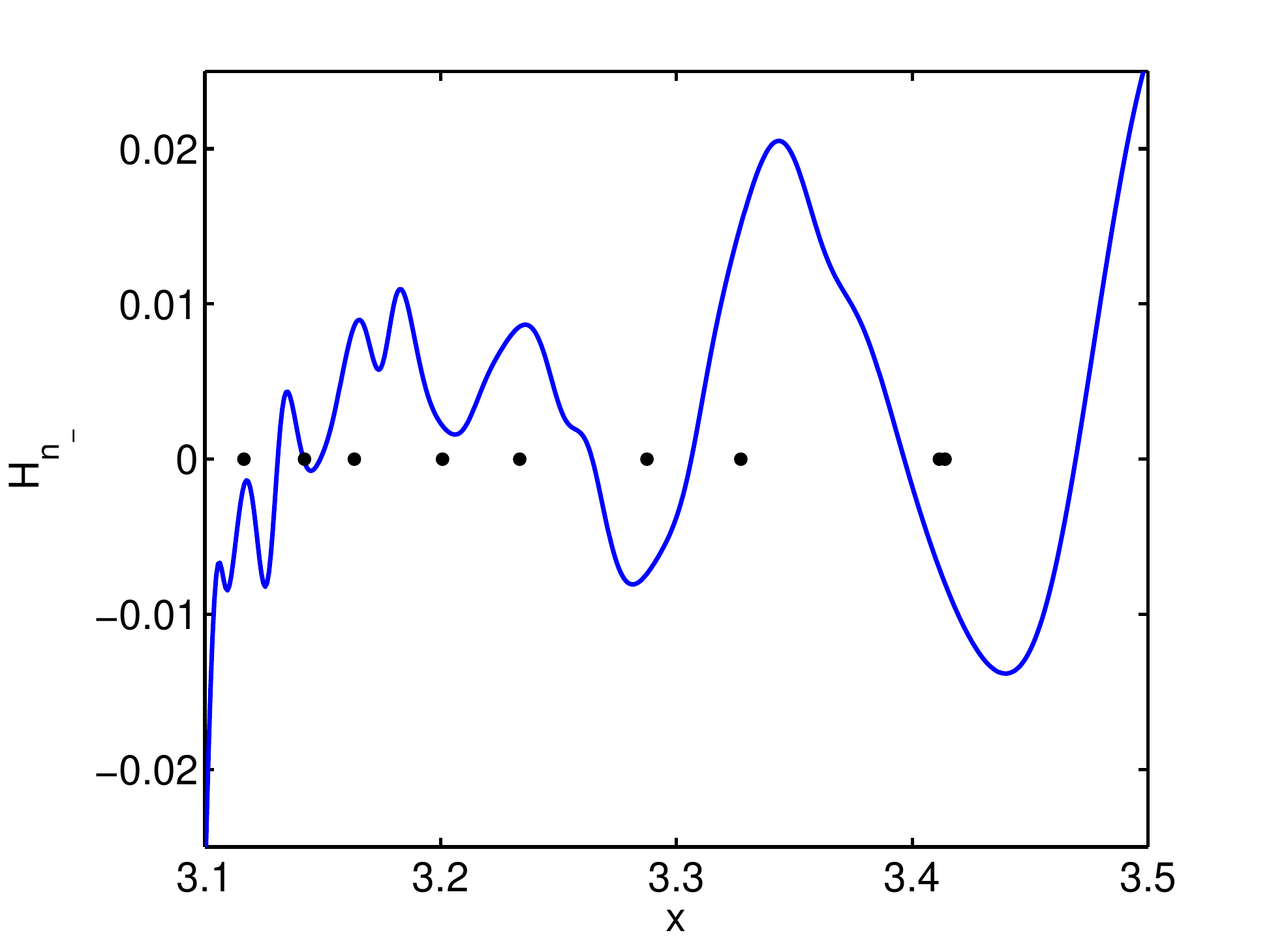} \includegraphics[scale=0.35]{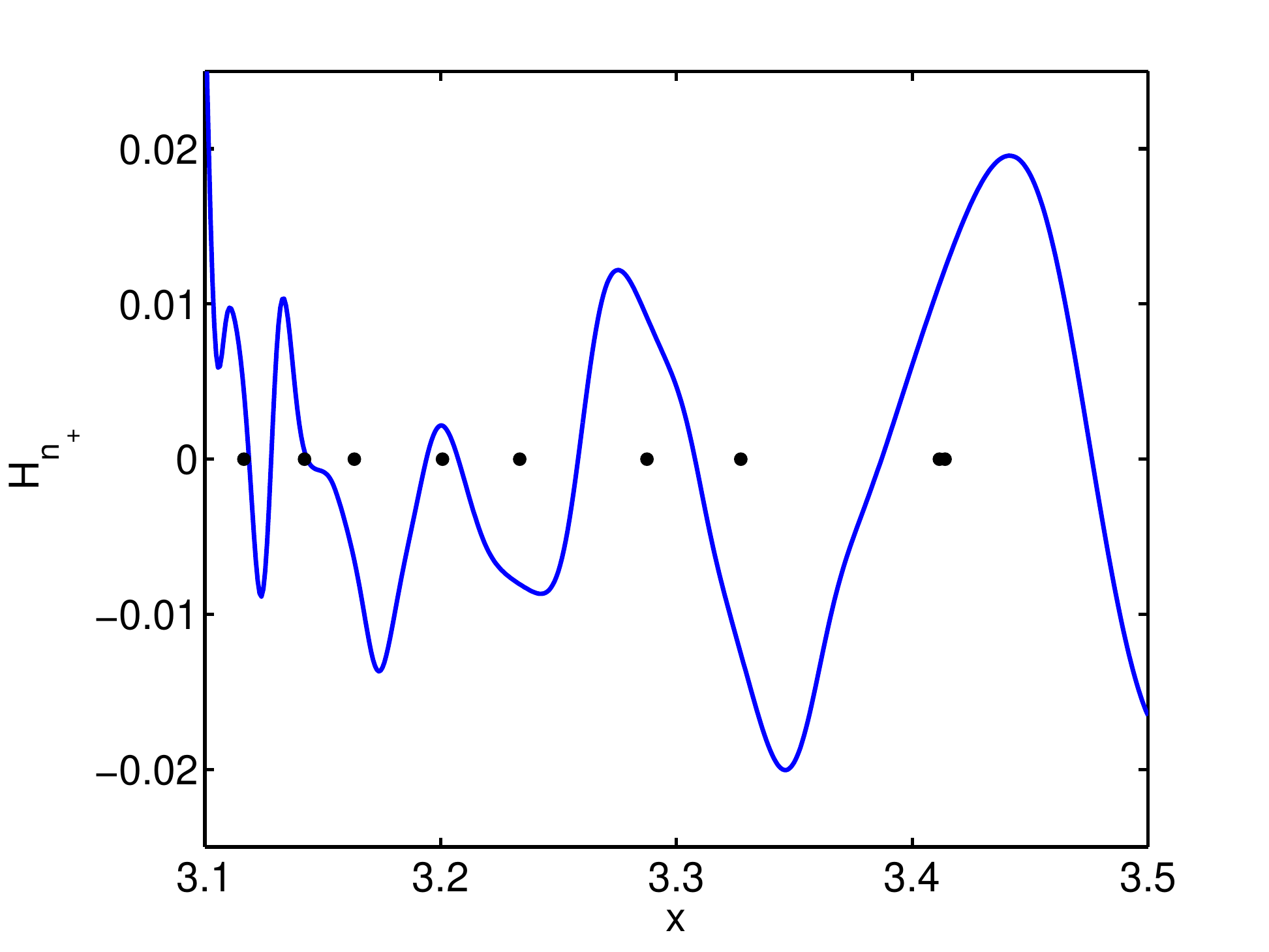} 
\par\end{centering}

\centering{}\includegraphics[scale=0.35]{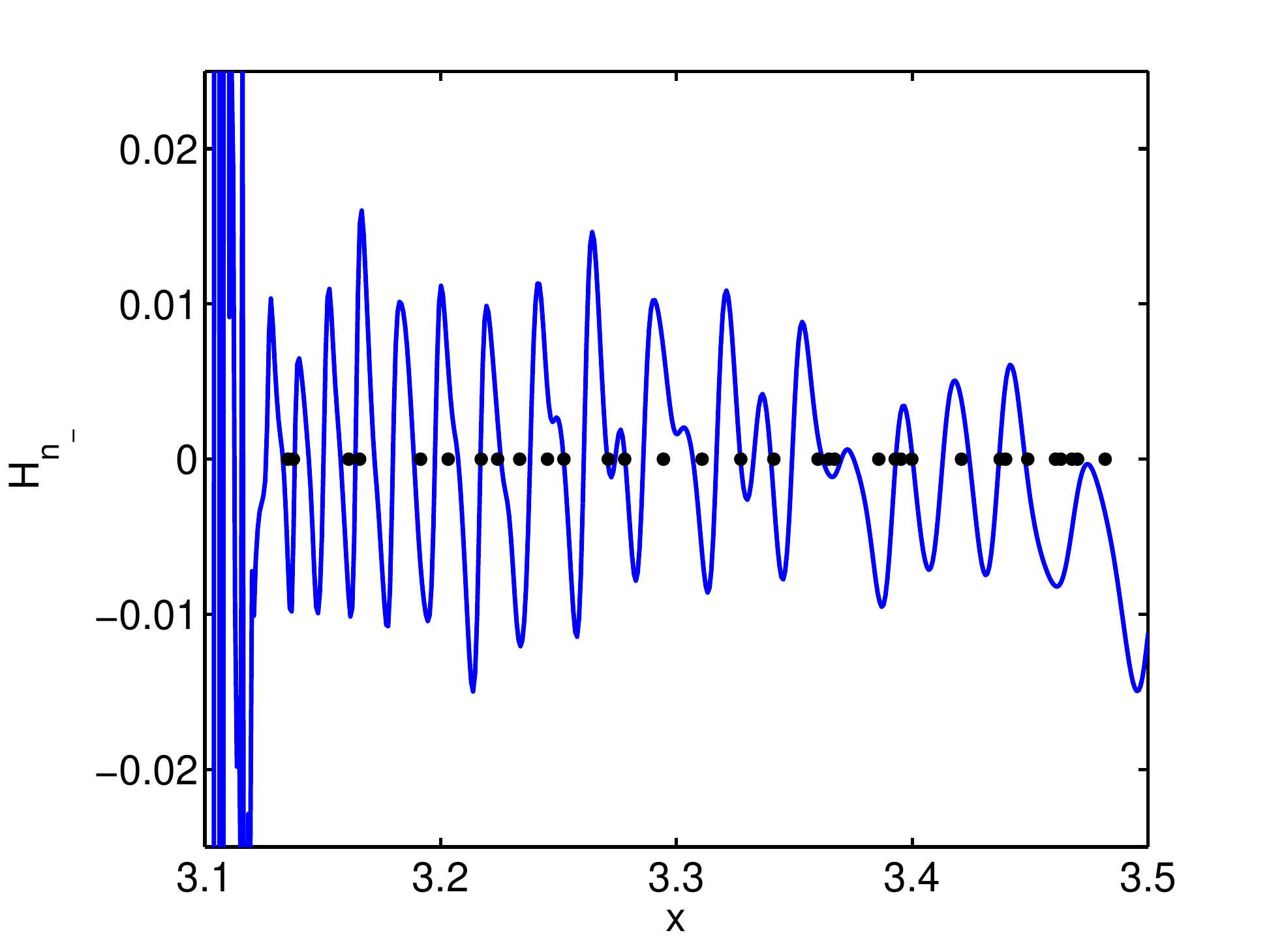} \includegraphics[scale=0.35]{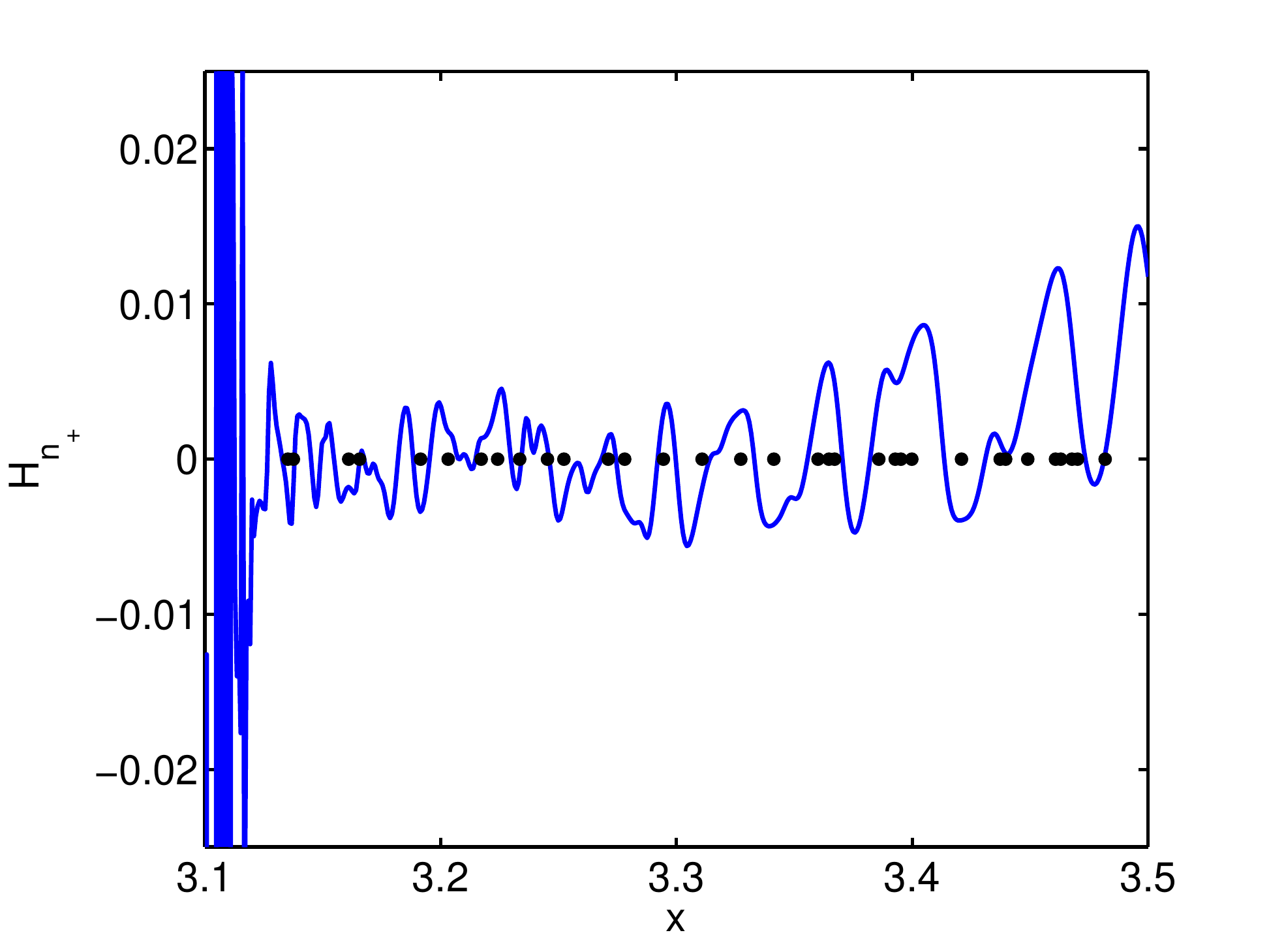}
\caption{Upper panels: The helicities $H_{n_{\pm}}$ of the shear vector fields
$n_{\pm}$, respectively, along the line $y=4.7$ for integration
length $T=40$. Black dots indicate the $x$ coordinate of the closed
shearlines shown in Fig. \ref{fig:shear_steady}. Lower panels: the
same for integration time $T=150.$}
\label{fig:hel_shear_steady} 
\end{figure}

\subsubsection{Repelling hyperbolic LCSs in the steady ABC flow}

We compute the Cauchy-Green strain tensor $C_{0}^{3}$ over a $500\times500$
grid on the plane $z=0$. The tolerance parameter in the computational
step H3 is chosen to be $\epsilon_{0}=10^{-4}$. Under this tolerance
level, intersections or repelling hyperbolic LCSs with the $z=0$
plane are shown in the left panel of Fig. \ref{fig:strlines_zero_helicity},
obtained as parametrized curves. Also shown is the vertical line $y=0.5$,
along which we compute the strain helicity $H_{\xi_{3}}$ (cf. the
right panel of Fig. \ref{fig:strlines_zero_helicity}). The latter
figure illustrates that the reduced strainlines shown in the left
panel are indeed intersections of the $z=0$ plane with repelling
hyperbolic LCSs in the sense of Definition 1. 

\begin{figure}
\centering{}\includegraphics[scale=0.35]{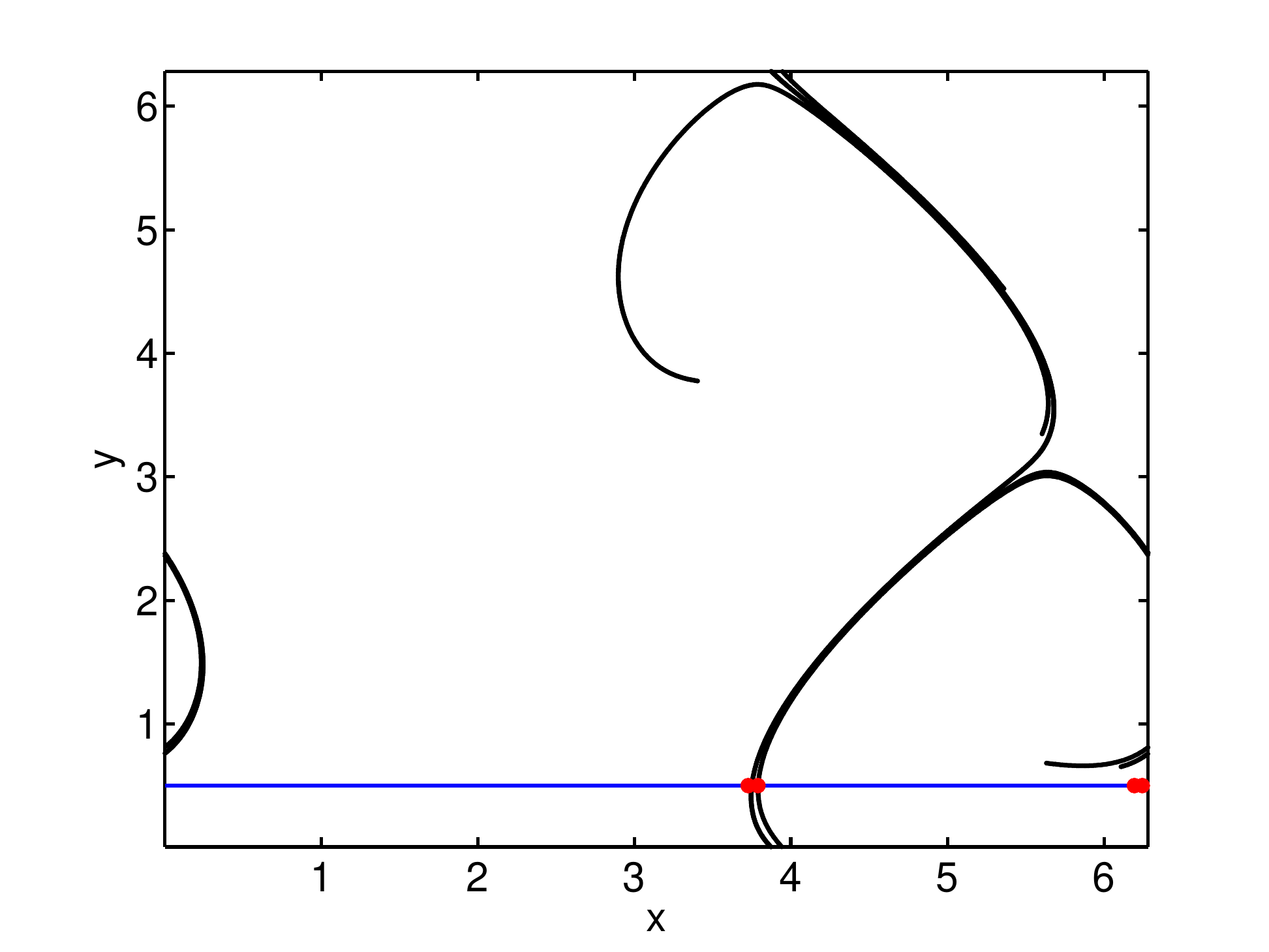}\includegraphics[scale=0.35]{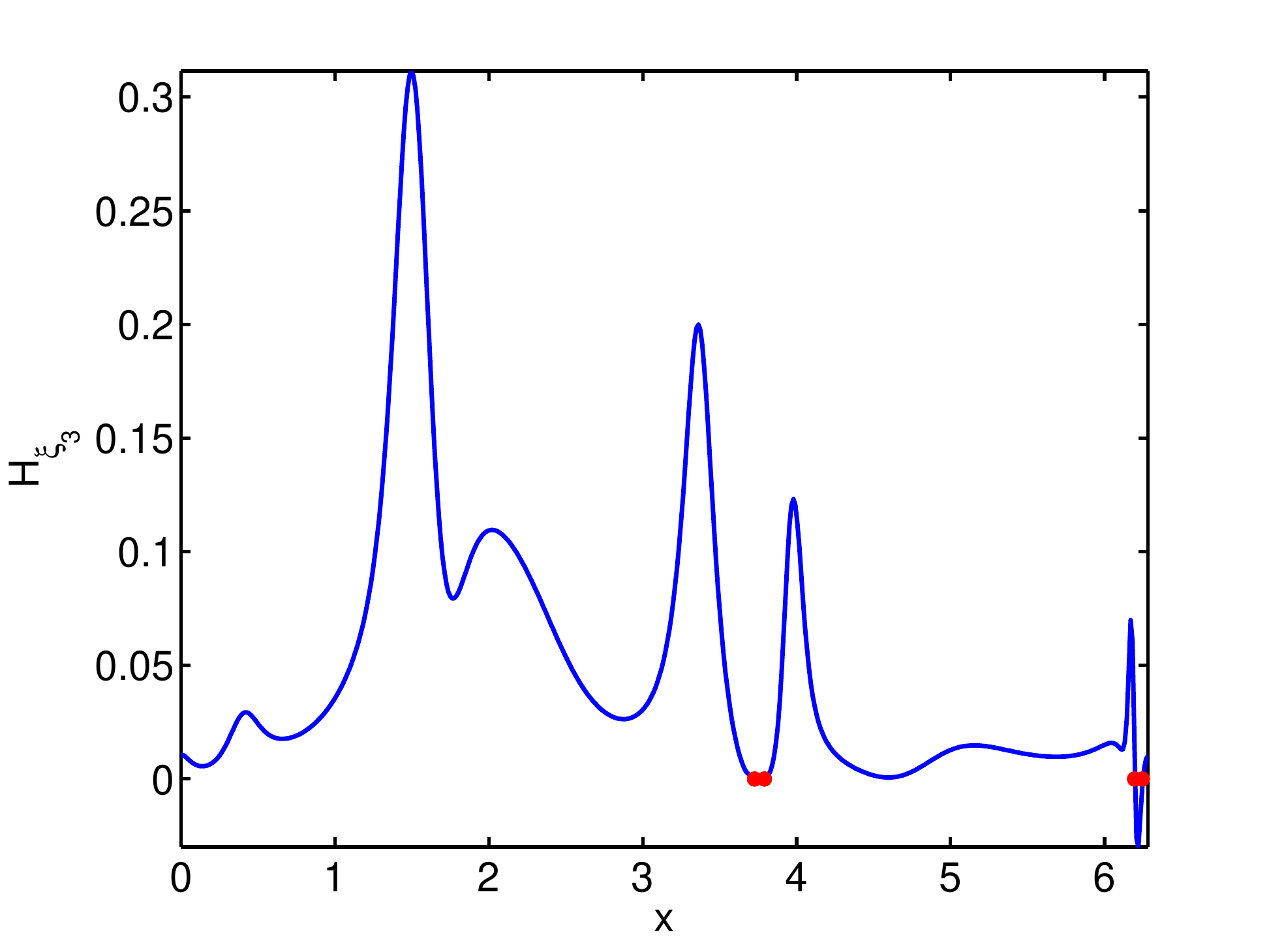}
\caption{Left panel: Reduced strainlines of minimal average helicity, with
the line $y=0.5$ shown in blue, and with red dots indicating points
where the helicity $H_{\xi_{3}}$ is exactly zero. Right panel: the
helicity $H_{\xi_{3}}$ plotted as a function of $x$ along the line
$y=0.5$, with its zeros highlighted in red.}
\label{fig:strlines_zero_helicity} 
\end{figure}

To illustrate the dynamical impact of the barrier surface emanating
from the reduced strainlines, we select one of these strainlines (shown
in green in the left panel of Fig. \ref{fig:strain_steady}). We perturb
this reduced strainline segment in the $x$ direction by $\pm0.01$
to obtain the blue and red curves shown in the same panel. We then
advect all three parameterized curves from $t_{0}=0$ to $t_{0}+T=3$
to obtain the surfaces shown in the left panel of Fig. \ref{fig:strain_steady}.
Note that the blue and red curves have noticeable upward and downward
$z$-drifts, respectively, while the surface evolving from the green
reduced strainline has no $z$-drift.

\begin{figure}
\includegraphics[width=0.3\textwidth]{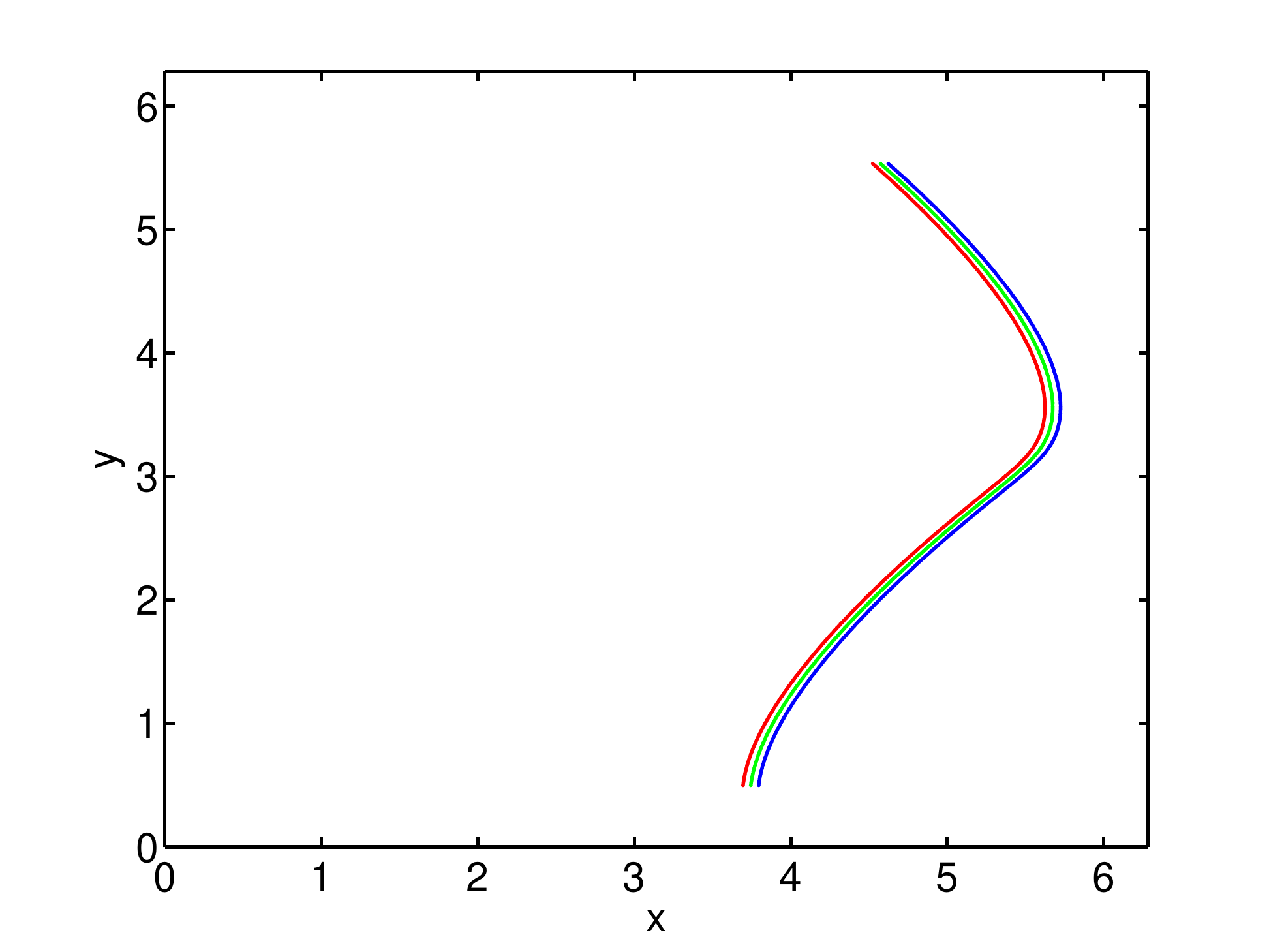}\includegraphics[width=0.75\textwidth]{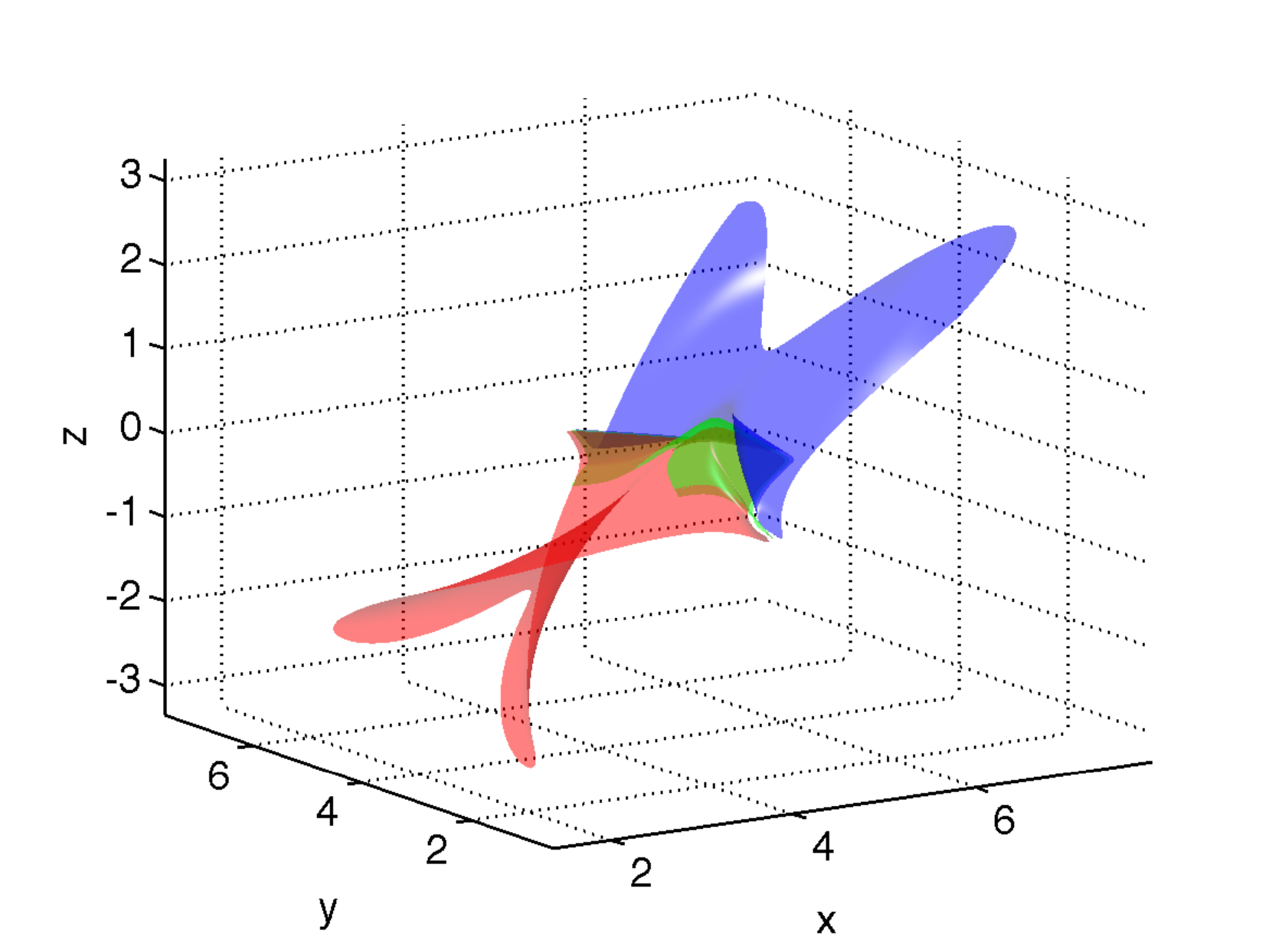}
\caption{Left panel: reduced strainline of zero helicity (green), and its perturbation
to the left (red) and to the right (blue) by $0.01$. Right panel:
Invariant surfaces through these three curves obtained by advection
under the flow map $F_{0}^{30}$. They illustrate the repelling barrier
property of the green LCS through the green reduced strainline. }
\label{fig:strain_steady}
\end{figure}

\subsection{Time-periodic ABC flow}

\label{Section:ABC_periodic} We now consider a temporally periodic
version of the ABC flow, given by

\begin{equation}
\begin{split} & \dot{x}=\left(A+0.1\sin t\right)\sin z+C\cos y,\\
 & \dot{y}=B\sin x+\left(A+0.1\sin t\right)\cos z,\\
 & \dot{z}=C\sin y+B\cos x.
\end{split}
\label{ABC_per_equations}
\end{equation}
The first return map to the plane $z=0$ is now a non-autonomous map.
Therefore, to gain insight into the flow from classical tools, only
a fully three-dimensional temporal Poincare map can be used. This
would result in spatially scattered points, as opposed to the sharply
defined shear and hyperbolic barriers that we will continue to obtain
form our approach.

\subsubsection{Elliptic LCSs in the time-periodic ABC flow }

We compute the Cauchy-Green strain tensor $C_{0}^{30\pi}$ over a
$500\times500$ grid in the plane $z=0$. This integration length
is equal to $15$ iterations of the temporal Poincare map $F_{0}^{2\pi}$.
The tolerance parameter in the computational step SH3 is chosen to
be $\epsilon_{0}=10^{-2}$. Fig. \ref{fig:shear_periodic} shows the
closed reduced shearlines we find as limit cycles of the equation
(\ref{eq:redshear-ODE}). The shear-helicity zero distribution along
these curves is similar to that in the steady case, and hence is omitted
here for brevity.

Next, we iterate the outermost closed reduced shearline under the
Poincare map $F_{0}^{2\pi}$. The result is a two-dimensional invariant
torus for $F_{0}^{2\pi}$, shown in Fig. \ref{fig:shear_periodic}
under the embedding (\ref{eq:trans2}). This torus is an intersection
of a three-dimensional invariant torus of the full, spatially and
temporally periodic flow (defined over the toroidal phase space $\mathbb{T}^{4}$)
with the $t=0$ hyperplane. 

\begin{figure}
\includegraphics[width=0.25\textwidth]{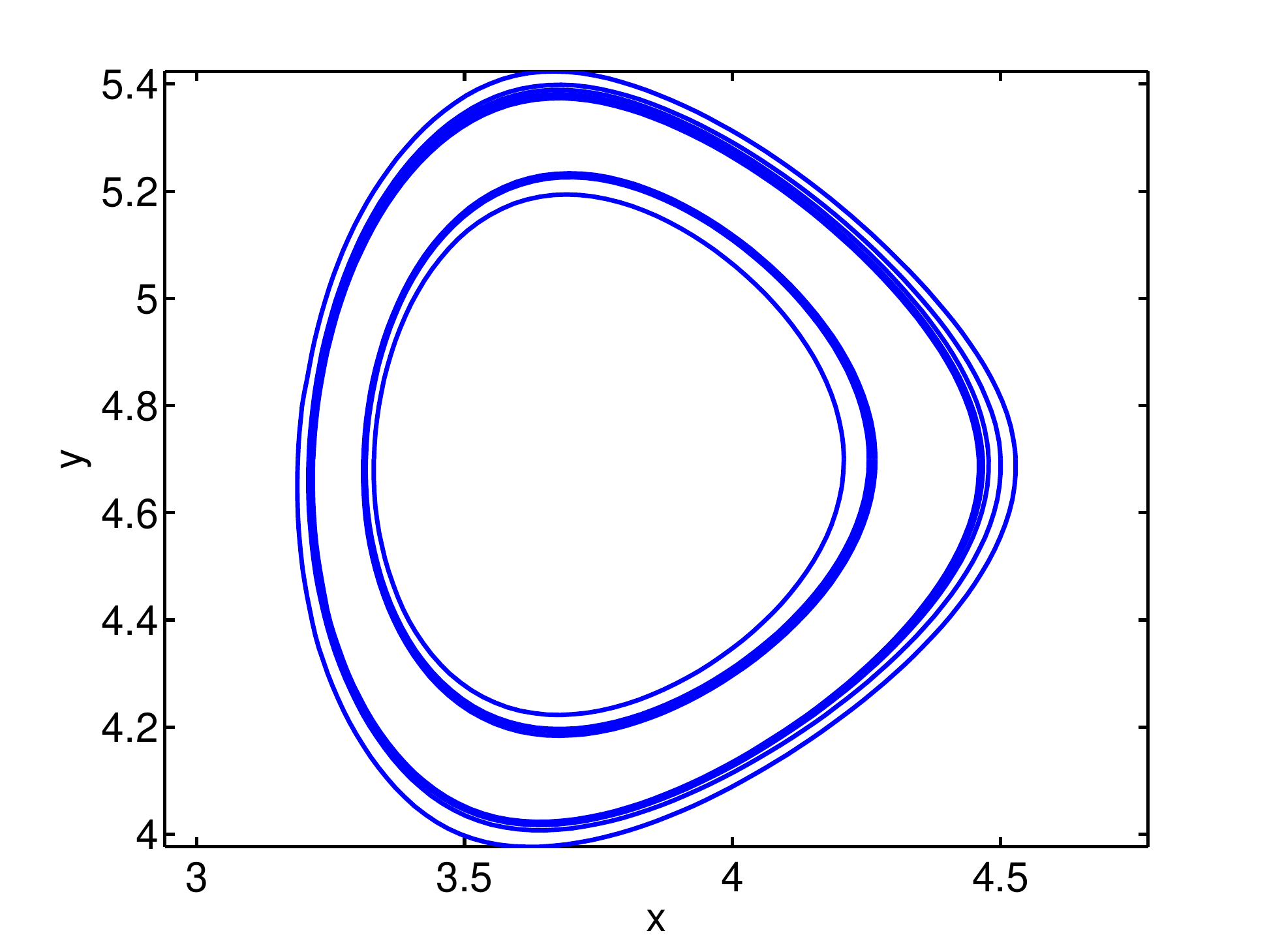}\includegraphics[width=0.75\textwidth]{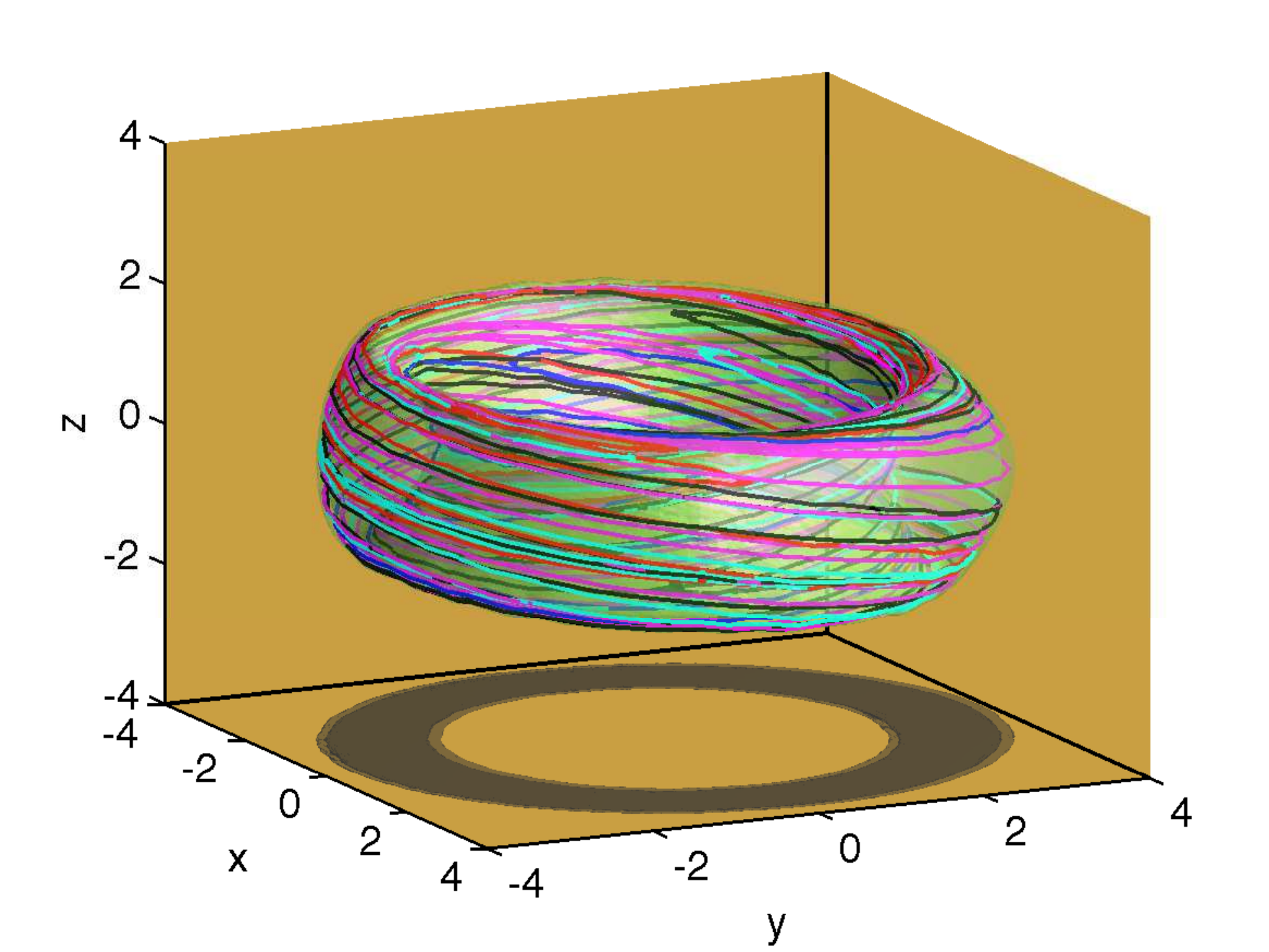}
\caption{Left: Closed reduced shearlines for the time-periodic ABC flow on
the plane $z=0$, constructed from flow data over the time interval
$[t_{0},t_{0}+T]=[0,30\pi]$. Right panel: invariant torus as elliptic
LCS for the temporal Poincare map, obtained from subsequent images
of the outermost closed shearline under iterations of $F_{0}^{2\pi}$.
(The blue, red, cyan, black, and magenta curves represent the $n$th
iterate of $F_{0}^{2\pi}$ for $n=50,75,110,160,200$, respectively).
To illustrate the invariance of the underlying torus, we also computed
400 iterates of the outermost closed shearline under $F_{0}^{2\pi}$,
obtaining the green, filamentation-free toroidal surface.}
\label{fig:shear_periodic} 
\end{figure}

We now illustrate the barrier property of the three-dimensional torus
represented by the elliptic LCS of Fig. \ref{fig:shear_periodic}.
To this end, we advect two initial conditions from the interior of
the two-dimensional torus starting from the $z=0$ reference plane,
and two other initial conditions from the exterior of this torus within
the same plane. These four initial conditions are placed on the grey
circle shown in the left panel of Fig. \ref{fig:shear_periodic}.
The center of this circle is on the outermost torus barrier, and is
advected as a blue trajectory. As seen in the right panel of Fig.
\ref{fig:shear_periodic}, the blue trajectory indeed remains confined
to a quasi-periodically deforming transport barrier in phase space.
This barrier keeps both the red and the yellow initial conditions
from spreading. In contrast, the green and black initial conditions
launched from outside the outermost torus barrier develop large excursions
over time.

\begin{figure}[h!]
\includegraphics[width=0.25\textwidth]{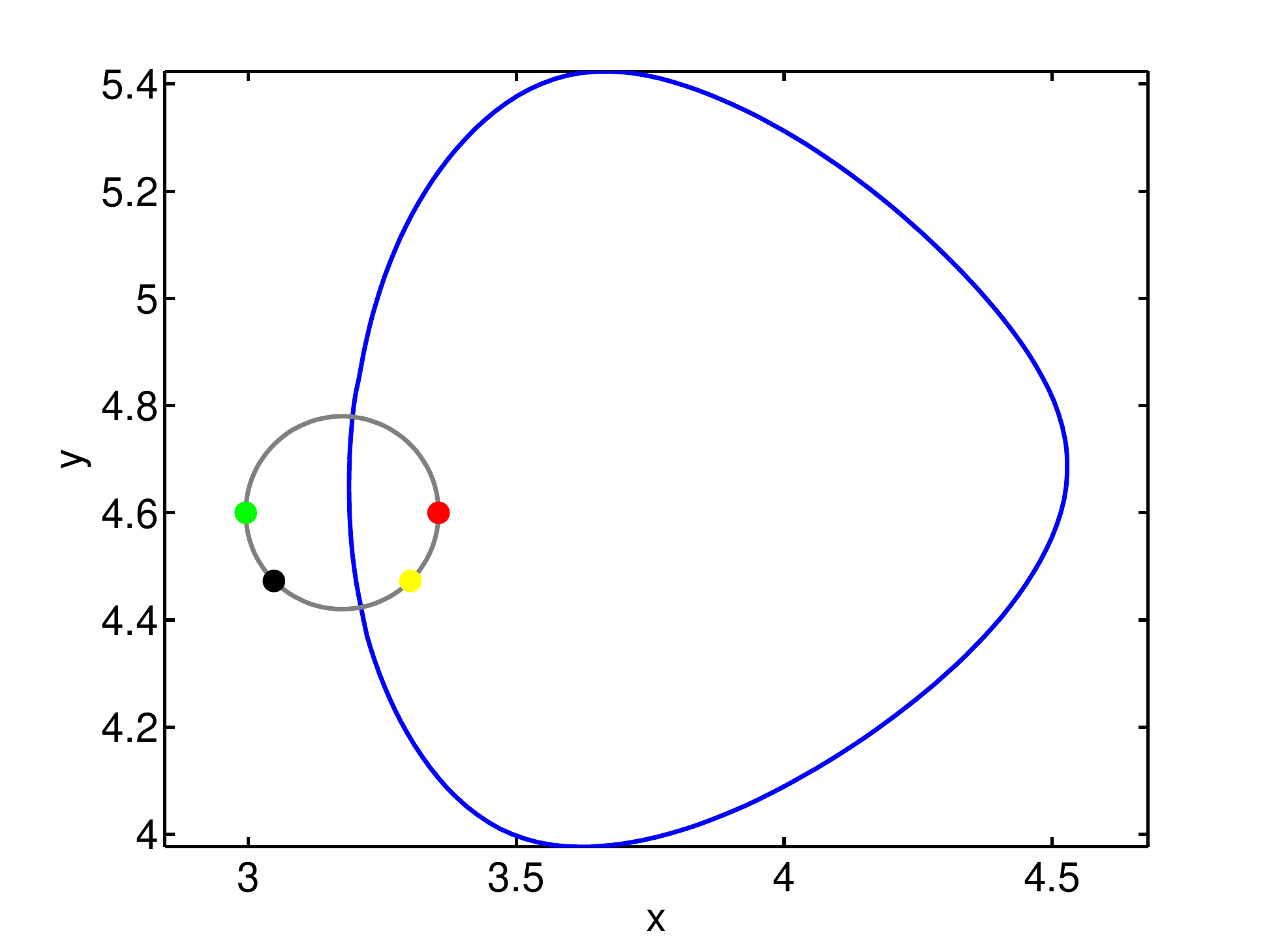}\includegraphics[width=0.75\textwidth]{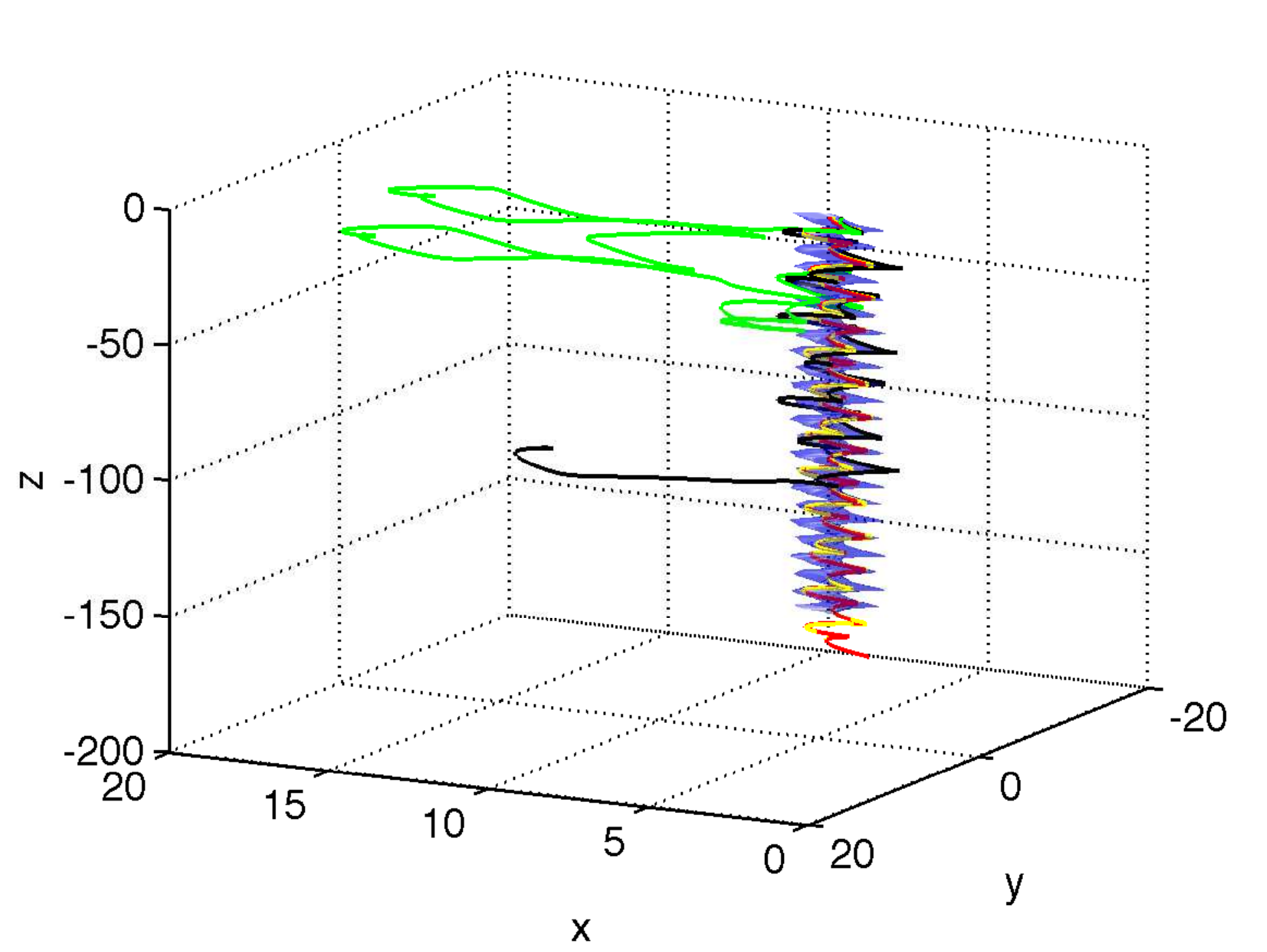}
\caption{Verification of the quasiperiodic transport barrier obtained from
reduced shearlines. Tracers launched inside (red and yellow), along
(blue), and outside (green and black) the outermost closed reduced
shearline show markedly different behavior in phase space (blue).
The time interval of advection was $[0,30\pi]$.}
\label{fig:shear_tracers_per} 
\end{figure}

\subsubsection{Repelling hyperbolic LCSs in the time-periodic ABC flow\label{sub:Repelling-hyperbolic-barriers}}

We now compute repelling hyperbolic barriers for the time-periodic
ABC flow using the slicing method described in steps H1-H5 of section
\ref{sub:Algorithm-for-hyperbolic}. We select the discrete family
of 21 planes
\[
\Pi(s_{1})=\left\{ (x,y,z)\in[0,2\pi]^{3}:\,\, z=s_{1}\right\} ,\qquad s_{1}=0.00,\,0.005,\,0.01,\ldots,\,0.1,
\]
with a $500\times500$ uniform grid $\mathcal{G}_{0}$ within each
such plane. Over the time interval ranging from $t_{0}=0$ to $t_{0}+T=4.0$,
we carried out the procedure outlined in steps H1-H5 of section \ref{section:comp}.
In step H2, the grid $\mathcal{G}_{1}$ was chosen as $600\times10$
in $(x,y)$, and the helicity parameter was chosen as $\epsilon_{0}=10^{-4}$.
The filtered reduced strainlines obtained from H1-H4 on the $\Pi(0)$
plane are shown in the left panel of Fig. \ref{fig:strain_per}. The
right panel of the same figure shows the reconstructed barrier surface
by performing step H5 across the plane family $\Pi(s_{1})$ and interpolating
smooth surfaces over the resulting reduced strainline segments. 

\begin{figure}
\includegraphics[width=0.2\textwidth]{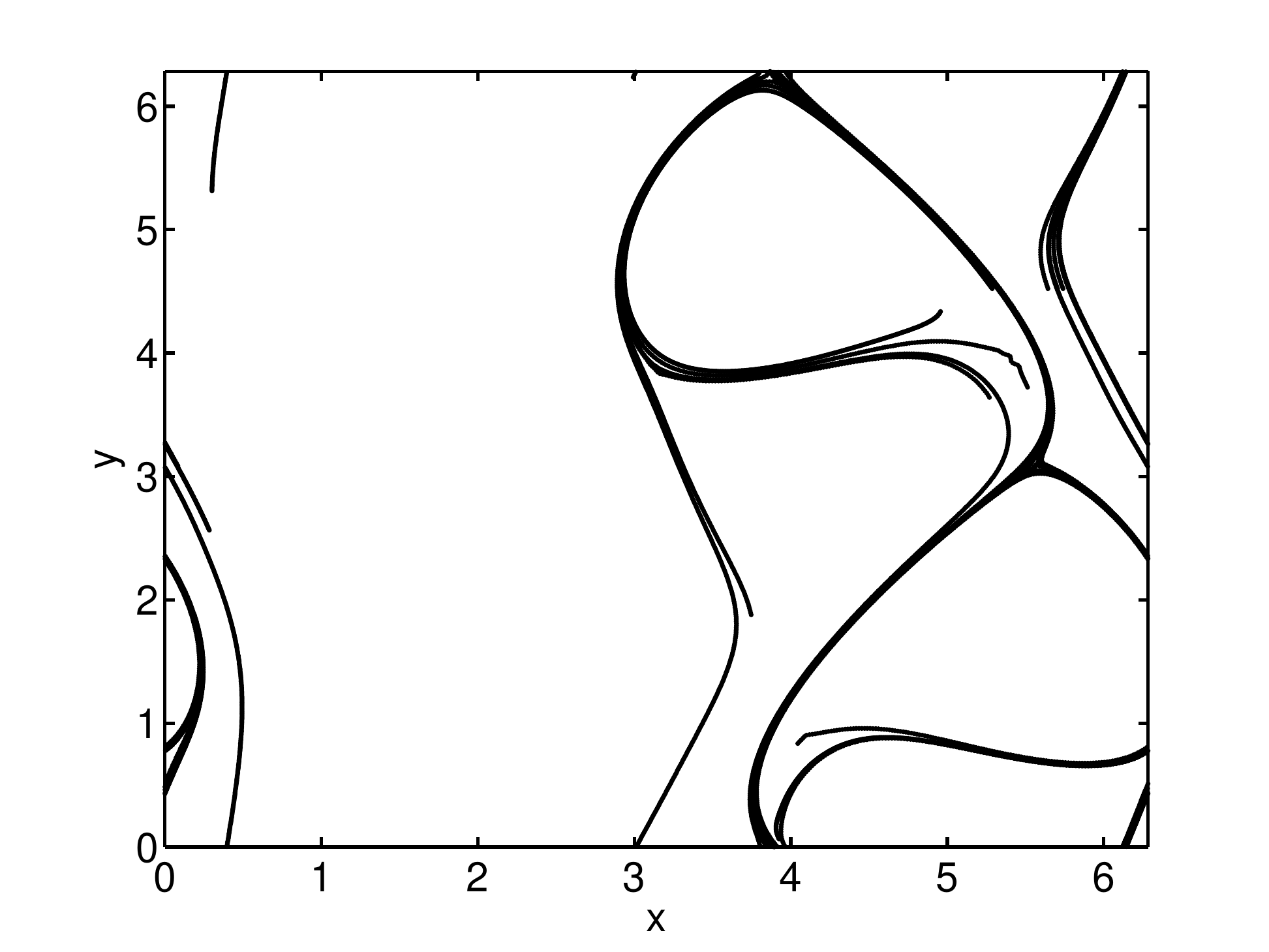}\includegraphics[width=0.8\textwidth]{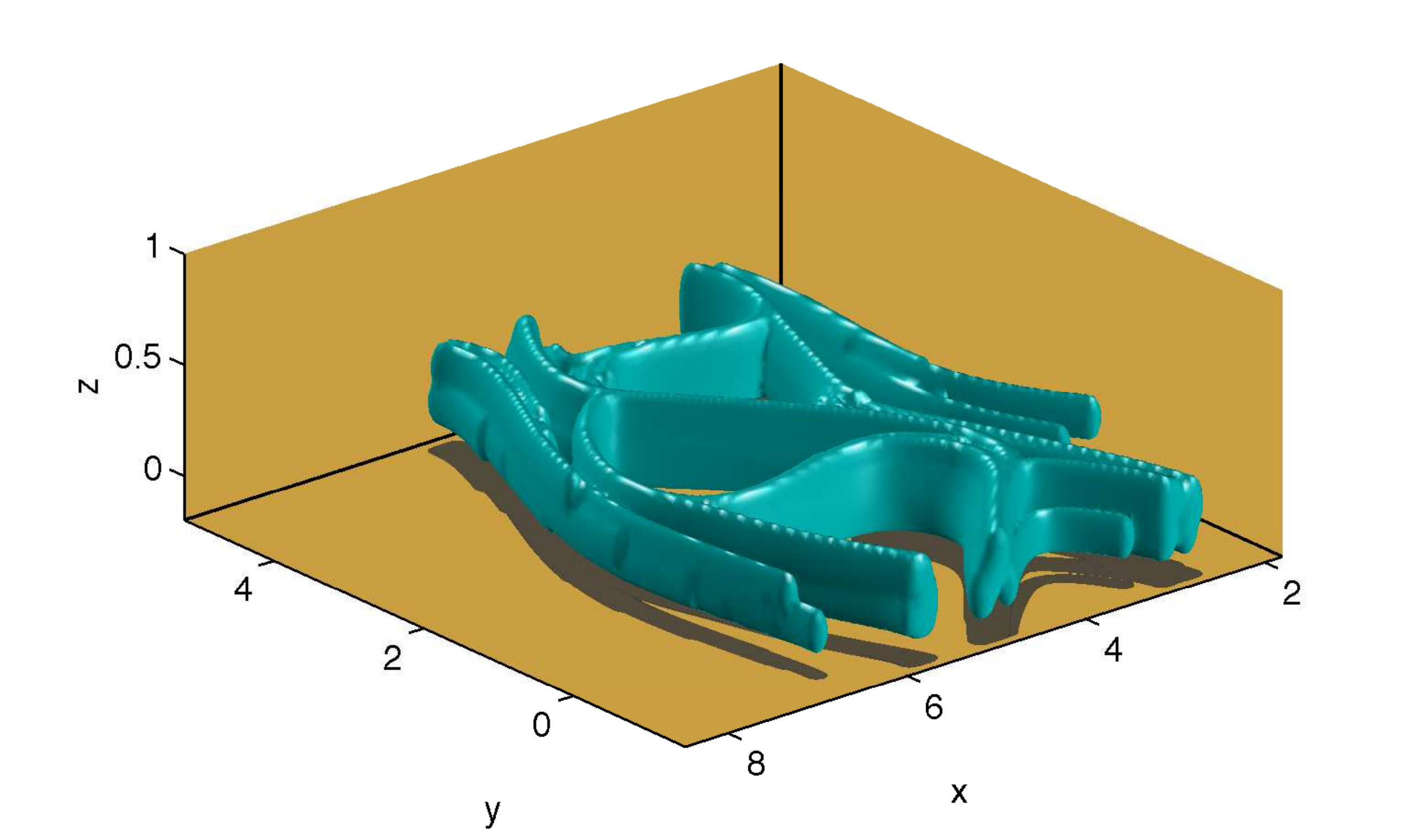}
\caption{Left: Reduced strainlines on the $\Pi(0)$ plane for the temporally
periodic ABC flow for the time interval $[0.0,4.0]$. Right: repelling
hyperbolic LCSs obtained by computing reduced strainlines over the
plane family $\Pi(s_{1})$ .}
\label{fig:strain_per} 
\end{figure}

\subsection{Chaotically forced ABC flow}

\label{Section:ABC_aperiodic}Here we consider a temporally aperiodic
version of the ABC flow, given by the equations

\begin{equation}
\begin{split} & \dot{x}=\left(A+F(t)\right)\sin z+C\cos y,\\
 & \dot{y}=B\sin x+A\left(A+F(t)\right)\cos z,\\
 & \dot{z}=C\sin y+B\cos x,
\end{split}
\label{ABC_aperiodic_equations}
\end{equation}
with $F(t)$ representing a chaotic signal. The signal is generated
by a trajectory close to the strange attractor of a periodically forced
and damped Duffing oscillator (see Fig. \ref{chaotic_signal}). The
temporally aperiodic flow \eqref{ABC_aperiodic_equations} admits
neither a well-defined spatial nor a well-defined temporal autonomous
first return map. Therefore, the simplified barrier visualization
methods used for elliptic barriers in the steady (Fig. \ref{fig:shear_steady_tori})
and time-periodic (Fig. \ref{fig:shear_periodic}) ABC flows are no
longer applicable.

\begin{figure}[h!]
\centering \includegraphics[scale=0.3]{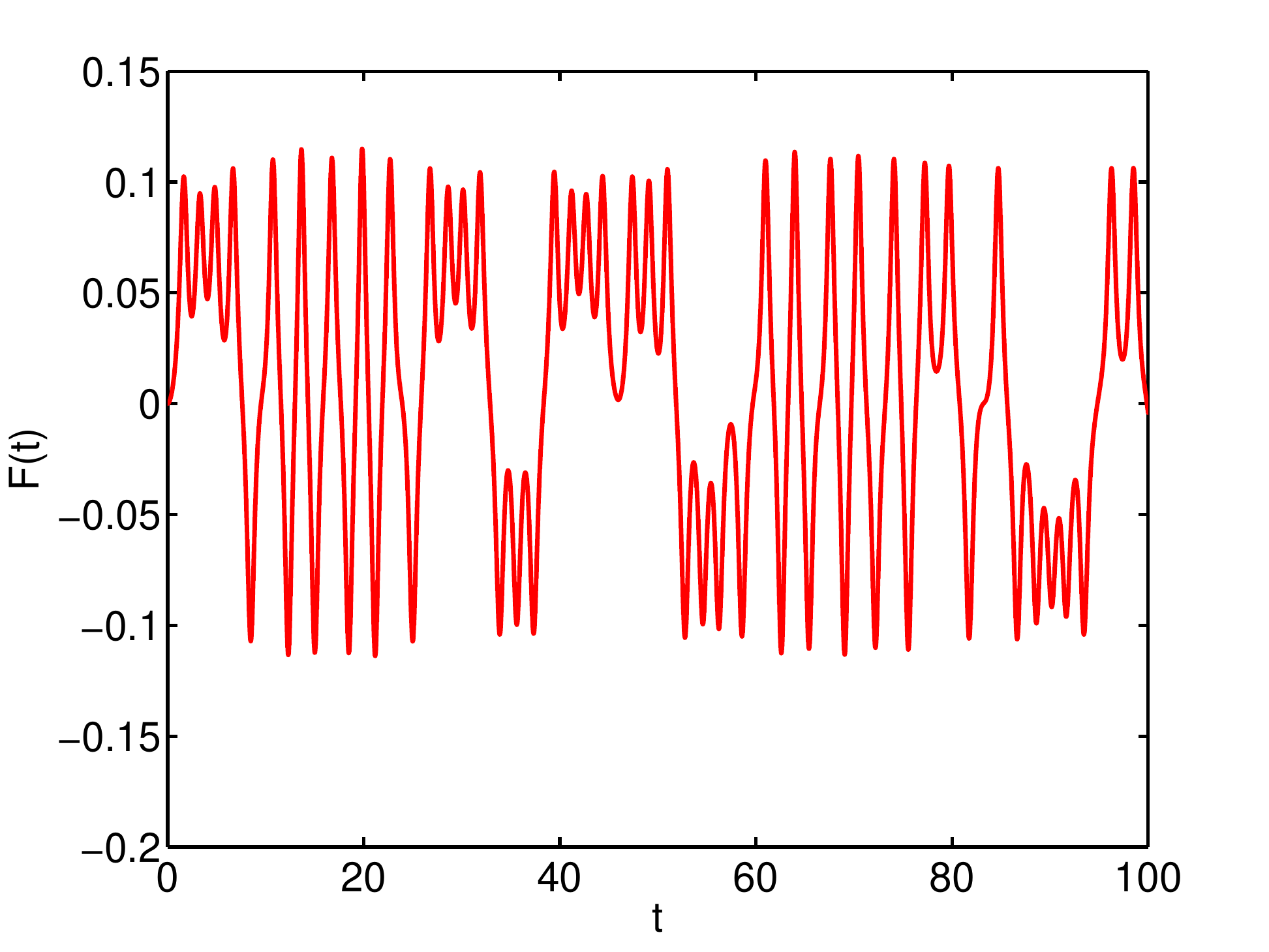} \caption{Aperiodic forcing used in the temporally aperiodic ABC-type flow \ref{ABC_aperiodic_equations}.}
\label{chaotic_signal}
\end{figure}

\subsubsection{Elliptic LCSs in the chaotically forced ABC flow}

We compute the Cauchy-Green strain tensor $C_{0}^{100}$ over a $500\times500$
grid in each member of the plane family 
\begin{equation}
\Pi(s_{1})=\left\{ (x,y,z)\in[0,2\pi]^{3}:\,\, z=s_{1}\right\} ,\qquad s_{1}=2k\pi/150,\qquad k=0,1,2,...,149.\label{eq:chaotic-elliptic-planes}
\end{equation}
The forthcoming computations were carried out in a parallelized fashion
over the $150$ $s_{1}$-slices defined in  The closed reduced strainlines
obtained form SH1-SH4 on the $\Pi(0)$ plane are shown in the upper
left panel of Fig. \ref{fig:shear_aperiodic}. The tolerance parameter
in the computational step SH3 is chosen to be $\epsilon_{0}=10^{-2}$.
The upper right panel of the same figure shows the reconstructed outermost
elliptic LCS by performing step SH5 across the plane family $\Pi(s_{1})$
and interpolating smooth surfaces over the resulting closed shearline
segments. The lower left panel of the figure confirms the coherence
of the detected barrier up to time $100$. The lower right panel of
the figure shows that the extracted barrier remains coherent under
advection even at time $150$. 

\begin{figure}[h!]
\begin{centering}
\includegraphics[width=0.3\textwidth]{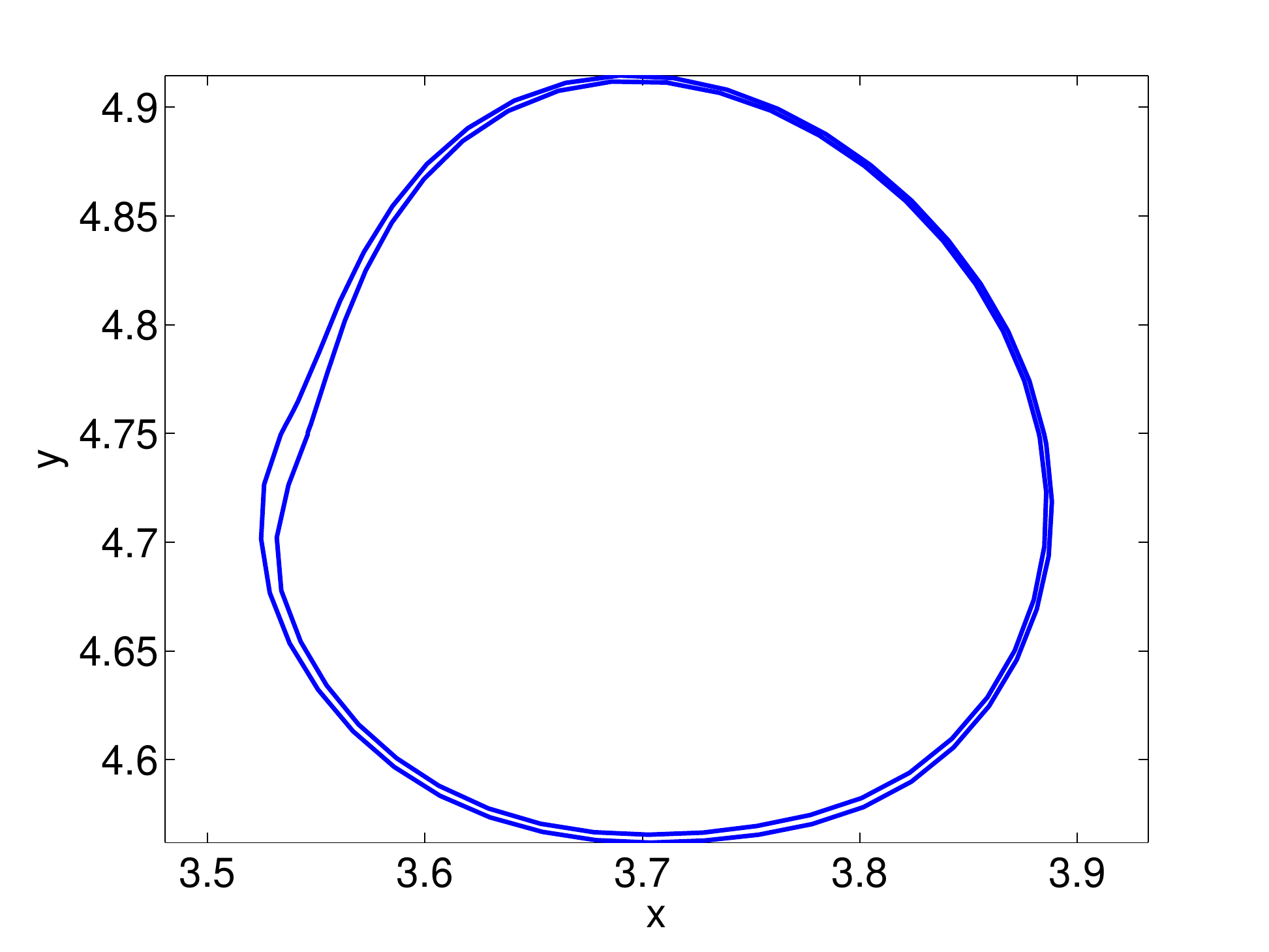}\includegraphics[width=0.7\textwidth]{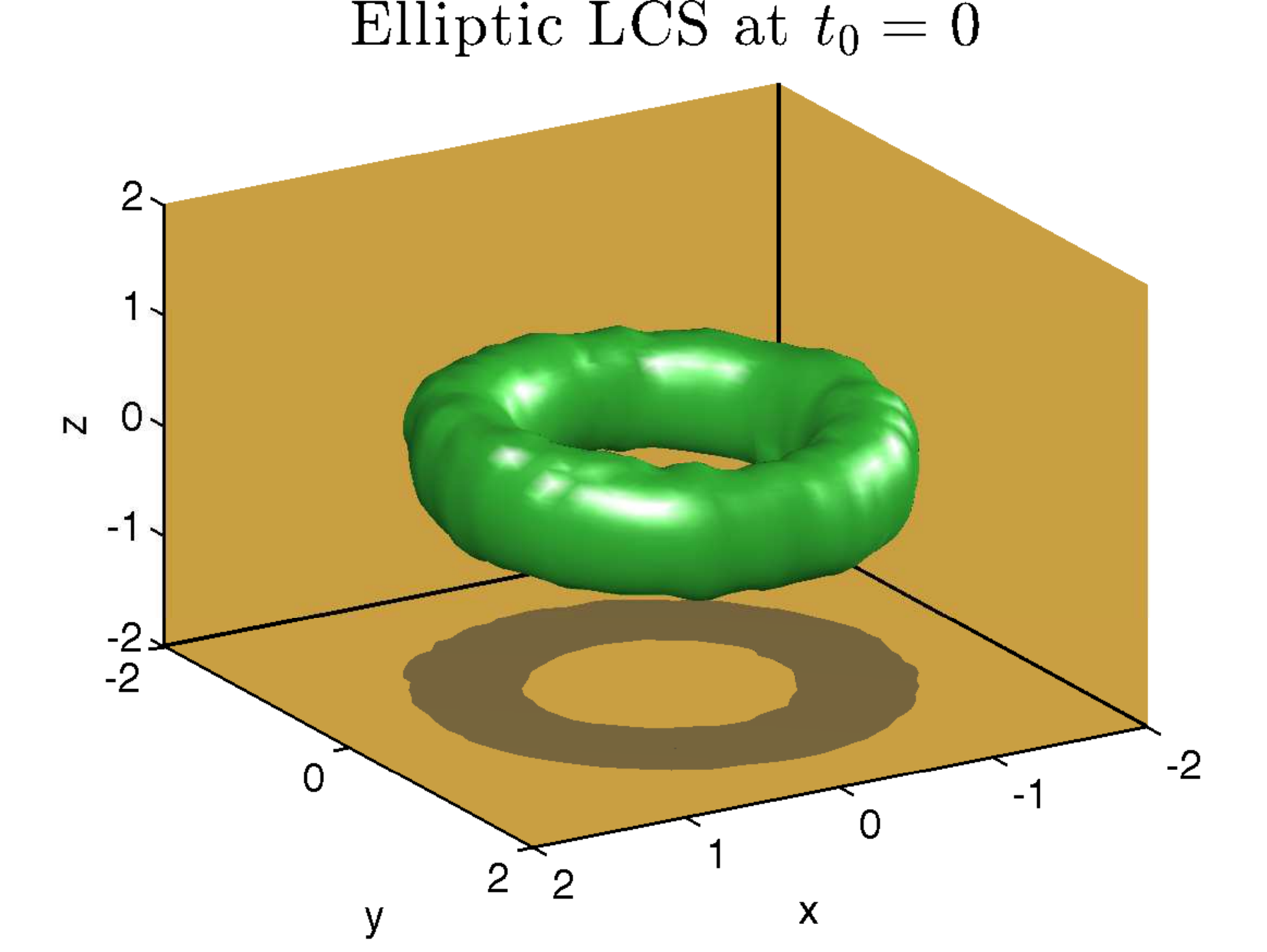} 
\par\end{centering}

\centering{}\includegraphics[width=0.55\textwidth]{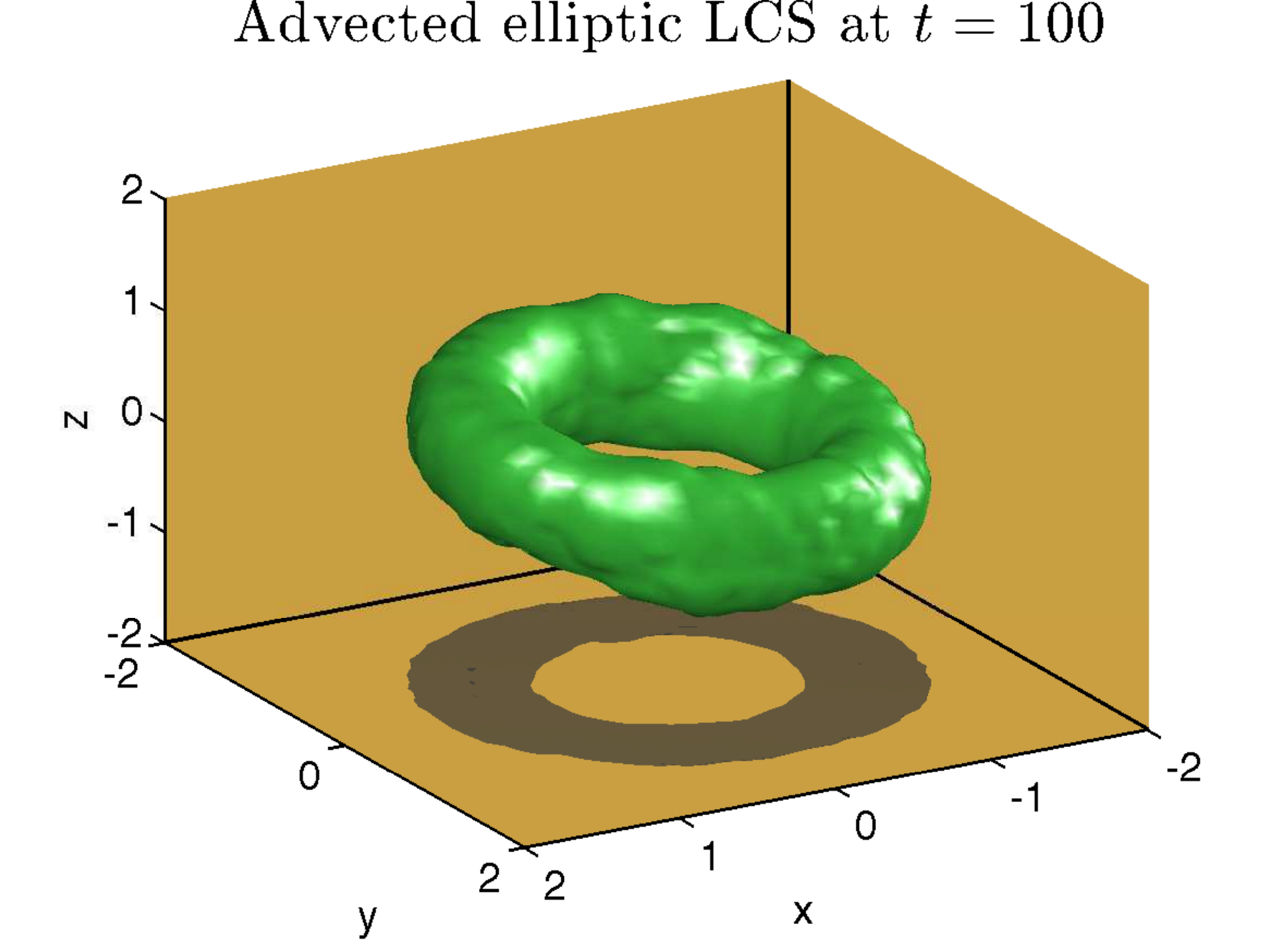}\includegraphics[width=0.55\textwidth]{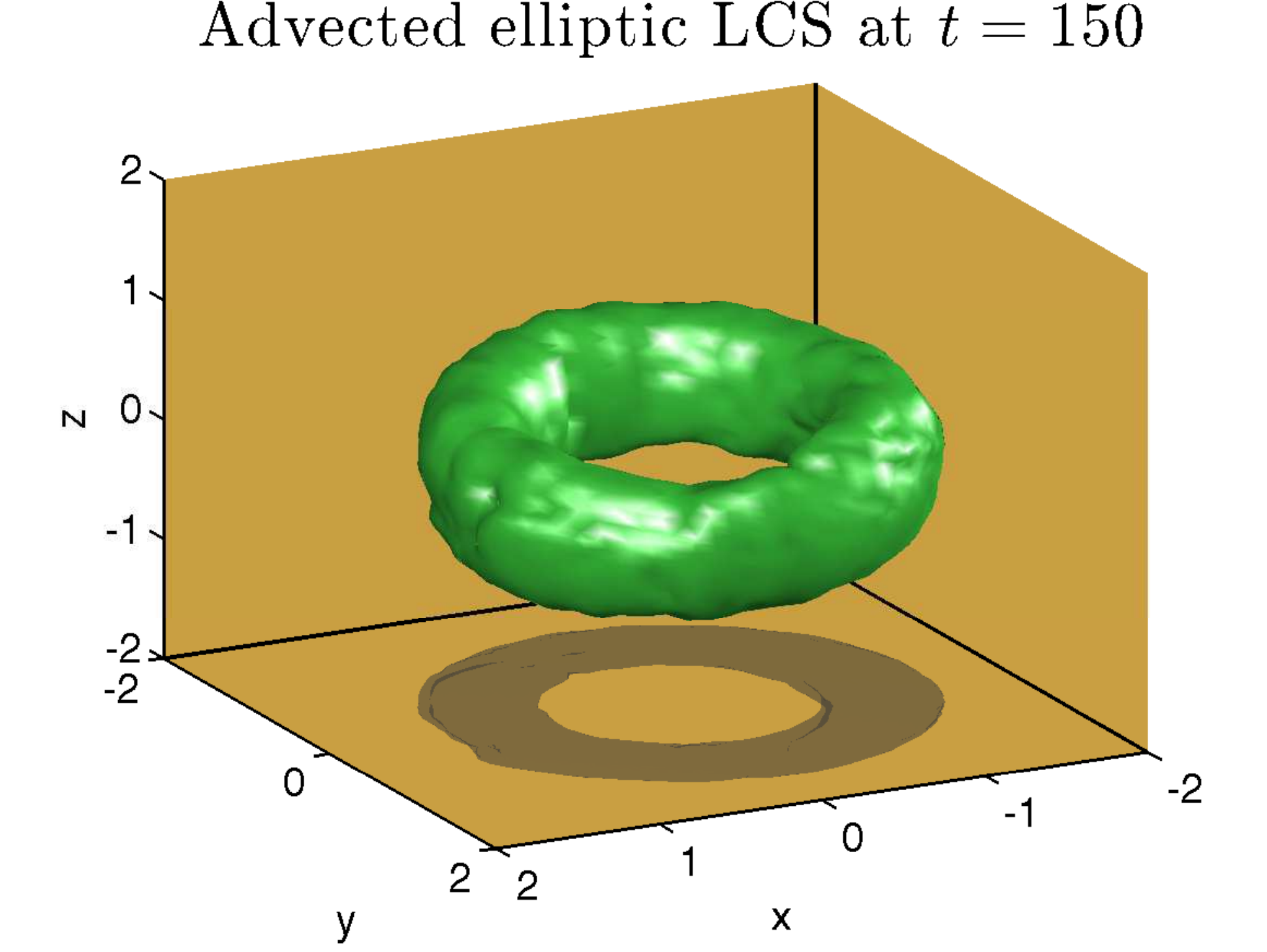}
\caption{Elliptic LCS in the chaotically forced ABC flow. \emph{Upper left:}
Reduced closed shearlines in the $\Pi(0)$ (i.e., $z=0$ ) plane computed
from $C_{0}^{100}$. \emph{Upper right:} Outermost elliptic LCS at
time $t_{0}=0$, visualized through the torus embedding \eqref{eq:trans2}.
\emph{Lower left: }Advected elliptic LCS at time $t_{0}+T=100$, the
final time used in its construction. \emph{Lower right:} Advected
elliptic LCS at time $t_{0}+T=150$ which is larger 50\% larger than
the final time used in its construction.\textbf{ }}
\label{fig:shear_aperiodic} 
\end{figure}

The time interval used in verifying sustained coherence for the elliptic
LCS in the lower right panel of Fig. \ref{fig:shear_aperiodic} is
50\% longer than the time interval used to extract this barrier. This
sustained coherence property is remarkable, as illustrated by Fig.
\ref{fig:ring_torus_aper}. In this figure, a circle of one million
initial conditions is selected as a perturbation to the smaller diameter
of the torus barrier. Just after an advection time of $t=13.0$, the
ring quickly loses all its coherence, stretching and folding by a
large amount in a visibly chaotic fashion. 

\begin{figure}
\includegraphics[width=0.35\textwidth]{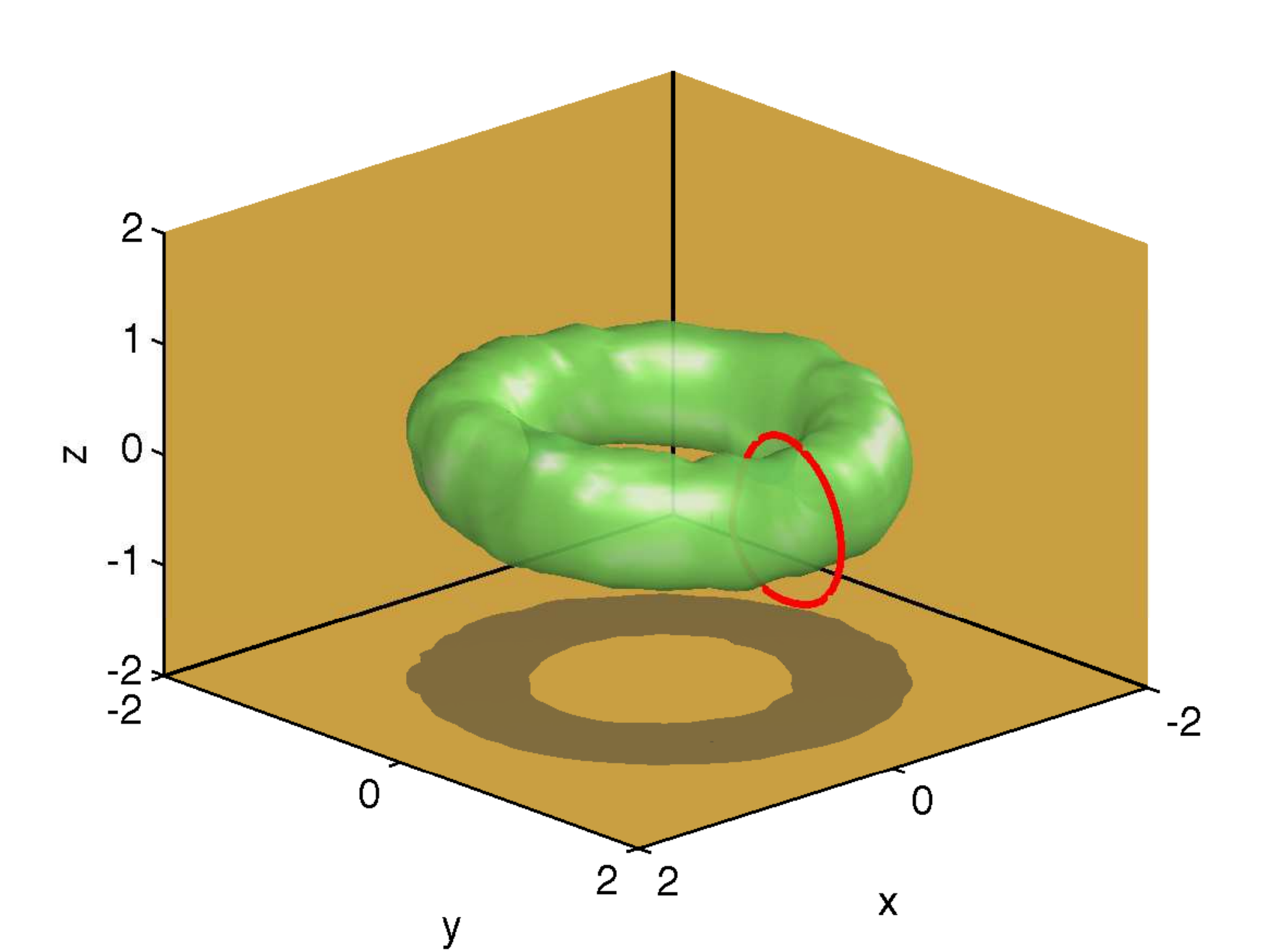}\includegraphics[width=0.65\textwidth]{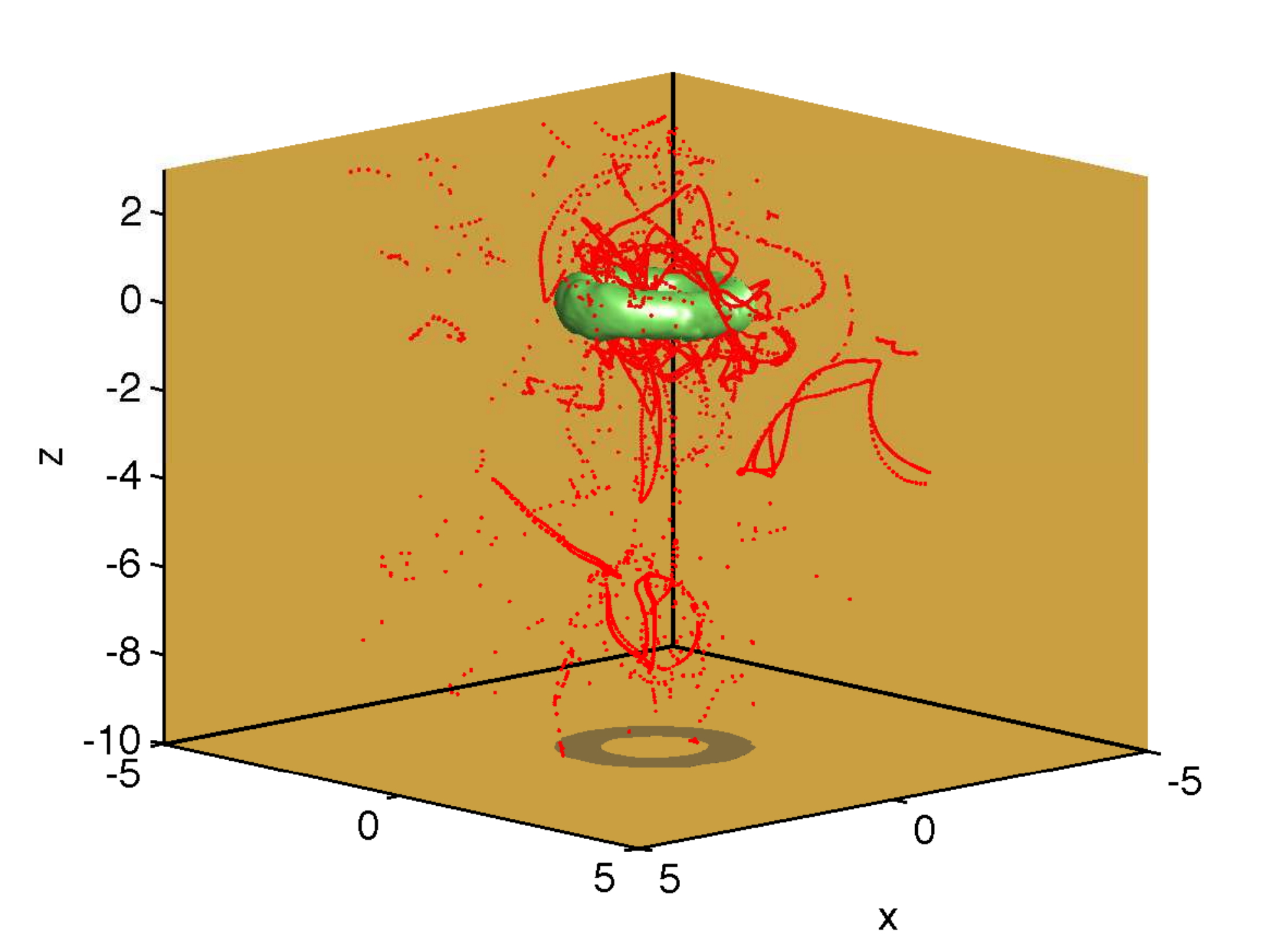}
\caption{Evolution of an elliptic LCS and of a ring placed near the LCS at
time $t_{0}=0.0$ (left) into their final position at time $t=13.0$
(right).}
\label{fig:ring_torus_aper} 
\end{figure}

Fig. \ref{fig:tracers_aper} shows the same type of verification of
the optimality of the barrier that we employed in Fig. \ref{fig:shear_tracers_per}
for the time-periodic ABC flow. Again, tracers launched inside the
barrier remain confined to the interior of the barrier, while tracers
launched slightly outside the barrier exhibit large excursions. 

\begin{figure}
\centering{}\includegraphics[width=0.3\textwidth]{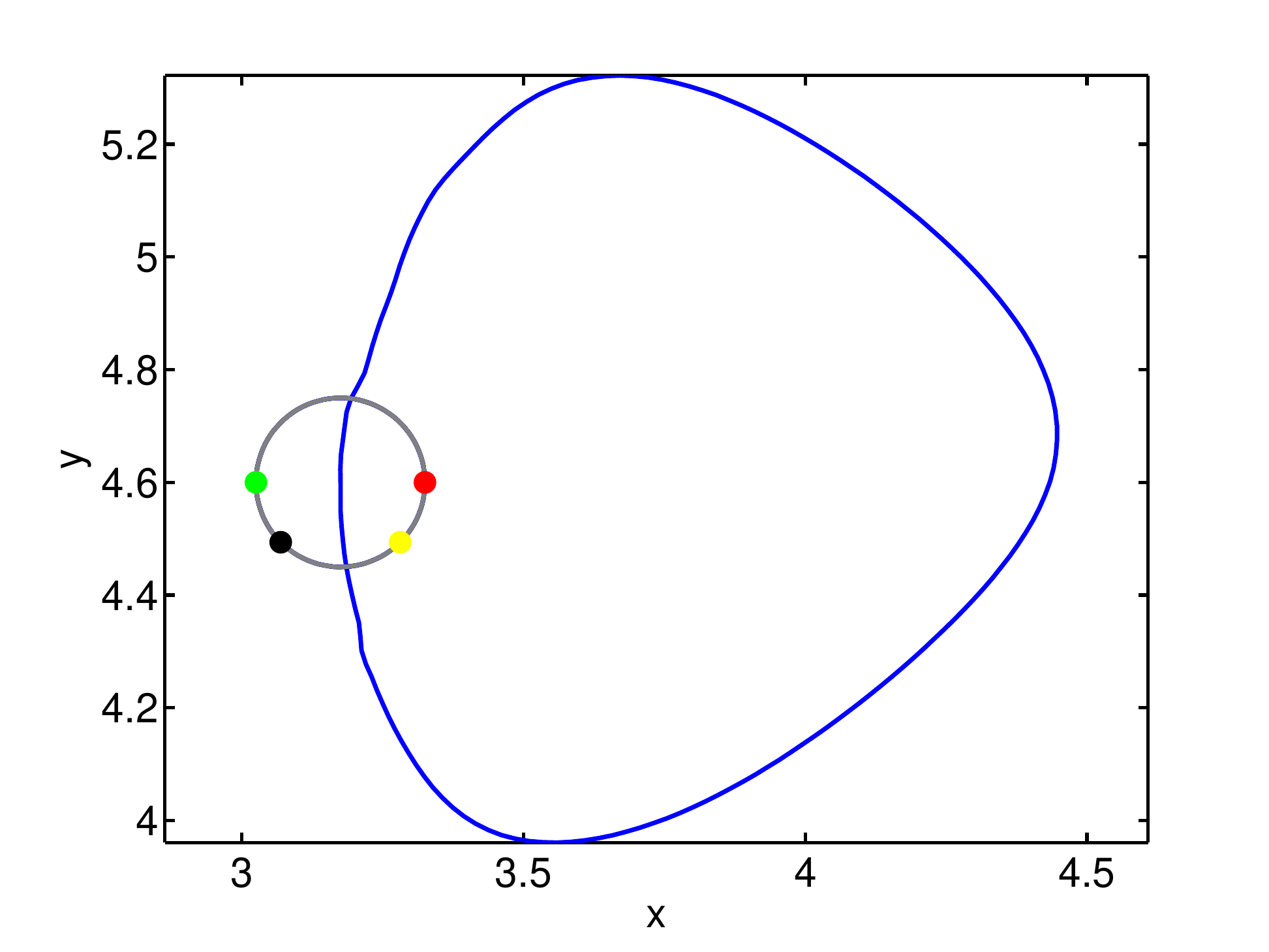}\includegraphics[width=0.7\textwidth]{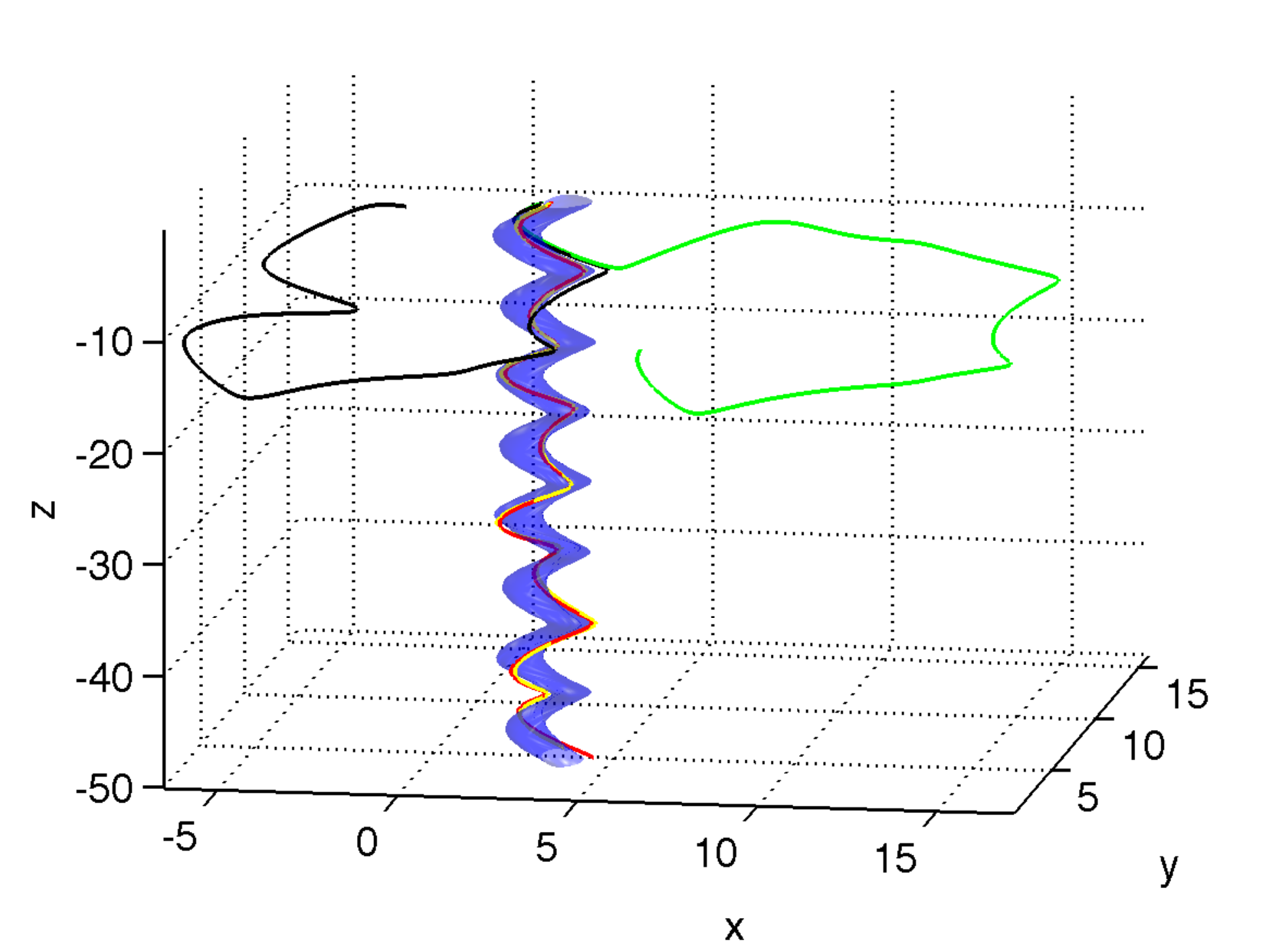}
\caption{Advection of tracers inside (red and yellow) and outside (green and
black) of the outermost closed shearline (blue) under the flow map
$F_{0}^{25}$. }
\label{fig:tracers_aper} 
\end{figure}

\subsubsection{Repelling hyperbolic barriers in the chaotically forced ABC flow}

Finally, we compute repelling hyperbolic barriers for the chaotically
forced ABC flow using steps H1-H5 of section \ref{section:comp}.
The Cauchy--Green strain tensor $C_{0}^{5}$ is computed over the
same plane family used in section \ref{sub:Repelling-hyperbolic-barriers}
for the time-periodic case. The grids $\mathcal{G}_{0}$ and $\mathcal{G}_{1}$,
as well as the admissible upper bound $\epsilon_{0}$ on the helicity
norm, are also selected the same as in section \ref{sub:Repelling-hyperbolic-barriers}.
Fig. \ref{fig:strain_aper} shows the final result, the set of extracted
repelling hyperbolic barriers in the chaotically forced case. 

\begin{figure}
\centering{}\includegraphics[width=0.3\textwidth]{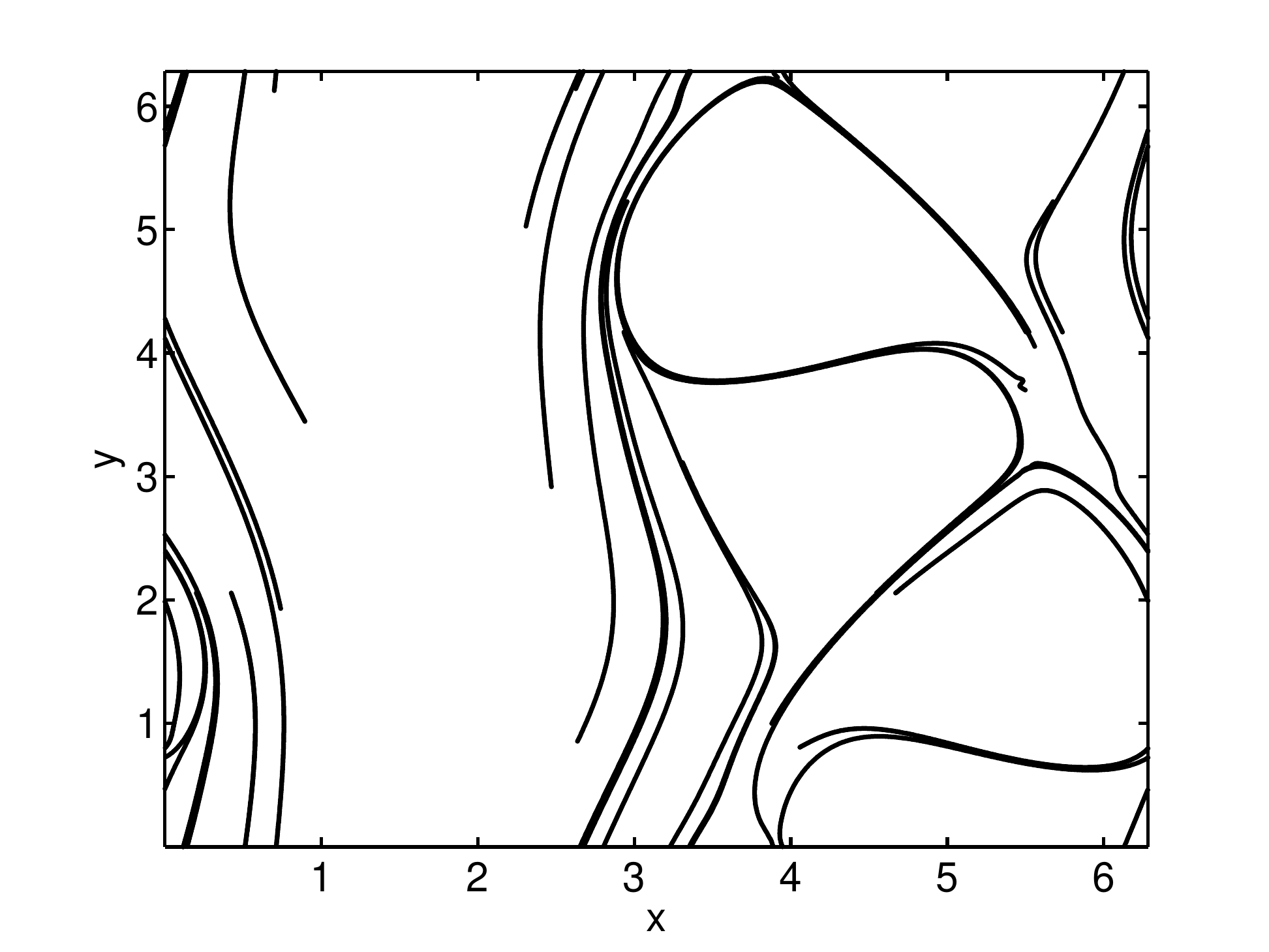}\includegraphics[width=0.7\textwidth]{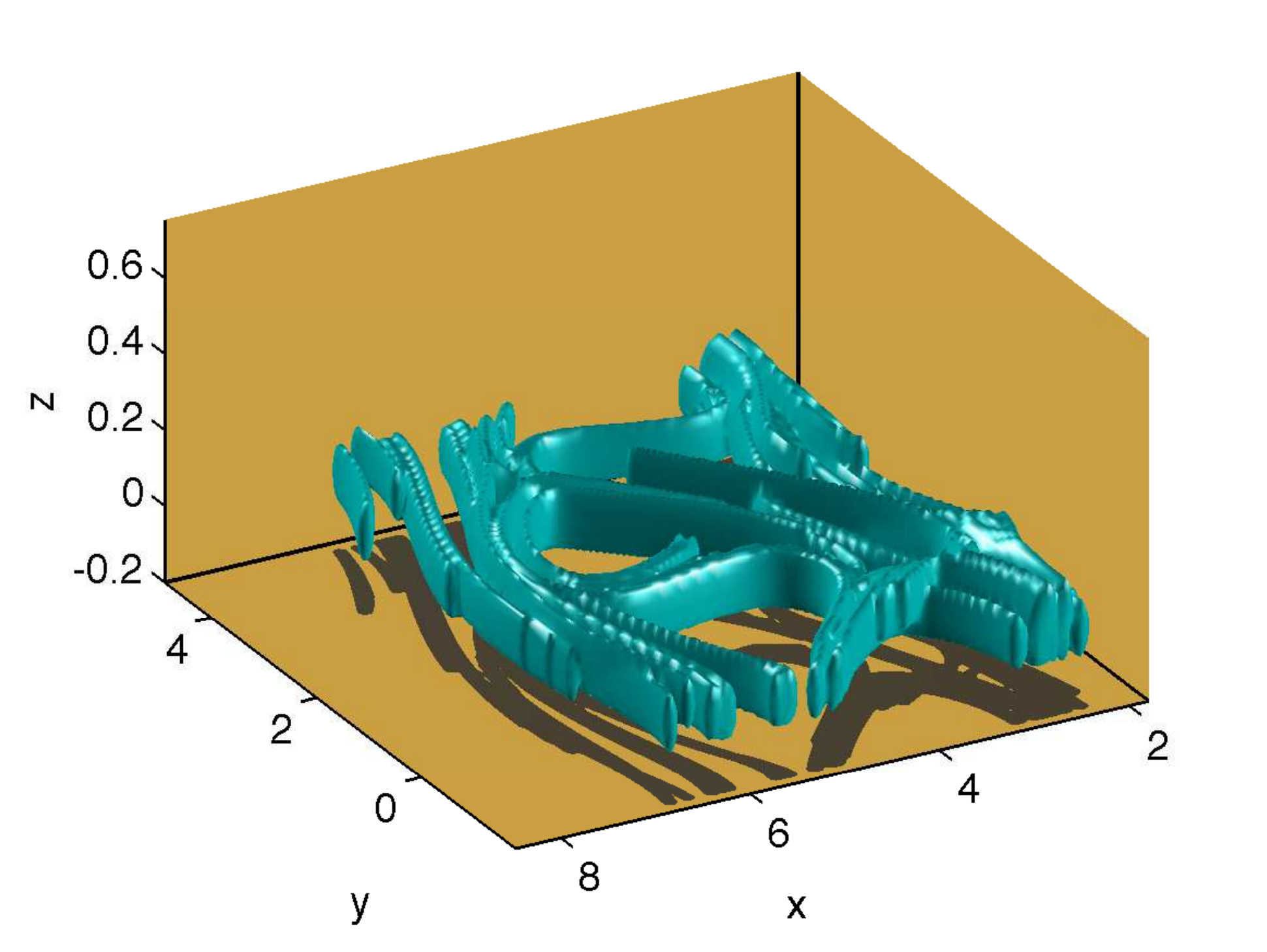}
\caption{Reduced strainlines on the $z=0$ planes (left) and repelling hyperbolic
LCSs at time $t_{0}=0$ (right) for the chaotically forced ABC flow,
reconstructed from the flow map $F_{0}^{5}$. }
\label{fig:strain_aper} 
\end{figure}

\section{Conclusions}

We have developed a unified theory of transport barriers for three-dimensional
unsteady flows. The barriers are attracting, repelling or shear LCS,
which are constructed to prevail as skeletons of material deformation
over a given finite time interval of observation. Out of general shear
LCSs, elliptic LCSs are distinguished by their tubular or toroidal
topology. Our approach renders all these LCSs as explicitly parametrized
surfaces with precisely understood impact on neighboring material
blobs. This is to be contrasted with alternative approaches that seek
the domains separated by transport barriers from various indicator
functions, without yielding specific dynamical information or a parametrization
for the barrier itself. 

Our approach closely reproduces known transport barriers in the steady
ABC flow, and provides similar results for time-periodic and time-aperiodic
version of the same flow. Remarkably, generalized KAM tori (Lagrangian
vortex rings) continue to exist in the general aperiodic case, providing
sharp boundaries for coherent toroidal islands in an otherwise chaotic
flow.

Although similar notions of multi-dimensional hyperbolic LCSs have
been used earlier \citep{var_theory}, the theory of elliptic LCS
as well as the computational methodology developed here for all types
of LCSs is new. Our notion of an ellkiptic LCS extends that arising
in the two-dimensional theory of shear barriers developed in \citep{geo_theory}.
This two-dimensional theory has identified highly coherent Lagrangian
eddies from satellite-mapped surface velocities in the Agulhas leakage
of the Southern Ocean \citep{Agulhas}. The direct analogy with the
two-dimensional theory promises similar results in the application
of the present techniques to three-dimensional numerical or experimental
flow data. 

\appendix

\section{Shear LCSs in unsteady, three-dimensional parallel shear flows}

\label{App:par_shear_flow}Consider the flow

\begin{equation}
\begin{split} & \dot{{x}}(t)=u(z,t),\\
 & \dot{{y}}(t)=v(z,t),\\
 & \dot{{z}}(t)=w(t),
\end{split}
\label{3D_parallel_gen}
\end{equation}
where the dependence of $u$, $v$, and $w$ on their arguments is
smooth but otherwise arbitrary. Trajectories of \eqref{3D_parallel_gen}
satisfy 
\begin{equation}
\begin{split} & x(t)=x_{0}+\int_{t_{0}}^{t_{0}+T}u(z(\tau),\tau)d\tau,\\
 & y(t)=y_{0}+\int_{t_{0}}^{t_{0}+T}v(z(\tau),\tau)d\tau,\\
 & z(t)=z_{0}+\int_{t_{0}}^{t_{0}+T}w(\tau)d\tau.
\end{split}
\label{3D_parallel_traj}
\end{equation}
We introduce the functions $a(z_{0},t_{0},T)$ and $b(z_{0},t_{0},T)$
as 
\begin{equation}
\begin{split} & a(z_{0},t_{0},T)=\int_{t_{0}}^{t_{0}+T}u_{z}(z(\tau),\tau)d\tau,\\
 & b(z_{0},t_{0},T)=\int_{t_{0}}^{t_{0}+T}u_{z}(z(\tau),\tau)d\tau.
\end{split}
\end{equation}

Suppressing the arguments of $a$ and $b$, we obtain the Cauchy-Green
strain tensor in the form 
\begin{equation}
C_{t_{0}}^{t_{0}+T}=\left(\begin{array}{ccc}
1 & 0 & a\\
0 & 1 & b\\
a & b & a^{2}+b^{2}+1
\end{array}\right).\label{CG_3D_parallel_general}
\end{equation}
We now show that the planes $z=k=const.$ are shear LCSs in the sense
of Definition 1, as obtained from an application of Theorem 1. To
do this, we use an expression for the angle $\phi$ that the vector
$v_{\pm}:=\xi_{2}\times n_{\pm}$ encloses with the vertical planes
$z=k$ for general $3$D flows.
\begin{lem}
\label{lem:angle_shear} Consider a general three-dimensional unsteady
flow, and letM
\[
v_{\pm}=\xi_{2}\times n_{\pm}=(\sin\phi\cos\theta,\sin\phi\sin\theta,\cos\phi).
\]
 Also, let $C_{ij}$ denote the $(i,j$)-th entry of the Cauchy-Green
strain tensor \textup{$C_{t_{0}}^{t_{0}+T}$}. We then have 
\begin{equation}
\begin{split} & C_{11}\sin^{2}\phi\cos^{2}\theta+C_{22}\sin^{2}\phi\sin^{2}\theta+C_{33}\cos^{2}\phi\\
 & +2\left(C_{12}\sin^{2}\phi\sin\theta\cos\theta+C_{13}\sin\phi\cos\phi\cos\theta+C_{23}\sin\phi\cos\phi\sin\theta\right)=\sqrt{\lambda_{1}\lambda_{3}}.
\end{split}
\label{eqn:angle_shear}
\end{equation}
\end{lem}
\begin{proof}
The two sides of equation \eqref{eqn:angle_shear} represent two different
ways of computing $\left<v_{\pm},C_{t_{0}}^{t_{0}+T}v_{\pm}\right>$.
The left-hand side is computed using the matrix elements of $C_{t_{0}}^{t_{0}+T}$.
To compute the same quantity and arrive at the quantity on the right-hand
side, recall first that $n_{\pm}=\alpha\xi_{1}\pm\beta\xi_{3}$, where
\[
\alpha=\sqrt{\frac{\sqrt{\lambda_{1}}}{\sqrt{\lambda_{1}}+\sqrt{\lambda_{3}}}},\qquad\beta=\sqrt{\frac{\sqrt{\lambda_{3}}}{\sqrt{\lambda_{1}}+\sqrt{\lambda_{3}}}}.
\]
 Hence $v_{\pm}=\alpha\xi_{3}\pm\beta\xi_{1}$, implying 
\begin{equation}
\left<v_{\pm},C_{t_{0}}^{t_{0}+T}v_{\pm}\right>=\alpha^{2}\lambda_{3}+\beta^{2}\lambda_{1}=\sqrt{\lambda_{1}\lambda_{3}},
\end{equation}
which proves the lemma. 
\end{proof}
For the unsteady parallel shear flow defined by \eqref{3D_parallel_gen},
one can verify that $\lambda=1.0$ is an eigenvalue of the Cauchy-Green
strain tensor \eqref{CG_3D_parallel_general} with eigenvector $\xi=(-\omega_{2}'(z_{0})/\omega_{1}'(z_{0}),1,0)$.
Moreover, symbolic computations in MATLAB show that the other eigenvalues
of \eqref{CG_3D_parallel_general} are greater than one, or less than
one. More specifically, another eigenvalue of $C_{t_{0}}^{t_{0}+T}$
is $\lambda=((a^{2}+b^{2})(a^{2}+b^{2}+4))(1/2)/2+a^{2}/2+b^{2}/2+1$,
which shows that as long as $a$ and $b$ are both nonzero, $C_{t_{0}}^{t_{0}+T}$
will have an eigenvalue greater than one. By incompressibility, another
eigenvalue is then less than one. Thus $\lambda_{2}=1$, and $\xi_{2}$
is parallel to the plane $z=k$.

Since $\xi_{2}$ is always orthogonal to $n_{\pm}$, to show that
$z=k$ is a shear LCS, it remains to argue that $v_{\pm}$ is also
tangent to $z=k$. Since the flow is incompressible, we conclude that
$\lambda_{1}\lambda_{3}=1$. As a result, the right-hand side of the
angle formula in Lemma \ref{lem:angle_shear} is one. Using our specific
form of the Cauchy-Green strain tensor \eqref{CG_3D_parallel_general},
the angle formula \eqref{eqn:angle_shear} becomes 
\begin{equation}
\cos\phi\left[\left(a^{2}+b^{2}\right)\cos\phi+2\sin\phi\left(a\cos\theta+b\sin\theta\right)\right]=1,
\end{equation}
which has $\phi=0$ as a solution. Therefore, we conclude that both
$\xi_{2}$ and $v_{\pm}=n_{\pm}\times\xi_{2}$ are tangent to the
plane $z_{0}=k$, which is therefore a shear LCS provided that $a(z_{0},t_{0},T)\neq0$
or $b(z_{0},t_{0},T)\neq0$.

\section{Evolution of LCS surface area}

\label{App:surface_area}

We consider how the surface area of an LCS changes under the flow
map. We have the following general result
\begin{lem}
\label{lem:Surf_area}{[}Surface area of a general material surface{]}
Let $\mathcal{M}(t)$ be a material surface, and $p(s_{1},s_{2})$
be a local parameterization of $\mathcal{M}(t_{0})$, where $(s_{1,}s_{2})$
lie in a connected open bounded subset $U\subset\mathbb{\mathbb{R}}^{2}$.
Then the surface area of $F_{t_{0}}^{t_{0}+T}\left(U\right)$ can
be computed as
\begin{equation}
S\left(F_{t_{0}}^{t_{0}+T}\left(U\right)\right)=\int_{U}\left|\det\left(\nabla F_{t_{0}}^{t_{0}+T}\right)\right|\cdot\sqrt{\left|\langle p_{s_{1}}\times p_{s_{2}},\left(C_{t_{0}}^{t_{0}+T}\right)^{-1}p_{s_{1}}\times p_{s_{2}}\rangle\right|}ds_{1}ds_{2},\label{eq:surf_area_material}
\end{equation}
 where $p_{s_{i}}:=\frac{\partial p}{\partial s_{i}}$.\end{lem}
\begin{proof}
Since $F_{t_{0}}^{t_{0}+T}\circ p$ is a parameterization of $\mathcal{M}(t_{0}+T)$,
the vectors $\nabla F_{t_{0}}^{t_{0}+T}p_{s_{i}}$ are tangent to
$F_{t_{0}}^{t_{0}+T}\left(U\right)$. The advected surface area, by
definition, is then 
\[
S\left(F_{t_{0}}^{t_{0}+T}\left(U\right)\right)=\int_{U}\sqrt{\left|\langle\nabla F_{t_{0}}^{t_{0}+T}p_{s_{1}}\times\nabla F_{t_{0}}^{t_{0}+T}p_{s_{2}},\nabla F_{t_{0}}^{t_{0}+T}p_{s_{1}}\times\nabla F_{t_{0}}^{t_{0}+T}p_{s_{2}}\rangle\right|}ds_{1}ds_{2}.
\]
This implies formula \ref{eq:surf_area_material} based on the general
identity $Mv\times Mu=\left(\det M\right)M^{-T}v\times u$, which
holds for any invertible square matrix $M$ and vectors $u$ and $v$.
$ $\end{proof}
\begin{prop}
\label{prop:Surf_area_shear} {[}Surface area of LCSs{]} Let $p(s_{1},s_{2})$
be a parameterization of a material surface $\mathcal{M}(t)\subset\mathbb{R}^{3}$
over the time interval $[t_{0},t_{0}+T]$.

(i) Suppose that $\mathcal{M}(t)$ is a repelling hyperbolic LCS.
Then we have
\[
S\left(F_{t_{0}}^{t_{0}+T}\left(U\right)\right)=\int_{U}\left|\det\left(\nabla F_{t_{0}}^{t_{0}+T}\right)\right|\cdot\|p_{s_{1}}\times p_{s_{2}}\|\frac{1}{\sqrt{\lambda_{3}}}ds_{1}ds_{2}.
\]

(ii) Suppose that $\mathcal{M}(t)$ is an attracting hyperbolic LCS.
Then we have 
\[
S\left(F_{t_{0}}^{t_{0}+T}\left(U\right)\right)=\int_{U}\left|\det\left(\nabla F_{t_{0}}^{t_{0}+T}\right)\right|\cdot\|p_{s_{1}}\times p_{s_{2}}\|\frac{1}{\sqrt{\lambda_{1}}}ds_{1}ds_{2}.
\]

(iii) Suppose that $\mathcal{M}(t)$ is a shear LCS. Then we have
\[
S\left(F_{t_{0}}^{t_{0}+T}\left(U\right)\right)=\int_{U}\left|\det\left(\nabla F_{t_{0}}^{t_{0}+T}\right)\right|\cdot\|p_{s_{1}}\times p_{s_{2}}\|\frac{1}{\sqrt[4]{\lambda_{1}\lambda_{3}}}ds_{1}ds_{2}.
\]

In the special case of a volume-preserving flow, we have 
\[
S\left(F_{t_{0}}^{t_{0}+T}\left(U\right)\right)=\int_{U}\|p_{s_{1}}\times p_{s_{2}}\|\sqrt[4]{\lambda_{2}}ds_{1}ds_{2}.
\]
\end{prop}
\begin{proof}
We proof the result for $ $\emph{(iii)}, as cases \emph{(i) }and
\emph{(ii)} are similar. The shear vector field $n_{\pm}$ has unit
length, and hence $\langle p_{s_{1}}\times p_{s_{2}},p_{s_{1}}\times p_{s_{2}}\rangle=\|p_{s_{1}}\times p_{s_{2}}\|^{2}$
. Observe that

\begin{equation}
\begin{split} & S\left(F_{t_{0}}^{t_{0}+T}\left(U\right)\right)=\int_{U}\left|\det\left(\nabla F_{t_{0}}^{t_{0}+T}\right)\right|\cdot\sqrt{\left|\langle p_{s_{1}}\times p_{s_{2}},\left(C_{t_{0}}^{t_{0}+T}\right)^{-1}p_{s_{1}}\times p_{s_{2}}\rangle\right|}ds_{1}ds_{2}\\
 & =\int_{U}|\det\left(\nabla F_{t_{0}}^{t_{0}+T}\right)|\cdot\|p_{s_{1}}\times p_{s_{2}}\|\sqrt{\left|\langle n_{\pm},\left(C_{t_{0}}^{t_{0}+T}\right)^{-1}n_{\pm}\rangle\right|}ds_{1}ds_{2}.
\end{split}
\label{eq:shear_area}
\end{equation}

Using the definition of $n_{\pm}$, one sees that $\langle n_{\pm},\left(C_{t_{0}}^{t_{0}+T}\right)^{-1}n_{\pm}\rangle=\frac{1}{\sqrt{\lambda_{1}\lambda_{3}}}$.
Substituting this identity into \eqref{eq:shear_area} proves (iii)
of Proposition \ref{prop:Surf_area_shear}. We can deduce the result
for repelling and attracting LCSs similarly, using the fact that $\langle\xi_{3},\left(C_{t_{0}}^{t_{0}+T}\right)^{-1}\xi_{3}\rangle=\frac{1}{\lambda_{3}}$
for a repelling LCS, and $\langle\xi_{1},\left(C_{t_{0}}^{t_{0}+T}\right)^{-1}\xi_{1}\rangle=\frac{1}{\lambda_{1}}$
for an attracting LCS.

\[
\]

\end{proof}
Proposition \ref{prop:Surf_area_shear} shows that the final surface
area along a shear LCS in incompressible flow is obtained by integrating
the initial surface element $\|p_{s_{1}}\times p_{s_{2}}\|$ weighted
by $\sqrt[4]{\lambda_{2}}$ . In \ref{App:par_shear_flow}, we showed
that $ $$\lambda_{2}=1$ golds globally in space and time, and hence
the corresponding shear LCS surface area is conserved. For the steady
ABC flow, we find that $\lambda_{2}$ computed over elliptic LCSs
oscillates around one. Fig. \ref{Fig:lam_2_shearline} shows this
along a specific closed reduced shearline. 

\textbf{}
\begin{figure}
\begin{centering}
\includegraphics[width=0.4\textwidth]{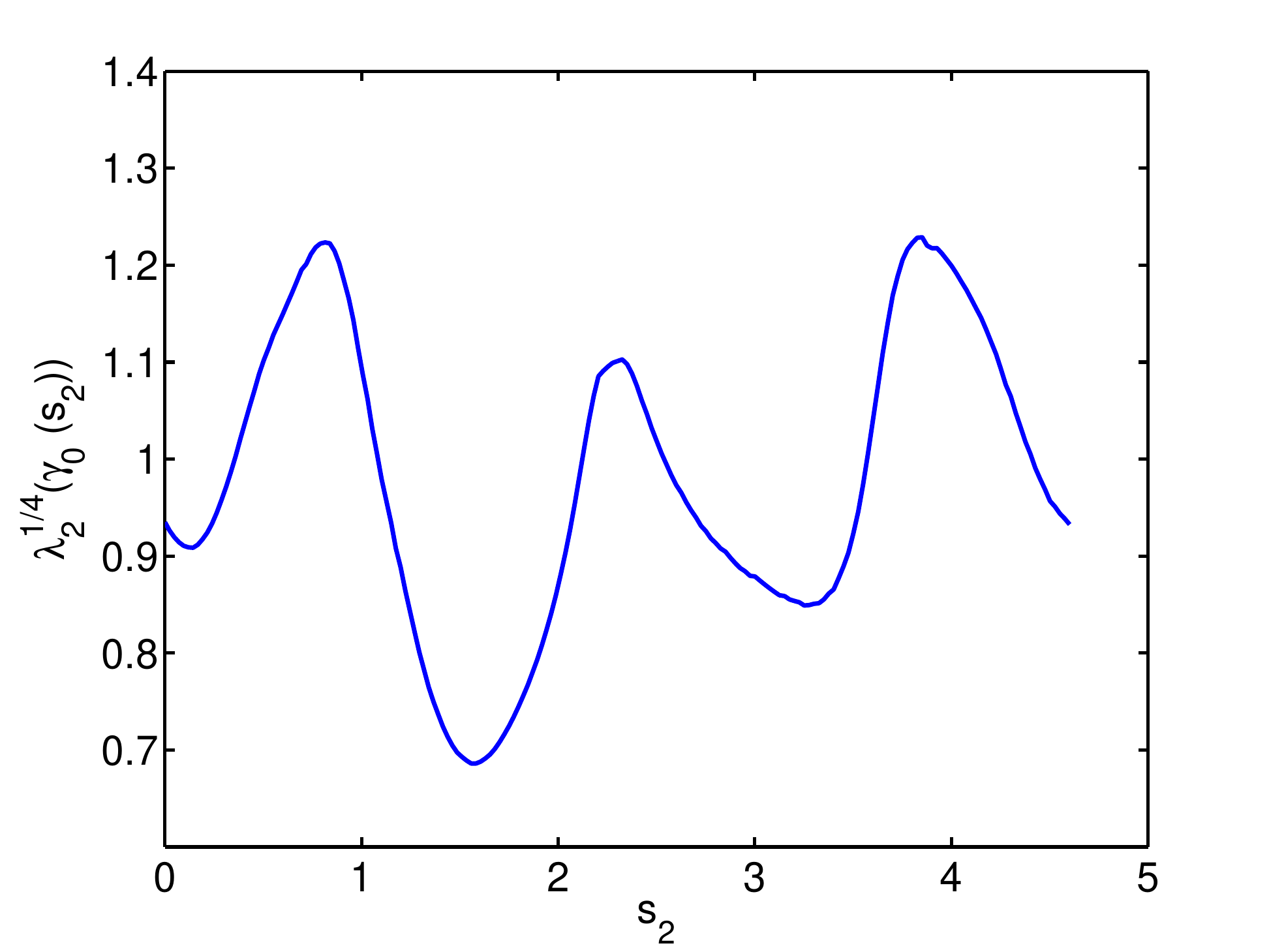}
\par\end{centering}

\textbf{\caption{\textbf{$\sqrt[4]{\lambda_{2}}$ }computed over a closed reduced shearline
$\gamma_{0}$ for the steady ABC flow, obtained from the flow map
$F_{0}^{150}$.}
}\label{Fig:lam_2_shearline}
\end{figure}

\section{Proof of Theorem 1
\[
\]
}

\label{App:proof_main_thm} We first need the following proposition
which allows us to compute the normal repulsion and tangential shear
in terms of $C_{t_{0}}^{t_{0}+T}$ and $n_{0}$.
\begin{prop}
The quantities $\rho$ and $\sigma$ can be expressed as

\begin{equation}
\begin{split} & \rho(x_{0},n_{0})=\frac{1}{\sqrt{\left<n_{0},\left[C_{t_{0}}^{t_{0}+T}(x_{0})\right]^{-1}n_{0}\right>}},\\
 & \sigma(x_{0},n_{0})=\sqrt{\left<n_{0},C_{t_{0}}^{t_{0}+T}(x_{0})n_{0}\right>-\frac{1}{\left<n_{0},\left[C_{t_{0}}^{t_{0}+T}(x_{0})\right]^{-1}n\right>}}.
\end{split}
\end{equation}
\end{prop}
\begin{proof}
The proof of the first identity can be found in \citep{var_theory}.
As for the second identity, we make use of the following formula for
the unit normal of a multi-dimensional material surface \citep{var_theory}:
\begin{equation}
n_{t}=\frac{\left[\left(\nabla F^{t_{0}}_{t}\right)^{*}n_{0}\right]}{\left|\left(\nabla F^{t_{0}}_{t}\right)^{*}n_{0}\right|}.
\end{equation}
We then obtain 
\begin{eqnarray*}
\left[\sigma(x_{0},n_{0})\right]^{2} & = & \left\vert \nabla F_{t_{0}}^{t_{0}+T}(x_{0})n_{0}-\left\langle n_{t},\nabla F_{t_{0}}^{t_{0}+T}(x_{0})n_{0}\right\rangle n_{t}\right\vert \\
 & = & \left\langle n_{0},\left(\nabla F_{t_{0}}^{t_{0}+T}(x_{0})\right)^{\ast}\nabla F_{t_{0}}^{t_{0}+T}(x_{0})n_{0}\right\rangle -2\left\langle \nabla F_{t_{0}}^{t_{0}+T}(x_{0})n_{0},\left\langle n_{t},\nabla F_{t_{0}}^{t_{0}+T}(x_{0})n_{0}\right\rangle n_{t}\right\rangle \\
 &  & +\left\langle \left\langle n_{t},\nabla F_{t_{0}}^{t_{0}+T}(x_{0})n_{0}\right\rangle n_{t},\left\langle n_{t},\nabla F_{t_{0}}^{t_{0}+T}(x_{0})n_{0}\right\rangle n_{t}\right\rangle \\
 & = & \left\langle n_{0},C_{t_{0}}^{t_{0}+T}(x_{0})n_{0}\right\rangle -\left[\rho(x_{0},n_{0})\right]^{2}\\
 & = & \left\langle n_{0},C_{t_{0}}^{t_{0}+T}(x_{0})n_{0}\right\rangle -\frac{1}{\left\langle n_{0},\left[C_{t_{0}}^{t_{0}+T}(x_{0})\right]^{-1}n_{0}\right\rangle },
\end{eqnarray*}
which proves the proposition. 
\end{proof}
We are now ready to prove Theorem 1. Let $S$ be a hyperbolic LCS
and $x_{0}$ a point on $S.$ Thus, by definition, for any other surface
$\tilde{S}$ passing through $x_{0}$ such that $T_{x_{0}}\tilde{S}\neq T_{x_{0}}S$,
the normal repulsion rate $\rho(x_{0},n_{0}^{\tilde{S}}$) along $\tilde{S}$
at $x_{0}$ is smaller than the normal repulsion rate $\rho(x_{0},n_{0}^{S}$)
along $S$ at $x_{0}$. Thus, at each point $x_{0}$ on $S$, the
quantity $\rho(x_{0},n_{0}$) is maximized with respect to changes
in $n_{0}$. Thus, we want to solve the following constrained optimization
problem: maximize $\rho(x_{0},n_{0}$) with respect to $n_{0}$ under
the constraint that $\|n_{0}\|=1$. To this end, expand $n_{0}$ in
the eigen-basis $\left\{ \xi_{1,},\xi_{2},\xi_{3}\right\} $ of the
Cauchy--Green strain tensor $C_{t_{0}}^{t_{0}+T}(x_{0}):$ 
\[
n_{0}=\sum_{i=1}^{3}n_{i}\xi_{i}.
\]
We then have 
\[
\rho(x_{0},n_{0})=\left(\sum_{i=1}^{3}\frac{n_{i}^{2}}{\lambda_{i}}\right)^{-1},
\]
and setting the gradient of $\rho(x_{0},n_{0}$) proportional to the
gradient of the constraint $\|n_{0}\|=1$ yields

\begin{equation}
\begin{split} & -2\rho^{2}\lambda_{2}\lambda_{3}n_{1}=\gamma n_{1},\\
 & -2\rho^{2}\lambda_{1}\lambda_{3}n_{2}=\gamma n_{2}.\\
 & -2\rho^{2}\lambda_{1}\lambda_{2}n_{n}=\gamma n_{3},
\end{split}
\label{lagrange_repulsion}
\end{equation}
where $\gamma$ is an appropriate constant. Thus, assuming that $\lambda_{1}>0$,
it follows that $ $two of the coordinates $n_{1},n_{2},n_{3}$ must
be zero. Therefore, the only extremum directions of the normal repulsion
rate are the eigenvectors $\xi_{i}$ of the Cauchy-Green strain tensor.
Since $\rho(x_{0},\xi_{i})=\sqrt{\lambda_{i}}$, it follows that $n_{0}=\xi_{3}$
is the global maximum of the normal repulsion $\rho$ with respect
to changes in $n_{0}$. Therefore, a repelling hyperbolic LCS is necessarily
orthogonal to $\xi_{3}$. The proof of statement $(ii)$ is analogous,
but involves the global minimum of the normal repulsion rate.

The prove statement (iii), we need to find the maximizing normal directions
$n_{0}$ of the tangential shear $\sigma(x_{0},n_{0})$ under the
constraint that $\|n_{0}\|=1$. We again represent $n_{0}$ in the
eigen-basis $\left\{ \xi_{1,},\xi_{2},\xi_{3}\right\} $ of $C_{t_{0}}^{t_{0}+T}(x_{0})$
as 
\begin{equation}
n_{0}=\sum_{i=1}^{3}n_{i}\xi_{i},\qquad\sum_{i=1}^{3}n_{i}^{2}=1,\label{ni}
\end{equation}
and seek to maximize $\sigma(n_{1,},n_{2},n_{3}):=\sigma(x_{0,}n_{0})$
subject to the constraint $\sum n_{i}^{2}=1$. Note that 
\[
\sigma(n_{1},n_{2},n_{3})=\sqrt{{\sum\lambda_{i}n_{i}^{2}-\frac{{1}}{\sum\lambda_{i}^{-1}n_{i}^{2}}}}.
\]
Setting the gradient of $\sigma(n_{1},n_{2},n_{3})$ proportional
to the gradient of the constraint $\sum n_{i}^{2}=1$ leads to the
equations 
\[
\frac{{1}}{2\sigma}\left(2n_{i}\lambda_{i}+\frac{{2n_{i}}}{\lambda_{i}\left(\sum\lambda_{i}^{-1}n_{i}^{2}\right)^{2}}\right)=2\gamma n_{i},
\]
which can also be rewritten as 
\[
\frac{{1}}{2\sigma}\left(2n_{i}\lambda_{i}+\frac{{2n_{i}}}{\lambda_{i}\left(\sum\lambda_{i}^{-1}n_{i}^{2}\right)^{2}}\right)=2\gamma n_{i},
\]
 or, equivalently, 

\begin{equation}
n_{i}\left[\left(\sum_{j=1}^{3}\frac{n_{j}^{2}}{\lambda_{j}}\right)^{2}\lambda_{i}\mathbf{+}\frac{1}{\lambda_{i}}-2\sigma\gamma\right]=0,\qquad i=1,2,3.\label{niplus}
\end{equation}

Assume now that for some index $i$, we have $n_{i}\neq0$, i.e.,
the $i^{\mathrm{th}}$ component of the unknown normal vector $n_{0}$
of the shear LCS is nonzero. In that case, we must have 
\begin{equation}
\frac{2\sigma\gamma-\frac{1}{\lambda_{i}}}{\lambda_{i}}=K\overset{def.}{=}\left(\sum_{j=1}^{3}\frac{n_{j}^{2}}{\lambda_{j}}\right)^{2},\label{niplus1}
\end{equation}
where the constant $K>0$ is the same for any choice of $i$. Taking
the square of equation (\ref{niplus1}) gives 
\[
\lambda_{i}^{2}-\frac{2\sigma\gamma}{K}\lambda_{i}+\frac{1}{K}=0,
\]
and hence there can be \emph{at most two} distinct $\lambda_{i}$
values for which (\ref{niplus1}) holds. Also note that there has
to be \emph{at least two} distinct $\lambda_{i}$ values for which
(\ref{niplus1}) holds, otherwise \textbf{$n_{0}$ }would be an eigenvector,
and hence a local minimizer of the tangential shear . We conclude
that there exist \emph{precisely two} eigenvalues, $\lambda_{k}$
and $\lambda_{l}>\lambda_{k}$, of $C_{t_{0}}^{t_{0}+T}$ that satisfy
(\ref{niplus1}).

All eigenvalues of $C_{t_{0}}^{t_{0}+T}(x_{0})$ are simple by assumption.
Therefore, by our argument above, $n_{0}$ must be of the form 
\begin{equation}
\, n_{0}=n_{k}\xi_{k}+n_{l}\xi_{l},\qquad n_{k}^{2}+n_{l}^{2}=1.\label{eq:n0rest}
\end{equation}
Substituting the expression (\ref{eq:n0rest}) into (\ref{niplus1})
with $i=k$ and $i=l$, eliminating the common constant $2\sigma\gamma$
from the resulting two equations, and using the notation 
\[
a^{2}=n_{k}^{2},\qquad b^{2}=n_{l}^{2},
\]
we obtain the system of equations 
\begin{eqnarray}
\lambda_{k}\left(\frac{a^{2}}{\lambda_{k}}+\frac{b^{2}}{\lambda_{l}}\right)^{2}+\frac{1}{\lambda_{k}} & = & \lambda_{l}\left(\frac{a^{2}}{\lambda_{k}}+\frac{b^{2}}{\lambda_{l}}\right)^{2}+\frac{1}{\lambda_{l}},\\
a^{2}+b^{2} & = & 1,
\end{eqnarray}
for the unknowns $a^{2}$ and $b^{2}$. The solution of these equations
is given by 
\[
a^{2}=\frac{\sqrt{\lambda_{k}}}{\sqrt{\lambda_{k}}+\sqrt{\lambda_{l}}},\qquad b^{2}=\frac{\sqrt{\lambda_{l}}}{\sqrt{\lambda_{k}}+\sqrt{\lambda_{l}}}.
\]

Thus $n_{0}$ must take the more specific form 
\begin{eqnarray}
n_{0} & = & a\xi_{k}+b\xi_{l},\qquad\lambda_{k}<\lambda_{l},\nonumber \\
a^{2} & = & \frac{\sqrt{\lambda_{k}}}{\sqrt{\lambda_{k}}+\sqrt{\lambda_{l}}},\qquad b^{2}=\frac{\sqrt{\lambda_{l}}}{\sqrt{\lambda_{k}}+\sqrt{\lambda_{l}}},\label{eq:n0final}
\end{eqnarray}
for some choice of $k,l\in\left\{ 1,2,3\right\} .$ We now check which
of these extrema are indeed local maxima. Computing the tangential
shear for expressions \eqref{eq:n0final} yields 

\begin{eqnarray}
\sigma(x_{0},n_{0}) & = & \sqrt{\left\langle n_{0},C_{t_{0}}^{t_{0}+T}n_{0}\right\rangle -\frac{1}{\left\langle n_{0},\left[C_{t_{0}}^{t_{0}+T}\right]^{-1}n_{0}\right\rangle }}\nonumber \\
 & = & \sqrt{a^{2}\lambda_{k}+b^{2}\mathbf{\lambda}_{l}-\frac{\lambda_{k}\lambda_{l}}{a^{2}\lambda_{l}+b^{2}\lambda_{k}}}\nonumber \\
 & = & \left\vert \sqrt{\lambda_{l}}-\sqrt{\lambda_{k}}\right\vert .\label{eq:muevaluated}
\end{eqnarray}

Next we prove that $k=1$ and $l=3$ must hold for the normal $n_{0}$
in formula (\ref{eq:n0final}). Assume the contrary, i.e., assume
that the pair of eigenvalues $(\lambda_{k},\lambda_{l})$ in formula
(\ref{eq:n0final}) does not coincide with the pair $(\lambda_{1},\lambda_{2})$.
We only consider the case of $\lambda_{k}\neq\lambda_{1}$, because
the case of $\lambda_{l}\neq\lambda_{3}$ can be handled in an identical
fashion. Assuming $\lambda_{k}\neq\lambda_{1}$, define the unit normal
\[
\hat{n}_{0}=\sqrt{\frac{\sqrt{\lambda_{1}}}{\sqrt{\lambda_{1}}+\sqrt{\lambda_{3}}}}\mathbf{\xi}_{1}+\sqrt{\frac{\sqrt{\lambda_{3}}}{\sqrt{\lambda_{1}}+\sqrt{\lambda_{3}}}}\xi_{3}.
\]
Note that $\hat{n}_{0}\neq n_{0}$ by our assumption, and by formula
(\ref{eq:muevaluated}), we have 
\[
\sigma(x_{0},n_{0}(x_{0}))=\left|\sqrt{\lambda_{l}}-\sqrt{\lambda_{k}}\right|<\left|\sqrt{\lambda_{n}}-\sqrt{\lambda_{1}}\right|=\sigma(x_{0},\hat{n}_{0}),
\]
which contradicts our maximality assumption for shear LCS in Definition
1 (namely that a shear LCS has tangential shear no less than any perturbations
of its normal direction $n_{0}$).

We have, therefore, obtained that for any shear LCS, the normal vector
$n_{0}(x_{0}$) featured in (\ref{eq:n0final}) must necessarily be
of the more specific form 
\begin{equation}
n_{0}=\pm\sqrt{\frac{\sqrt{\lambda_{1}}}{\sqrt{\lambda_{1}}+\sqrt{\lambda_{3}}}}\xi_{1}\pm\sqrt{\frac{\sqrt{\lambda_{3}}}{\sqrt{\lambda_{1}}+\sqrt{\lambda_{3}}}}\xi_{3},\label{eq:n0proof}
\end{equation}
where $\lambda_{1}$ and $\lambda_{n}$ are multiplicity-one eigenvalues
of the Cauchy-Green strain tensor $C_{t_{0}}^{t_{0}+T}(x_{0})$, and
the two $\pm$ signs can be chosen independently of one another. All
in all, formula (\ref{eq:n0proof}) defines two linearly independent
unit normal directions, corresponding to maximal positive and maximal
negative shear. This proves that a shear LCS is necessarily orthogonal
to either $n_{+}$ or $n_{-}$.

\section{Proof of Theorem 2 and relation to Frobenius Integrability \label{App:proof_thm2} }

\subsection{Proof of Theorem 2}

For a general three-dimensional vector field $v$, consider the problem
of finding a surface $S$ orthogonal to $v$. The following proposition
shows that a necessary condition for the existence of $S$ is that
the helicity of $v$ 
\begin{equation}
H_{v}(x)=\left<\nabla\times v,v\right>
\end{equation}
must vanish on $S$. This fact was pointed out in \citep{Palm}; here
we provide an alternative proof using Stokes' Theorem.
\begin{prop}
\label{Helicity_prop} Let $v$ be a smooth vector field in $\mathbb{R}^{3}$
and $S$ a surface orthogonal to $v$. Then for any $x\in S$, the
helicity of $v$must vanish, i.e., 
\end{prop}
\begin{equation}
H_{v}(x):=\left<\nabla\times v(x),v(x)\right>=0.
\end{equation}

\begin{proof}
Consider an open neighborhood $D\subset S$ of $x$ in $S$ . By Stokes'
Theorem, we have that 
\begin{equation}
\int_{D}\left(\nabla\times v\right)\cdot n=\int_{C}v\cdot dr.\label{Stokes}
\end{equation}
Since $v$ is orthogonal to $S$, the integral on the right-hand side
of \eqref{Stokes} is zero. Thus, since $v=\left<v,n\right>n$, we
have 
\begin{equation}
\int_{D}\frac{1}{\left<v,n\right>}H_{v}\, dA=0.
\end{equation}
Since $D$ was arbitrary, $H_{v}$ must vanish on $S$. 
\end{proof}
Theorem 2 then follows directly from Theorem 1 and Proposition \ref{Helicity_prop}.

\subsection{Relation to Frobenius Integrability}

We can rephrase the problem of computing a surface orthogonal to $\xi_{3}$
for repelling hyperbolic, $\xi_{1}$ for attracting hyperbolic, and
$n_{\pm}$ for shear LCS as finding surfaces tangent to $\left\{ \xi_{1},\xi_{2}\right\} $,
$\left\{ \xi_{2},\xi_{3}\right\} $ and $\left\{ \xi_{2},n_{\pm}\times\xi_{2}\right\} $,
respectively. The problem of finding surfaces tangent to two specified
vector fields is then related to the Frobenius Integrability Theorem\textbf{
}\citep{AMR}. 

As a special case, this theorem states that if $X$ and $Y$ are two
vector fields in $\mathbb{R}^{3}$, then necessary and sufficient
conditions for the existence of a\emph{ foliation }of $\mathbb{R}^{3}$
by surfaces tangent $X$ and $Y$ is that
\begin{equation}
\left[X,Y\right]\in\text{Span}\left\{ X,Y\right\} \label{eq:Frob}
\end{equation}
 
\[
\]

In our context, we do not seek to find global foliations of $\mathbb{R}^{3}$,
but only individual, isolated surfaces. Nevertheless, as we show in
this section, the Frobenius condition \eqref{eq:Frob} is still a
necessary, albeit not sufficient condition for the existence of such
surfaces. 

If $X,Y$ and $Z$ are smooth vector fields in $\mathbb{R}^{3}$,
consider both the Frobenius and helicity conditions 
\begin{equation}
\begin{split} & F_{X,Y,Z}=\left<\left[X,Y\right],Z\right>=0,\\
 & H_{V}\left<\nabla\times Z,Z\right>=0.
\end{split}
\end{equation}
We show in Proposition \ref{prop:Frob_hel_equiv} that the zero sets
of $F$ and $H_{Z}$ coincide, provided that $X,Y$ and $Z$ are pairwise
orthogonal vector fields. Combined with Theorem 2, Proposition \ref{prop:Frob_hel_equiv}
then shows that \eqref{eq:Frob} is a necessary condition. We make
use of the following lemma, which states that the zero sets of $F$
and $H_{Z}$ are invariant under a nonlinear rescaling of the three
vector fields involved.
\begin{lem}
\label{lem:scaling} Let $X,Y$ and $Z$ be pairwise orthogonal vector
fields in $\mathbb{R}^{3}$, and $\phi_{1}$, $\phi_{2}$, $\phi_{3}$
be nonzero scalar function on $\mathbb{R}^{3}$. Then we have that
\begin{equation}
\begin{split} & F_{\phi_{1}X,\phi_{2}Y,\phi_{3}Z}=\phi_{1}\phi_{2}\phi_{3}F_{X,Y,Z},\\
 & H_{\phi_{3}Z}=\phi_{3}^{2}H_{Z}.
\end{split}
\end{equation}
In particular, the zero sets of $F_{\phi_{1}X,\phi_{2}Y,\phi_{3}Z}$
and $H_{\phi_{3}Z}$ coincide with those of $F_{X,Y,Z}$ and $H_{Z}$,
respectively. \end{lem}
\begin{proof}
By definition, we have that 
\begin{equation}
\begin{split} & F_{\phi_{1}X,\phi_{2}Y,\phi_{3}Z}=\left<\left[\phi_{1}X,\phi_{2}Y\right],\phi_{3}Z\right>\\
 & =\phi_{3}\left<D(\phi_{1}X)Y-D(\phi_{2}Y)X,Z\right>=\phi_{3}\left<\left(\nabla\phi_{1}\right)X^{T}Y+\phi_{1}\phi_{2}DXY-\phi_{1}\left(\nabla\phi_{2}\right)Y^{T}X-\phi_{1}\phi_{2}DYX,Z\right>.
\end{split}
\end{equation}
However, $X^{T}Y=Y^{T}X=0$ by our orthogonality assumption, and hence
\begin{equation}
F_{\phi_{1}X,\phi_{2}Y,\phi_{3}Z}=\phi_{1}\phi_{2}\phi_{3}\left<\left[X,Y\right],Z\right>=\phi_{1}\phi_{2}\phi_{3}F_{X,Y,Z}.
\end{equation}
As for the claim on the helicity, note that 
\begin{equation}
\begin{split} & H_{\phi_{3}Z}=\left<\nabla\times(\phi_{3}Z),\phi_{3}Z\right>\\
 & =\phi_{3}\left<\nabla\phi_{3}\times Z+\phi_{3}\left(\nabla\times Z\right),Z\right>=\phi_{3}^{2}\left<\nabla\times Z,Z\right>
\end{split}
\end{equation}
\end{proof}
\begin{prop}
\label{prop:Frob_hel_equiv} Let $X,Y$ and $Z$ be a smoothly varying,
pairwise orthogonal family of vector fields in $\mathbb{R}^{3}$.
Then the zero set of $F=\left<\left[X,Y\right],Z\right>$ coincides
with the zero set of $H_{Z}\left<\nabla\times Z,Z\right>=0$. \end{prop}
\begin{proof}
By Lemma \ref{lem:scaling}, it suffices to assume that $X,Y$ and
$Z$ is an orthonormal family of vector fields. Assume that 
\begin{equation}
\left<\left[X,Y\right],Z\right>=\left<\left(\nabla X\right)Y-\left(\nabla Y\right)X,Z\right>=0.\label{Frob_proof_eqn}
\end{equation}
Then differentiating the orthonormality assumptions $\left<X,Y\right>=\left<X,Z\right>=0,\|X\|=\|Y\|=\|Z\|=1$,
we obtain 
\begin{equation}
\begin{split} & \left(\nabla X\right)^{T}Z+\left(\nabla Z\right)^{T}X=0,\\
 & \left(\nabla Y\right)^{T}Z+\left(\nabla Z\right)^{T}Y=0,
\end{split}
\end{equation}
which, after substitution into the Frobenius relation \eqref{Frob_proof_eqn},
yields 
\begin{equation}
\left<X,\left[\nabla Z-\nabla Z^{T}\right]Y\right>=0
\end{equation}
Now we recall the following general identity for vector fields in
$\mathbb{R}^{3}$: 
\begin{equation}
\left[\nabla a-\nabla a^{T}\right]b=\left(\nabla\times a\right)\times b.
\end{equation}
Applying this to $a=Z$ and $b=Y$, we obtain 
\begin{equation}
\left<X,\left(\nabla\times Z\right)\times Y\right>=0.
\end{equation}
Finally, using the identity $\left(a\times b\right)\cdot c=\left(b\times c\right)\cdot a$
with $a=\nabla\times Z$, $b=Y$, and $c=X$, we obtain that \eqref{Frob_proof_eqn}
is equivalent to 
\begin{equation}
\left<\nabla\times Z,Z\right>=0
\end{equation}
as claimed.\end{proof}
\begin{rem}
Proposition \ref{prop:Frob_hel_equiv} shows that the helicity conditions
in Theorem 2 are equivalent to the Frobenius conditions $\left<\left[\xi_{1},\xi_{2}\right],\xi_{3}\right>=0$,
$\left<\left[\xi_{2},\xi_{3}\right],\xi_{1}\right>=0$, and $\left<\left[\xi_{2},n_{\pm}\times\xi_{2}\right],n_{\pm}\right>=0$
for repelling hyperbolic, attracting hyperbolic, and shear LCSs, respectively. 
\end{rem}
~~~~~
\begin{rem}
Frobenius Integrability Theorem applied to the existence of tangent
foliations provides a necessary and sufficient condition. By contrast,
the zero helicity condition (and its equivalent Frobenius condition)
are only necessary conditions for the existence of isolated surfaces
normal to a vector field $v$. For example, let $v(x,y,z)=(y,z,x)$.
Then $H_{v}(x,y,z)=\left<\nabla\times v,v\right>=-y-z-x$, which has
a plane as its zero set, but this plane is not orthogonal to $v$.
Thus $H_{v}(x,y,z)=0$ is not sufficient for the existence of a surface
orthogonal to $v$.
\end{rem}

\section{Proof of Theorem 3\label{App:proof_thm3} }

At any point $x_{0}\in\Pi(s_{1})$, a tangent vector to such a potential
intersection curve $\gamma_{s_{1}}=\Pi(s_{1})\cap\mathcal{M}(t_{0}$)
must be orthogonal both to the unit normal vector $n_{\Pi(s_{1})}$
of $\Pi(s_{1})$, and either to $\xi_{3}$ (repelling hyperbolic barriers),
to $\xi_{1}$ (attracting hyperbolic barriers) or to $n_{\pm}$ (shear
barriers). As a result, the intersection of a transport barrier $\mathcal{\mathcal{\mathcal{M}}}(t_{0})$
with $\Pi(s_{1})$ must be a curve tangent to one of the following
three vector fields on $\Pi(s_{1})$: 
\[
u_{\xi_{3}}(x_{0};s_{1})=n_{\Pi(s_{1})}(x_{0})\times\xi_{3}(x_{0}),\qquad u_{\xi_{1}}(x_{0};s_{1})=n_{\Pi(s_{1})}(x_{0})\times\xi_{1}(x_{0}),
\]
\[
u_{n_{\pm}}(x_{0};s_{1})=n_{\Pi(s_{1})}\times n_{\pm}(x_{0}).
\]
We call $u_{\xi_{3}}(x_{0};s_{1})$ the \emph{reduced strain vector
field} and $u_{\xi_{1}}(x_{0};s_{1})$ the \emph{reduced stretch vector
field} on the reference surface $\Pi(s_{1})$. Similarly, we call
and $u_{n_{\pm}}(x_{0};s_{1})$ the \emph{reduced shear vector fields}
on $\Pi(s_{1})$. Since the manifold family $\Pi(s_{1})$ is assumed
orientable, the unit normal vector field $n_{\Pi(s_{1})}(x_{0})$
can be selected smoothly globally on $\Pi(s_{1})$. By contrast, the
vector fields $\xi_{3}(x_{0})$ and $n_{\pm}(x_{0})$ are typically
not globally orientable, and can only be selected smoothly over open
subset of $\Pi(s_{1})$.

The resulting local orientability of the vector fields $u_{\xi_{3}}(x_{0};s_{1})$
and $u_{n_{\pm}}(x_{0};s_{1})$ on $\Pi(s_{1})$ is enough for the
construction of all possible intersection curves $\gamma_{s_{1}}=\Pi(s_{1})\cap\mathcal{\mathcal{M}}(t_{0})$.
This can be achieved by solving piecewise oriented versions of one
of the differential equations (\ref{eq:redstrain-ODE})-(\ref{eq:redshear-ODE}).
Furthermore, any trajectory of these differential equations that represents
a true intersection curve $\gamma_{s_{1}}=\Pi(s_{1})\cap\mathcal{M}(t_{0})$
with a barrier $\mathcal{\mathcal{\mathcal{M}}}(t_{0})$ must necessarily
lie in the zero set of the appropriate helicity function featured
in Theorem 2.

\section{Numerical aspects of computing hyperbolic and shear LCS}

\subsection{Computing hyperbolic LCS}

The algorithmic steps H1-H6 provide a way to compute intersections
between hyperbolic LCSs and a family of reference planes. This involves
computing trajectories of the reduced strain vector fields in (\ref{eq:redstrain-ODE})
and (\ref{eq:redstrain-ODE-1}), both of which will generally have
singularities and orientational discontinuities. A systematic description
of the numerical procedures to deal with these issues can be found
in \citep{Mo_var_theory}.

\subsection{Computing shear LCS}

Again, the algorithmic steps SH1-SH5 provide a way to compute intersections
of shear LCSs and with a family of reference planes. The reduced shear
vector field (\ref{eq:redshear-ODE}) has singularities and orientational
discontinuities that can be handles as in the case of reduced strain
vector field. 

An additional complication for shear LCS is the smooth choice of $n_{\pm}$
along reduced shear trajectories. The normal fields $n_{\pm}$ have
the general form 
\[
n_{\pm}=\alpha\xi_{1}\pm\beta\xi_{3},
\]
representing four different directions in the three-dimensional phase
space. 

In the theory of transport barriers for $2$D flows in \citep{geo_theory},
an analogous shear vector field 
\[
\eta_{\pm}=\alpha\xi_{1}\pm\beta\xi_{2}
\]
arises. For this field, one can adopt the global orientation convention
$\xi_{2}=\Omega\xi_{1}$, where $\Omega$ denotes the operator of
counter-clockwise rotation by $90$ degrees. With this way of computing
$\xi_{2}$ in terms of $\xi_{1}$, the vector field $\eta_{\pm}$
only represents two vectors due to the sign ambiguity of $\xi_{1}$,
which we assume $\xi_{2}$ inherits. One can then solve for trajectories
of $\eta_{\pm}$ by solving 
\[
r'(s)=\text{sign}\left(\left<\eta_{\pm}(r(s)),r(s-\Delta)\right>\right)\eta_{\pm}(r(s)),
\]
 with $\Delta$ denoting a numerical time step. 

In the $3$D setting of the present paper, we cannot find a linear
transformation $\Omega$ that would globally fix the orientation of
$\xi_{3}$ relative to $\xi_{1}$. One therefore does not have simply
two globally defined vector fields to integrate. Rather, one starts
the integration and has to consider at each point along a reduced
shearline all four possible directions, resulting in four possible
choices of $n_{0}$ at that point. The correct vector to select is
the one that results in a smooth reduced shearline (as the transverse
intersection of a shear LCS with the reference plane). The method
used in this paper to select the correct vectors at each time step
is:

1) At the initial condition $r(0)$ in the reference plane $\Pi(s_{1})$,
compute the vectors $\xi_{1}(r(0))$ and $\xi_{3}(r(0))$.

2) Assuming one has the solution $r(s_{2})$ and the vectors $\xi_{1}(r(s_{2}))$
$\xi_{3}(r(s_{2}))$ stored, one computes the reduced shearline at
the point $r(s_{2}+\Delta)$ by matching the direction of $\xi_{1}(r(s_{2}+\Delta))$
and $\xi_{3}(r(s_{2}+\Delta))$ with the direction of $\xi_{1}(r(s_{2}))$
and $\xi_{3}(r(s_{2}))$ when forming $n_{\pm}(r(s_{2})+\Delta)$. 

\bibliographystyle{plainnat}
\bibliography{HD_Transport}

\end{document}